\documentclass[12pt,a4paper]{amsart}

\usepackage[utf8]{inputenc}
\usepackage{amsmath,amsfonts,amssymb}
\usepackage{amsthm,upref}
\usepackage[a4paper]{geometry}
\usepackage{color} 
\usepackage[dvipsnames]{xcolor}
\usepackage{graphicx}
\usepackage{subfigure}
\usepackage{tabularx,float}
\usepackage[hidelinks]{hyperref}
\usepackage{comment} 
\usepackage{tabularx} 

\newcommand{\note}[1]{ \textcolor{blue}{ \ [ \ #1 \  ] \ }} 
\newtheorem{thm}{Theorem}[section]
\newtheorem{definition}[thm]{Definition}

\newtheorem{proposition}[thm]{Proposition}
\newtheorem{corollary}[thm]{Corollary}
\newtheorem{lemma}[thm]{Lemma}
\newtheorem{remark}[thm]{Remark}
\newtheorem{assumption}{Assumption}

\newcommand{\tr}{\textrm{tr}\,}
\newcommand{\sgn}{\textrm{sgn\,}}

\usepackage{todonotes}
\newcommand\SJ[1]{{\color{Black}{#1}}} 
\newcommand\SJJ[1]{{\color{Black}{#1}}} 
\newcommand\KUK[1]{{\color{Black}{#1}}} 
\newcommand\KUKK[1]{{\color{Black}{#1}}} 
\newcommand\WM[1]{{\color{Black}{#1}}} 
\newcommand\KUKKK[1]{{\color{Black}{#1}}} 
\title{Singularly Perturbed Boundary-Equilibrium Bifurcations}
\author[S.~Jelbart]{S.~Jelbart${}^\ast$}\thanks{${}^\ast$Corresponding author. Department of Mathematics, Technical University of Munich, Garching, Bavaria 85748, Germany}
\author[K.~U.~Kristiansen]{K.~U.~Kristiansen${}^\dagger$}\thanks{${}^\dagger$Department of Applied Mathematics and Computer Science, Technical University of Denmark, Lyngby, Kgs.~2800, Denmark}
\author[M.~Wechselberger]{M.~Wechselberger${}^\ddagger$}\thanks{${}^\ddagger$School of Mathematics \& Statistics, University of Sydney, Camperdown, NSW 2006, Australia}

\begin{document}

\begin{abstract}
	\textit{Boundary equilibria bifurcation (BEB)} arises in piecewise-smooth systems when an equilibrium collides with a discontinuity set under parameter variation. \textit{Singularly perturbed BEB} refers to a bifurcation arising in singular perturbation problems which limit as some $\epsilon \to 0$ to \textit{piecewise-smooth (PWS)} systems which undergo a BEB. This work completes a classification for \WM{codimension-1 singularly perturbed BEB} in the plane initiated by the present authors in \cite{Jelbart2020d}, using a combination of tools from PWS theory, \textit{geometric singular perturbation theory (GSPT)} and a method of geometric desingularization known as \textit{blow-up}. After deriving a local normal form capable of generating all 12 singularly perturbed BEBs, we \WM{describe the unfolding in each case}. Detailed quantitative results on saddle-node, Andronov-Hopf, homoclinic and \SJJ{codimension-2} Bogdanov-Takens bifurcations involved in the unfoldings and classification are presented. Each bifurcation is singular in the sense that it occurs within a domain which shrinks to zero as $\epsilon \to 0$ at a rate determined by the rate at which the system loses smoothness. Detailed asymptotics for a distinguished homoclinic connection \SJJ{which forms the} boundary between two singularly perturbed BEBs in parameter space are also given. Finally, we describe the explosive onset of oscillations arising in the unfolding of a particular singularly perturbed \textit{boundary-node (BN)} bifurcation. We prove the existence of the oscillations as perturbations of PWS cycles, and derive a growth rate \SJJ{which is} polynomial in $\epsilon$ and dependent on the rate at which the system loses smoothness. For all the results presented herein, corresponding results for regularized PWS systems are obtained via the limit $\epsilon \to 0$.
	
	\bigskip
	\smallskip
	\noindent \SJJ{\textbf{keywords.} singular perturbations, piecewise-smooth systems, blow-up, boundary-equilibrium bifurcation, regularization
		
	}
\end{abstract}

\maketitle



\section{Introduction}

This manuscript concerns the unfolding of singularities in planar singular perturbation problems which limit to piecewise-smooth (PWS) systems. The underlying PWS system is assumed to have a smooth codimension-1 discontinuity set, or \textit{switching manifold} $\Sigma \subset \mathbb R^2$, which has an isolated \textit{boundary equilibrium (BE)}. BEs are PWS singularities which \WM{unfold generically in a codimension-1 bifurcation} known as a \textit{boundary equilibrium bifurcation (BEB)}, whereby an isolated equilibrium collides with the switching manifold $\Sigma$ under parameter variation.


A first classification of planar BE singularities appeared in Filippov's seminal work on discontinuous PWS systems \cite{Filippov1988}. \SJJ{Here it was shown} 
that generically, there are 8 topologically distinct classes of BE singularities, comprised of 2 boundary-saddle (BS), 2 boundary-focus (BF) and 4 boundary-node (BN) singularities. A treatment of the unfolding of these singularities came later in \cite{Kuznetsov2003}, where the authors identify 10 topologically distinct unfoldings, and provide `prototype systems' for each. Subsequently in \cite{Hogan2016}, two more unfoldings were identified, bringing the total count to 12. Here \WM{we} present a single prototype system capable of generating all 12 unfoldings, and a completeness theorem \cite[Theorem 2]{Hogan2016} ruling out the possibility of additional missing cases. Explicit local normal forms (as opposed to `prototypes') for \SJJ{a large number of} BE singularities have been derived in \cite{Carvalho2018}, but \WM{their} work did not treat the unfolding via BEB. 

The notion of \textit{singularly perturbed BEB} was developed more recently in \cite{Jelbart2020d}, for the analysis of smooth singular perturbation problems limiting to PWS systems with a BF bifurcation. The motivation to study smooth perturbations of PWS systems arises from the observation that PWS systems often serve as approximations for smooth dynamical systems with abrupt transitions in phase space. Hence, it is natural to consider a class of smooth singular perturbation problems, which limit to PWS systems that are discontinuous along a switching manifold $\Sigma$ as a perturbation parameter $\epsilon \to 0$. Abrupt dynamical transitions in such systems occur within an $\epsilon-$dependent neighbourhood $U_\epsilon \subset \mathbb R^2$ about $\Sigma$ known as the \textit{switching layer}, which satisfies $U_\epsilon \to \Sigma$ as $\epsilon \to 0$. It is important to note that singular perturbation problems in this class can arise either (i) naturally, or (ii) by a process of \textit{regularization} whereby a modeller `smooths out' discontinuity in a PWS system. In the former case, the problem is given as a smooth singular perturbation problem with a PWS singular limit; see e.g.~\cite{Jelbart2019c,Kristiansen2019d} for applications of this kind. In the latter case, the PWS system is given, and the modeller introduces a method of regularization based on the characteristics of the problem at hand; many examples of this kind can be found in \cite{Jeffrey2018}. In both cases, analytical techniques from PWS systems and \textit{Geometric Singular Perturbation Theory (GSPT)} \cite{Jones1995,Kuehn2015,Wechselberger2019}, in combination with a method of geometric desingularization known as the \textit{blow-up method} \cite{Dumortier1996,Krupa2001b}, \SJ{provide a powerful analytical framework;} 
see e.g.~\cite{Bonet2016,Buzzi2006,Guglielmi2015,Jeffrey2018,Kristiansen2015b,Kristiansen2015a,Kristiansen2019,Kristiansen2017,Kristiansen2019d,Llibre2007,Sotomayor1996,Teixeira2012}. It is worthy to note that the authors in \cite{Carvalho2011} consider a large number of BEBs in the context of regularized PWS systems, however the degeneracy associated with the BE singularity is not fully resolved.

\

The present manuscript provides a classification and detailed dynamical study of singularly perturbed BEBs in the plane. 
The work can be seen as a continuation of recent work in \cite{Jelbart2020d}, see also the PhD thesis \cite{Jelbart2020b}, 
\WM{where the analysis was} restricted to a subset of singularly perturbed BF bifurcations, treating 3 of the total 12 BE unfoldings in detail, and successfully resolving the degeneracies associated with these cases. This manuscript aims to complete the project, by providing a `complete' description for all 12 unfoldings. 
Similarly to \cite{Jelbart2020b,Jelbart2020d}, emphasis is placed on understanding the smooth dynamics with $0 < \epsilon \ll 1$. This allows for the treatment of problems arising \SJ{either} naturally or via regularization simultaneously, since the corresponding results for (regularized) PWS systems are easily obtained upon taking the non-smooth singular limit $\epsilon \to 0$.

First, we show that the $C^{r\geq 1}$ local normal form derived for singularly perturbed BF bifurcations in \cite{Jelbart2020d} is in fact sufficient to generate all 12 unfoldings. A corresponding PWS local normal form is obtained from this expression in the non-smooth singular limit $\epsilon \to 0$.

We then study all 12 unfoldings for $0 < \epsilon \ll 1$. As found in the analysis of singularly perturbed BF bifurcations in \cite{Jelbart2020d}, each unfolding typically involves singular bifurcations, in some cases codimension-2, occurring within an $\epsilon-$dependent domain which shrinks to zero as $\epsilon \to 0$ at a rate which can be quantified explicitly in terms of 
\SJ{rate at which the system loses smoothness}. We present 2-parameter bifurcation diagrams for a desingularized system with $\epsilon = 0$ 
which determines the qualitative dynamics for $0 < \epsilon \ll 1$. It is worthy to note that within the class of smooth monotonic regularizations considered, the dynamics are shown to be qualitatively determined by the underlying PWS problem, i.e.~the bifurcation structure is qualitatively independent of the choice of regularization, and determined by the type of PWS unfolding in the limit $\epsilon \to 0$. It is shown how the choice of regularization does, however, effect the dynamics quantitatively, particularly due to its determination of the rate at which the system loses smoothness as $\epsilon \to 0$.

Following an analysis of the unfoldings, we present new results on the asymptotics of \SJ{distinguished homoclinic solutions corresponding to boundaries between singularly perturbed BF$_1$ and BF$_2$ bifurcations.} 
\KUKKK{Finally, special attention is devoted to the singularly perturbed BN$_3$ bifurcation, which provides the necessary local mechanisms for the onset of relaxation-type oscillations. This was first observed in \cite{Kristiansen2019d} in the context of substrate-depletion oscillations. 
{Whereas emphasis there was on the existence of relaxation-type oscillations for the specific model, we will in the present work identify and describe the \textit{explosive} onset of oscillations in the case of generic singularly perturbed BN$_3$ bifurcations.} 
Similarly to the \textit{canard explosion} phenomena known to occur in classical slow-fast systems \cite{Dumortier1996,Kuehn2015,Krupa2001b}, we will show that limit cycles in the singularly perturbed BN$_3$ bifurcation perturb from a continuous family of singular cycles, i.e.~closed concatenated orbits having segments along a critical manifold. However, the local mechanism for the onset of explosive dynamics differs from that of classical canard explosion, and functions without the need for canard solutions. 
In contrast to the exponential growth rate associated with classical canard explosion, we show that the growth rate of the cycles arising in singularly perturbed BN$_3$ is polynomial in $\epsilon$. We quantify this growth rate in terms of properties of the regularization.}

\

The manuscript is structured as follows: Basic definitions and setup are introduced in Section \ref{sec:setup}. The $C^{r\geq 1}$ local normal form capable of generating all 12 unfoldings for $0<\epsilon\ll1$, as well as the resulting smooth and PWS classifications are also given in Section \ref{sec:setup}. Main results are presented in Section \ref{sec:results}. Specifically, the blow-up analysis is outlined in Section \ref{ssec:results_blow-up}, unfoldings and corresponding 2-parameter bifurcation diagrams are presented in Section \ref{ssec:results_bifurcatons}, asymptotic results on boundary separatrices are presented in Section \ref{ssec:results_separatrices}, and results on the singularly perturbed BN$_3$ explosion are given in Section \ref{ssec:results_bn3}. Main results on the unfolding and boundary separatrices are proved in Section \ref{sec:proofs}, and a proof for the results pertaining to BN explosion are presented in Section \ref{sec:thm_connection_proof}. Finally in Section \ref{sec:Outlook}, we conclude and summarise our findings.


\section*{Acknowledgement}

\SJ{The first author acknowledges partial funding from \SJJ{the SFB/TRR 109 Discretization and Geometry in Dynamics grant}, and together with the third author, partial funding from the \WM{ARC Discovery Project grant DP180103022}.}

\KUK{The second author is grateful for his discussions with Peter Szmolyan on the explosive growth of limit cycles for the boundary node case in the context of the substrate-depletion oscillator.}

\section{Setup and normal form}
\label{sec:setup}

\subsection{Setup}

The setup is taken from \cite{Kristiansen2019c}, which has also been adopted in \cite{Jelbart2020b,Jelbart2020d,Kristiansen2019d}. We consider planar systems
\begin{equation}
\label{eq:general}
\dot z = Z\left(z,\phi\left(y \epsilon^{-1}\right),\alpha\right) ,
\end{equation}
where $z=(x,y) \in \mathbb R^2$, $\phi:\mathbb R\to\mathbb R$, $\epsilon \in (0,\epsilon_0]$, and $\alpha \in I \subset \mathbb R$. The vector field $Z:\mathbb R^2 \times \mathbb R \times I \to \mathbb R^2$ is assumed to be smooth in all arguments, but generically non-smooth in the limit $\epsilon \to 0$. 

\begin{assumption}
	\label{assumption:1}
	The map $p \mapsto Z(z, p, \alpha)$ is affine, i.e.
	\begin{equation}
	\label{eq:regularized}
	Z(z, p, \alpha) = p Z^+(z, \alpha) + (1 - p) Z^-(z, \alpha) ,
	\end{equation}
	where the vector fields $Z^\pm : \mathbb R^2 \times I \to \mathbb R^2$ are smooth.
\end{assumption}
\begin{assumption}
	\label{assumption:2}
	The smooth `regularization function' $\phi : \mathbb R \to \mathbb R$ satisfies the monotonicity condition
	\[
	\frac{\partial \phi(s)}{\partial s} > 0 ,
	\]
	for all $s \in \mathbb R$ and, moreover,
	\begin{equation}
	\label{eq:mono}
	\phi(s) \to
	\begin{cases}
	1  & \text{for } s \to \infty , \\
	0  & \text{for } s \to -\infty .
	\end{cases}
	\end{equation}
\end{assumption}

It follows from Assumption \ref{assumption:1} and the form of the (non-uniform) limit in \eqref{eq:mono} that system \eqref{eq:general} is (generically) PWS in the singular limit $\epsilon \to 0$. In particular, the limiting system
\begin{equation}
\label{eq:PWS_main}
\dot z = 
\begin{cases}
Z^+(z,\alpha) \qquad \text{if } y > 0, \\
Z^-(z,\alpha) \qquad \text{if } y < 0,
\end{cases}
\end{equation}
is PWS and (generically) discontinuous along the switching manifold
\begin{equation}
\label{eq:Sigma}
\Sigma = \left\{(x,y) : f_\Sigma(x,y) = y = 0 \right\} .
\end{equation}

\begin{remark}
	The more general scenario where $\Sigma = \{z\in\mathbb R^2:f_\Sigma(z)=0\}$ for any smooth function $f_\Sigma:\mathbb R^2 \to \mathbb R$ such that $Df_\Sigma|_{\Sigma} \neq (0,0)$, can easily be incorporated into the preceding formalism by replacing $y$ with $f_\Sigma(z)$ in system \eqref{eq:general} and adjusting Assumptions \ref{assumption:1} and \ref{assumption:2} accordingly. Since we restrict to a local analysis throughout, we may assume that $f_\Sigma(z)=y$ without loss of generality.
\end{remark}

Notice that system \eqref{eq:PWS_main} can be `regularized' via \eqref{eq:regularized} with $p=\phi(y \epsilon^{-1})$. 
Hence, system \eqref{eq:general} can be viewed as either of the following:
\begin{itemize}
	\item A smooth singularly perturbed system with a PWS singular limit;
	\item A smooth regularization of the PWS system \eqref{eq:PWS_main}.
\end{itemize}
In this work we shall prioritise the former interpretation, since (i) this case is treated in less detail so far in the literature, 
and (ii) findings pertinent to the latter case can be immediately inferred from the dynamics of the nearby smooth system upon taking the limit $\epsilon \to 0$.

We impose one more technical assumption, which restricts the class of regularization functions $\phi$:
\begin{assumption}
	\label{assumption:3}
	The regularization function $\phi(s)$ has algebraic decay as $s \to \pm\infty$, i.e. there exist $k_\pm \in \mathbb N_+$ and smooth functions $\phi_\pm:[0,\infty] \to [0,\infty)$ such that
	\begin{equation}
	\label{eq:reg_asymptotics}
	\phi(s)=
	\begin{cases}
	1 - s^{-k_+} \phi_+\left(s^{-1} \right) , &\qquad s >0\,,\\
	(- s)^{-k_-} \phi_-\left((-s)^{-1} \right) , &\qquad  s<0\,,
	\end{cases}
	\end{equation}
	and
	\begin{equation}
	\label{eq:beta}
	\beta_\pm := \phi_\pm(0) > 0 . 
	\end{equation}
\end{assumption}

Assumption \ref{assumption:3} restricts to the class of regularization functions with algebraic decay toward $0,1$, and is natural in the context of general systems \eqref{eq:general} with analytic or sufficiently smooth right-hand-side. Specifically, it follows that both mappings $u\mapsto \phi(\pm u^{-1})$ for $u>0$ have well-defined Taylor expansions at $u=0$, which are each nondegenerate in the sense that there are leading nonzero terms ($1-u^{k_+} \beta_+$ and $u^{k_-} \beta_-$, respectively) at order $k_\pm$, respectively. Note this assumption precludes regularizations like $\phi(s)=\tanh(s)$ or $\phi(s)=e^s/(1+e^s)$, which have exponential decay toward $0,1$ and thus $k=\infty$. We omit the rigorous treatment of these cases, 
but refer to \cite{Jelbart2019c,Kristiansen2017} for details on how to handle \SJ{non-algebraic asymptotics using an adaptation of the blow-up method}.

\begin{remark}
	Regularisation functions $\phi$ which satisfy Assumptions \ref{assumption:2} and \ref{assumption:3} can be analytic, and should be distinguished from the well known class of non-analytic Sotomayor-Teixeira (ST) regularizations. In particular, the regularizations considered herein do not feature an artificial cutoff at the boundary to the switching layer.
\end{remark}

\subsection{PWS preliminaries}

It follows from our assumptions that the PWS system \eqref{eq:PWS_main} is Filippov-type \cite{Filippov1988}. In particular, sliding and crossing regions of $\Sigma$ can be determined in accordance with their usual definitions.
\begin{definition}
	Given system \eqref{eq:PWS_main} and a point $p \in \Sigma$. Then $p \in \Sigma$ is called a crossing (sliding) point if the quantity
	\begin{equation}
	\label{eq:crossing}
	\left(Z^+f(p)\right) \left(Z^-f(p)\right) 
	\end{equation}
	is positive (negative), where 
	$Z^\pm f(\cdot) = \langle \nabla f(\cdot), Z^+(\cdot,\alpha) \rangle$ denotes a Lie derivative. We denote the set of crossing (sliding) points by $\Sigma_{cr}$ ($\Sigma_{sl}$). 
\end{definition}

It follows from our assumptions on $\phi$ 
that \WM{the sliding/Filippov vector field described} as a convex combination in \cite{Filippov1988} can be written as
\begin{equation}
\label{eq:sliding_VF}
\dot z = - \left((Z^+ - Z^-)(f_\Sigma)(z) \right)^{-1} \left[Z^+, Z^-\right] (f_\Sigma) (z)  , \qquad z \in \Sigma_{sl} ,
\end{equation}
where 
$[Z^+,Z^-]$ denotes a Lie bracket. If $f_\Sigma(x,y)=y$ as in \eqref{eq:Sigma}, then the sliding/Filippov vector field is given in the $x-$coordinate chart by
\begin{equation}
\label{eq:Filippov_x}
\dot x = \frac{\det\left(Z^+(x,0)|Z^-(x,0)\right)}{Z_2^-(x,0)-Z_2^+(x,0)} = Z_{sl}(x,0) , \qquad (x,0) \in \Sigma_{sl},
\end{equation}
where $\det(Z^+(x,0)|Z^-(x,0))$ denotes the determinant of the 2$\times$2 matrix with columns $Z^+(x,0)$, $Z^-(x,0)$.

Sliding trajectories can leave $\Sigma_{sl}$ at a point of tangency with either vector field $Z^{\pm}$. Depending on the order of the tangency, such a point may also separate sliding and crossing regions of $\Sigma_{sl}$. The following definition characterises the least degenerate case, i.e.~quadratic tangency with either $Z^{\pm}$.

\begin{definition}
	\label{def:fold}
	Given system \eqref{eq:PWS_main} and a point $F \in \Sigma$. Then $F$ is a fold point if either
	\begin{equation}
	\label{eq:fold}
	Z^+f(F) = 0 , \quad Z^+(Z^+f)(F) \neq 0 , \quad \text{or} \quad Z^-f(F) = 0 , \quad Z^-(Z^-f)(F) \neq 0 .
	\end{equation}
	A fold point $F$ with $Z^+f(F) = 0$ is visible (invisible) if the inequality $Z^+(Z^+f)(F) \neq 0$ is positive (negative). Conversely, a fold point $F$ with $Z^-f(F) = 0$ is visible (invisible) if the inequality $Z^-(Z^-f)(F) \neq 0$ is negative (positive).
\end{definition}

\begin{figure}[t!]
	\centering
	\includegraphics[scale=0.205]{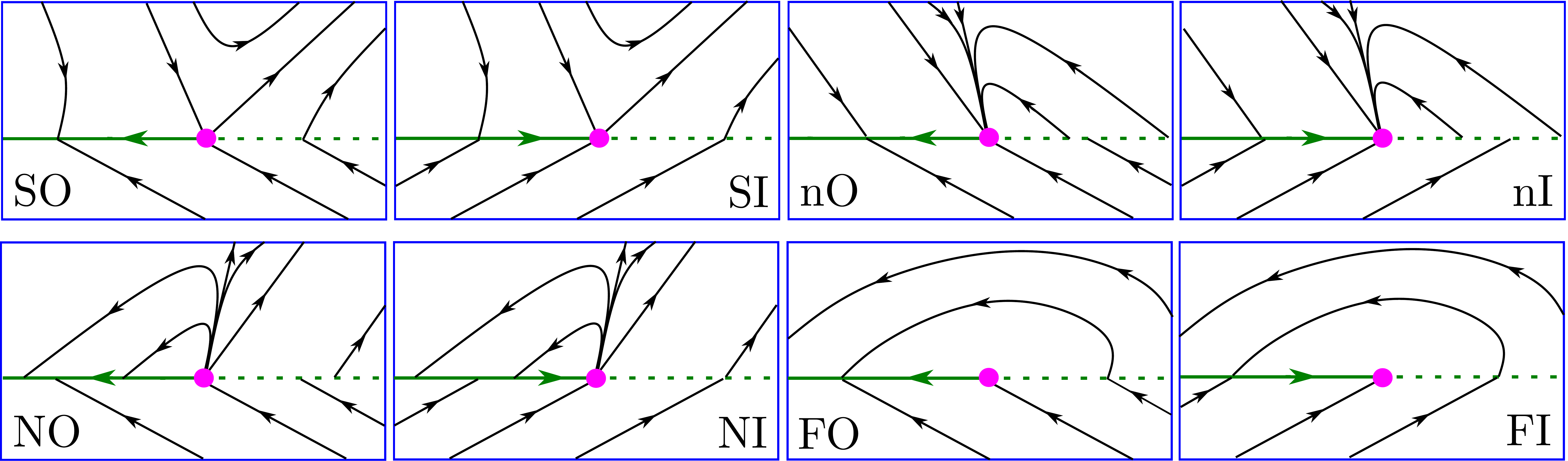}
	\caption{The 8 BE singularities arising in Filippov's topological classification \cite{Filippov1988}. The switching manifold $\Sigma=\{y=0\}$ is shown in green, with sliding/crossing submanifolds in bold/dashed respectively. We adopt the labelling convention in \cite{Hogan2016} with S, n, N, F denoting saddle, stable node, unstable node, focus respectively, and I/O denoting inward/outward flow along $\Sigma_{sl}$. \SJJ{We have chosen an orientation such} \SJJ{that the $\Sigma_{sl}$ always lies to the left.}}
	\label{fig:be}
\end{figure}

It remains to review the notion of BE singularities and BEB. BE singularities arise when one or both of the vector fields $Z^\pm(z_{be},\alpha_{be})=(0,0)^T$ for some \SJ{$z_{be} \in \Sigma$ and parameter value $\alpha = \alpha_{be}$}. We consider the least degenerate case in which \SJ{$z_{be} \in \Sigma$} is a hyperbolic equilibrium of \SJ{$Z^+(\cdot, \alpha_{be})$}, and $Z^-$ is locally transverse to $\Sigma$. Filippov showed in \cite{Filippov1988}, see also \cite{Hogan2016}, that there are 8 topologically distinct cases depending on:
\begin{itemize}
	\item The type of equilibrium (focus, node or saddle);
	\item The orientation of the sliding dynamics (towards or away from \SJ{$z_{be}$});
	\item In the case that \SJ{$z_{be}$ is a node of $Z^+(\cdot,\alpha_{be})$}, its asymptotic stability (stable or unstable);
\end{itemize}
see Figure \ref{fig:be}. As described in \cite{Hogan2016}, the 8 cases can be neatly categorised if we let S, n, N, F denote `saddle', `stable node', `unstable node', `focus' respectively, and let I/O define inward/outward sliding flow \WM{(i.e.~towards or away from $z_{be}$)}. Then the possible cases are: SO, SI, nO, nI, NO, NI, FO and FI.

BE singularities unfold generically under parameter variation in a BEB. Below we provide a formal definition for BEB in general PWS systems \eqref{eq:PWS_main}. 

\begin{definition}
	\label{def:beb}
	The PWS system \eqref{eq:PWS_main} has a BEB at $z = z_{bf} = (x_{bf},0) \in \Sigma$ for $\alpha = \alpha_{bf}$ if $Z^+(z_{bf},\alpha_{bf}) = (0,0)^T$ and the following nondegeneracy conditions hold:
	\begin{equation}
	\label{eq:beb_conds}
	Z_2^-(z_{bf},\alpha_{bf}) \neq 0 , \ \ \det \left(\frac{\partial Z^+}{\partial \alpha} \big| \frac{\partial Z^+}{\partial x} \right)\bigg|_{(z_{bf},\alpha_{bf})} \neq 0 , \ \ \frac{\partial Z_{sl}}{\partial x}\bigg|_{(x_{bf},\alpha_{bf})} \neq 0 ,
	\end{equation}
	where $Z^\pm = (Z^\pm_1,Z^\pm_2)^T$ and $\det (X | Y)$ denotes the determinant of the matrix with columns $X,Y$.
	
	\SJ{Let $\lambda_\pm$ 
	and $v_\pm$} 
	denote the eigenvalues and corresponding eigenvectors of the Jacobian $(\partial Z^+/\partial z)|_{(z_{bf},\alpha_{bf})}$. We distinguish the following cases:
	\begin{enumerate}
		\item[(BF)] 
		\SJ{$\lambda_\pm 
		 = A \pm i B$} for $A , B \in \mathbb R\setminus\{0\}$, {\em (boundary-focus);}
		\item[(BN)] 
		\SJ{$\lambda_+ 
		 / \lambda_- 
		   > 0$ and $v_\pm 
		   $} are transversal to $\Sigma$, {\em (boundary-node);}
		\item[(BS)] 
		\SJ{$\lambda_+ 
		 / \lambda_- 
		   < 0$ and $v_\pm 
		    $} are transversal to $\Sigma$, {\em (boundary-saddle).}
	\end{enumerate}
\end{definition}

The topological classification in \cite{Hogan2016} shows that generically, the 8 BE bifurcations in Figure \ref{fig:be} unfold in 12 topologically distinct BEBs. Specifically, there are 
5 BF bifurcations, 4 BN bifurcations and 3 BS bifurcations. We shall label these by BF$_i$, BN$_i$ and BS$_i$ for $i \in \{1,\ldots, 5\}$, $i \in \{1,\ldots,4\}$ or $i \in \{1,2,3\}$ respectively, in accordance with the notational conventions from \cite{Kuznetsov2003}. The two unidentified BN bifurcations later described in \cite{Hogan2016} will be denoted BN$_3$ and BN$_4$. The BN$_3$ unfolding \WM{is} of particular interest in this work \WM{and} shown in Figure \ref{fig:bn3}. The fact that there may be more than one unfolding per BE is a consequence of the relative positioning of separatrices; topologically distinct BEBs can be separated by so-called `double separatrices' \cite{Hogan2016} which connect equilibria and points of tangency on $\Sigma$. The role of separatrices is also discussed in e.g.~ \cite{Bernardo2008,Guardia2011}. 

\begin{figure}[t!]
	\centering
	\includegraphics[scale=.27]{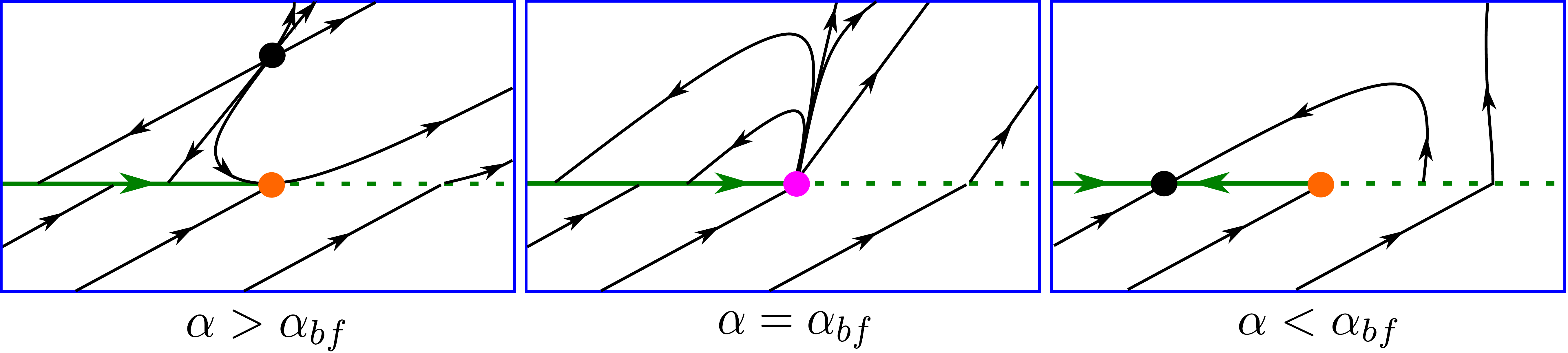}
	\caption{Unfolding of the NI BE (c.f.~Figure \ref{fig:be}) in a BN$_3$ bifurcation.}
	\label{fig:bn3}
\end{figure}

\begin{remark}
	The determinant condition in \eqref{eq:beb_conds} ensures that the equilibrium of \SJ{$Z^+(\cdot,\alpha)$} collides with $\Sigma$ transversally under variation in $\alpha$. To see this, notice that in the extended $(z,\alpha)-$space, the vector $T_{bf} := (\nabla Z_1^+ \wedge \nabla Z_2^+)|_{(z_{bf},\alpha_{bf})}$ is tangent to the curve \KUKK{defined} implicitly by $Z^+(z,\alpha) = (0,0)^T$. The stated determinant condition follows by the requirement that $T_{bf}$ has a non-zero $y-$component.
\end{remark}

Finally, we introduce the notion of singularly perturbed BEB.

\begin{definition}
	We say that system \eqref{eq:general} under Assumptions \ref{assumption:1} \SJ{and} \ref{assumption:2} has a singularly perturbed BEB if the PWS system \eqref{eq:PWS_main} obtained in the singular limit $\epsilon \to 0$ has a BEB. Notions of singularly perturbed BF bifurcation, singularly perturbed BN bifurcation and singularly perturbed BS bifurcation are similarly defined.
\end{definition}

By definition, the existence of 12 BEBs implies the existence of 12 singularly perturbed BEBs.



\subsection{Normal form and classification}

We show that the normal form derived for singularly perturbed BF bifurcations in \cite{Jelbart2020d} generalises to a single normal form capable of generating all 12 singularly perturbed BEBs. 


\begin{thm}
	\label{theorem:prop_normal_form}
	Consider system \eqref{eq:general} under Assumptions \ref{assumption:1}, \ref{assumption:2}, 
	and assume that the PWS system \eqref{eq:PWS_main} obtained in the limit $\epsilon \to 0$ has a BEB of type BF, BN or BS at $z_{bf} \in \Sigma$ when $\alpha = \alpha_{bf}$. Then there exists constants
	\[
	\tau \in \mathbb R \setminus \{0\}, \qquad  \delta \in \mathbb R \setminus \{0,\tau^2 / 4\} , \qquad \gamma \in \mathbb R ,
	\]
	such that system \eqref{eq:general} can be smoothly transformed, up to a reversal of orientation, into the local normal form
	\begin{equation}
	\label{eq:normal_form}
	\begin{pmatrix}
	\dot{x}\\\dot{y}
	\end{pmatrix}
	=
	\begin{pmatrix}
	\tau - \gamma + \phi\left(y \epsilon^{-1} \right) \left(\gamma - \tau + \mu + \tau x - \delta y + \theta_1(x,y,\mu) \right) \\
	1 + \phi\left(y \epsilon^{-1} \right) \left(-1 + x + \theta_2(x,y,\mu) \right)
	\end{pmatrix}
	= :X(x,y,\mu,\epsilon) ,
	\end{equation}
	where $\theta_i(x,y,\mu)$, $i=1,2$ are real-valued smooth functions such that
	\[
	\theta_1(x,y,\mu) = 
	\mathcal O(x^2,xy,y^2,x\mu,y\mu,\mu^2) , \qquad 
	\theta_2(x,y,\mu) = \mathcal O \left(x^2, xy, y^2, x \mu, y \mu \right) ,
	\]
	and $\mu$ is a new bifurcation parameter related to $\alpha$ via $\mu = g(\alpha)$, where $g: I_\alpha \to \mathbb R$ is a smooth function such that $g(\alpha_{bf}) = 0$ and $g'(\alpha_{bf}) \neq 0$.
	
	The PWS system
		\begin{equation}
		\label{eq:PWS_normal_form}
		\begin{pmatrix}
		\dot{x}\\\dot{y}
		\end{pmatrix}
		=
		\left\{
		\begin{aligned}
		&
		\begin{pmatrix}
		\mu + \tau x - \delta y + \theta_1(x,y,\mu)  \\
		x + \theta_2(x,y,\mu) 
		\end{pmatrix}
		=: X^+(x,y,\mu),
		& \quad(y>0) ,
		\\
		&
		\begin{pmatrix}
		\tau - \gamma  \\
		1 
		\end{pmatrix}
		=: X^-(x,y,\mu),
		&\quad(y<0) ,
		\end{aligned}
		\right.
		\end{equation}
	obtained from \eqref{eq:normal_form} in the limit $\epsilon \to 0^+$ has a BEB at the origin for $\mu = 0$, and a Filippov/sliding vector field given by
	\begin{equation}
	\label{eq:Filippov_VF}
	\dot x = \frac{\mu + \gamma x + \theta_1(x,0,\mu) - (\tau - \gamma) \theta_2(x,0,\mu)}{1 - x - \theta_2(x,0,\mu)} =: X_{sl}(x,\mu) , \qquad (x,0) \in \Sigma_{sl} .
	\end{equation}
\end{thm}

\begin{proof}
	The proof is similar to \SJJ{derivation of the normal form for singularly perturbed BF bifurcations presented in} \cite[p.38]{Jelbart2020d}, and deferred to Appendix \ref{app:normal_form} for brevity.
\end{proof}

\begin{remark}
	Note the qualifier ``up to a reversal of orientation" in Theorem \ref{theorem:prop_normal_form}. Orientation should be reversed if the vector field component $Z_2^-(z,\alpha)$ in system \eqref{eq:general} satisfies $Z^-_2(0,0) < 0$.
\end{remark}





A classification of singularly perturbed BEBs with $0 < \epsilon \ll 1$ can be given via the classification of the underlying PWS system for $\epsilon \to 0$. 
This approach is similar to the classification of singularities in slow-fast systems in terms of their `singular imprint' for $\epsilon = 0$.

Similarly to the prototype system given in \cite{Hogan2016}, the PWS normal form \eqref{eq:PWS_normal_form} can be used to generate all 12 BEBs by a suitable restriction of parameters in the PWS normal form \eqref{eq:PWS_normal_form}. Each unfolding can be identified with an open region in $(\tau,\delta,\gamma)-$parameter space determined by the quantities
\[
\tau, \ \delta, \ \Delta := \tau^2 - 4\delta , \ \gamma .
\]
Double-separatrices which connect a visible fold point with an equilibrium on $\Sigma_{sl}$ also play a role in separating regions corresponding to BF$_{1,2}$, and regions corresponding to BS$_{1,2}$. Here the distinction lies in whether or not the separatrix emanating from the fold point connects to the region $\tilde \Sigma_{sl} \subset \Sigma_{sl}$ which is bounded between the fold point and the equilibrium. The resulting classification, which is equivalent to that in \cite[Table 1]{Hogan2016}, is given in Table \ref{tab:classification}.

\begin{table}
	\centering
\begin{tabularx}{\textwidth}{ccccccc}
	\textbf{Bifurcation} \ & \ \textbf{Singularity} \ &  \ \ \ $\tau$  \ \ \ & \ \ \ $\delta$ \ \ \ & \ \ \ $\Delta$ \ \ \ & \ \ \ $\gamma$ \ \ \ & \ \textbf{Separatrix} \ \\ 
\hline
BS$_1$ & SI &   & - & + & - & Does not hit $\tilde \Sigma_{sl}$ \\
BS$_2$ & SI &   & - & + & - & Hits $\tilde \Sigma_{sl}$ \\
BS$_3$ & SO &   & - & + & + & \\
BN$_1$ & nI & - & + & + & - & \\
BN$_2$ & NO & + & + & + & + & \\
BN$_3$ & NI & +& + & + & - & \\
BN$_4$ & nO & - & + & + & + & \\
BF$_1$ & FO & + & + & - & + & Hits $\tilde \Sigma_{sl}$ \\
BF$_2$ & FO & + & + & - & + & Does not hit $\tilde \Sigma_{sl}$ \\
BF$_3$ & FI & + & + & - & - & \\
BF$_4$ & FI & - & + & - & - & \\
BF$_5$ & FO & - & + & - & + & \\
\hline
\vspace{0.05cm}
\end{tabularx}
\caption{Classification for the singularly perturbed BEBs generated by the local normal form \eqref{eq:normal_form}, given in terms of a PWS classification for the PWS local normal form \eqref{eq:PWS_normal_form} obtained in the singular limit $\epsilon \to 0$. The classification is equivalent to \SJJ{the} PWS classification in \cite[Table 1]{Hogan2016}. Here $\pm$ denotes the sign of the corresponding quantity.}
\label{tab:classification}
\end{table}

\section{Main results}
\label{sec:results}

In this section we present our main results. We begin in Section \ref{ssec:results_blow-up} with an outline of the sequence of blow-up transformations necessary to resolve all degeneracy associated with singularly perturbed BEB in system \eqref{eq:normal_form}. This allows for the identification of a desingularized system governing the unfolding of the singularity. In Section \ref{ssec:results_bifurcatons}, we present the unfolding for all 12 singularly perturbed BEBs. In Section \ref{ssec:results_separatrices} we present results on the asymptotics of \SJ{a homoclinic double-separatrix which separates singularly perturbed BF$_{1,2}$ bifurcations. The BS$_{1,2}$ boundary is also discussed.} Finally in Section \ref{ssec:results_bn3}, we present results on an observed `explosion' in the case of singularly perturbed BN$_3$ bifurcations.

\subsection{Resolution via blow-up}
\label{ssec:results_blow-up}

\SJ{We} describe the blow-up analysis used to resolve degeneracy in system \eqref{eq:normal_form} due to either (i) the loss of smoothness along $\Sigma$, or (ii) the loss of hyperbolicity at fixed points. The sequence of blow-up transformations is the same as in \cite{Jelbart2020d}, \KUKK{so we restrict ourselves here to an overview.}

\

System \eqref{eq:normal_form} loses smoothness along $\Sigma$ in the singular limit $\epsilon \to 0$. To describe this, we follow \cite{Kristiansen2019c,Kristiansen2019} and others and consider extended system
\[
\left\{ (x',y') = \epsilon X(x,y,\mu,\epsilon), \epsilon' = 0 \right\} ,
\]
with respect to a fast time, 
recall \eqref{eq:normal_form}. For this system $\Sigma\times \{0\}$ is a set of equilibria with a loss of smoothness. We gain smoothness via a homogeneous cylindrical blow-up transformation of the form
\begin{equation}
\label{eq:bu_cyl}
r \geq 0, \ \left(\bar y, \bar \epsilon \right) \in S^1 \mapsto
\begin{cases}
y = r \bar y , \\
\epsilon = r \bar \epsilon ,
\end{cases}
\end{equation}
which replaces $\Sigma \times \{0\}$ by the cylinder $\{r=0\} \times \mathbb R \times S^1$, see Figure \ref{fig:bu1}. The subspace $\{r=0\}$ corresponding to the blow-up cylinder is invariant. \KUKK{After a suitable desingularization amounting to division by $\bar \epsilon$} 
, the dynamics within $\{r=0\}$ are governed by a slow-fast system with a normally hyperbolic and attracting critical manifold, denoted $S$ in Figure \ref{fig:bu1}(b) \cite{Bonet2016,Buzzi2006,Kristiansen2019c,Kristiansen2019,Llibre2009,Llibre2007}. Moreover, there is a reduced flow on $S$ which is topologically conjugate to the sliding/Filippov flow induced by \eqref{eq:Filippov_VF}.

\begin{figure}[t!]
	\centering
	\subfigure[]{\includegraphics[width=.44\textwidth]{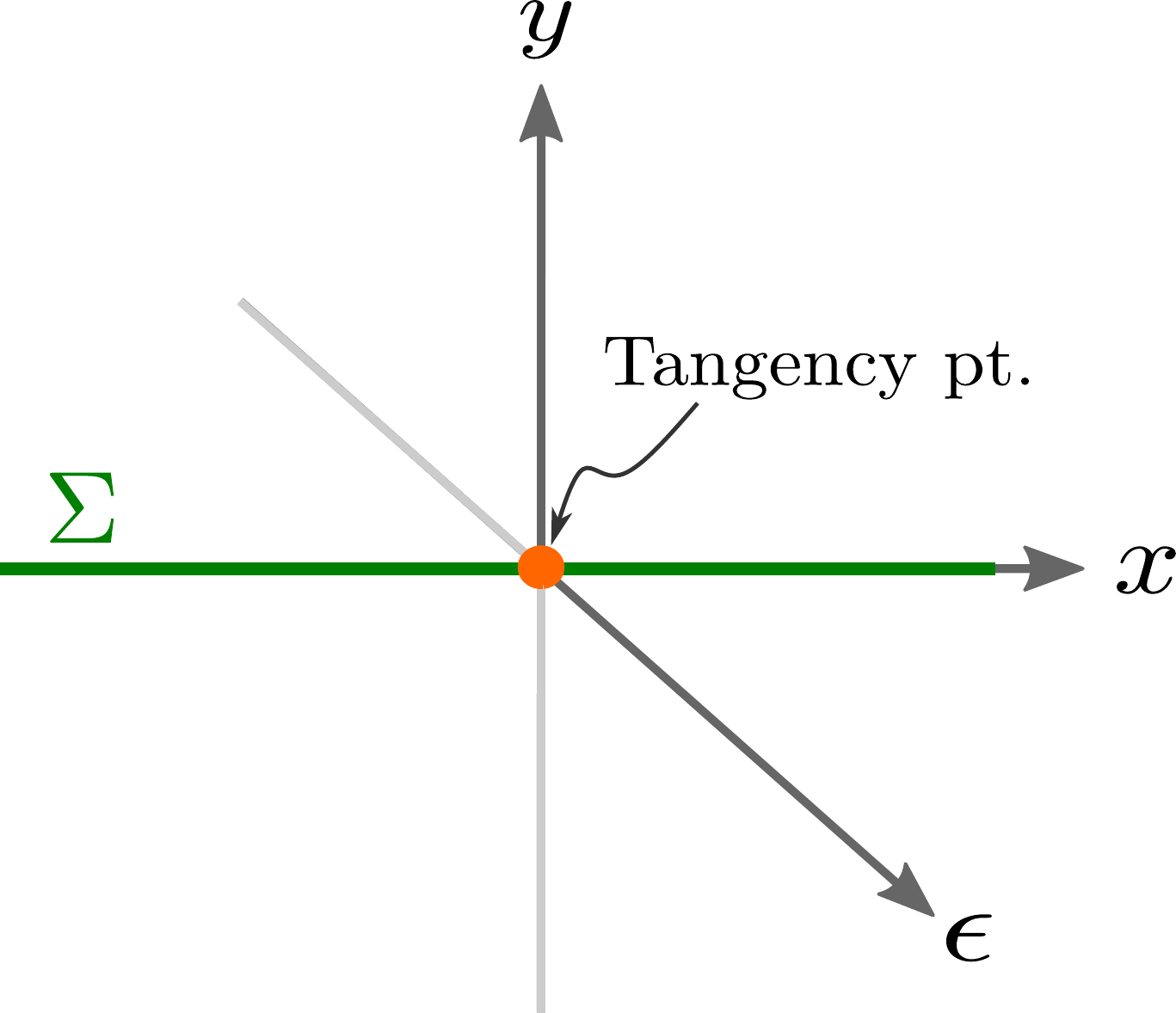}}
	\subfigure[]{\includegraphics[width=.54\textwidth]{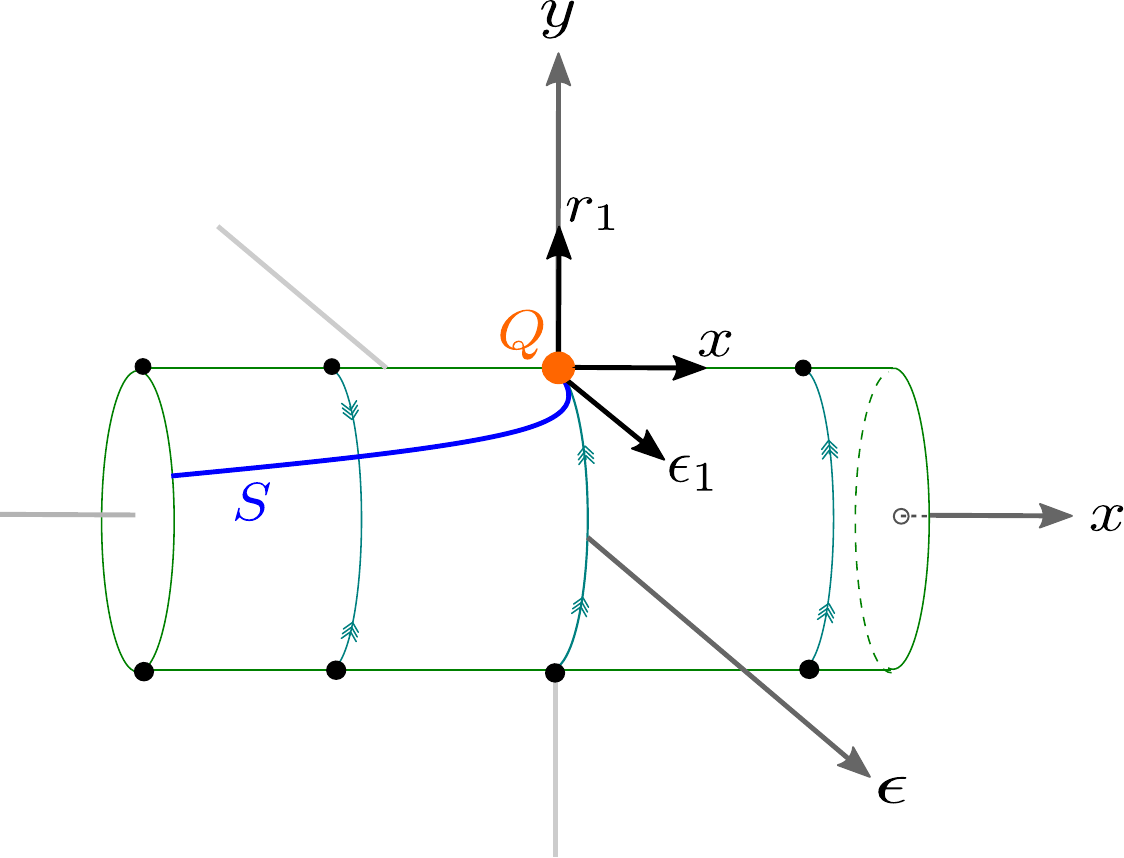}}
	\caption{Effect of the cylindrical blow-up \eqref{eq:bu_cyl}. (a) The switching manifold $\Sigma$ is shown in green, embedded in the extended $(x,y,\epsilon)-$space. The tangency point is shown in orange. (b) Dynamics and geometry following cylindrical blow-up of $\Sigma \times \{0\}$. The loss of smoothness along $\Sigma \times \{0\}$ has been resolved, but a degenerate point $Q$ (also in orange) stemming from the tangency point persists. An attracting critical manifold $S$ terminating at $Q$ is identified on the cylinder, and shown here in blue. The local projective coordinates $(x,r_1,\epsilon_1)$ defined in \eqref{eq:K1} and centered at $Q$ are also shown.}
	\label{fig:bu1}
\end{figure}

\

The critical manifold $S$ terminates tangentially to the fast flow at a degenerate point $Q \in \{r = \epsilon =0 \}$, which is also a point of tangency with the outer dynamics induced by the vector field $X^+$ within $\{\epsilon = 0\}$; \KUKKK{see again Figure \ref{fig:bu1}(b)}. 
Choosing local coordinates of the form
\begin{equation}
\label{eq:K1}
\bar y = 1 : \  y = r_1, \qquad \epsilon = r_1 \epsilon_1 ,
\end{equation}
with $x$ unchanged, this degeneracy is identified as a fully nonhyperbolic (\KUKKK{i.e. no eigenvalues with non-zero real part}) equilibrium \SJ{at $(x,r_1,\epsilon_1) = (0,0,0)$}.

The point $Q$ is degenerate for all $\mu \in \mathbb R$, however degeneracy stemming from the presence of the tangency is resolved via the weighted spherical blow-up
\begin{equation}
\label{eq:bu_Q}
\rho \geq 0, \ \left(\hat x, \hat r, \hat \epsilon \right) \in S^2 \mapsto
\begin{cases}
x = \rho^{k(1+k)} \hat x , \\
r_1 = \rho^{2k(1+k)} \hat r , \\
\epsilon_1 = \rho^{1+k} \hat \epsilon ,
\end{cases}
\end{equation}
where $k := k_+ \in \mathbb N_+$ is the decay exponent associated with the regularization function $\phi$, see equation \eqref{eq:reg_asymptotics}. After another desingularization (division by $\rho^{k(1+k)}$), nontrivial dynamics are identified within the invariant subspace $\{\rho = 0\}$ corresponding to the blow-up sphere $\{\rho = 0\} \times S^2$. The critical manifold $S$ now connects to a partially hyperbolic and \SJ{(partially)} attracting (\KUKKK{i.e. there is an eigenvalue with negative real part}) equilibrium $p_a$ contained within the intersection of the blow-up cylinder and blow-up sphere, see \KUKKK{Figure \ref{fig:bu2}(a)}. An attracting center manifold $\mathcal W \in \{\rho=0\}$ emanates from $p_a$, thereby `extending' $S$. Whether or not equilibria are also identified along the intersection of the blow-up sphere with $\{\epsilon = 0\}$ depends on whether the corresponding BEB is type BS, BN or BF, as well as \WM{on the sign of $\mu$}; see Figure \ref{fig:bu2}(b), (c) and (d) \SJJ{(additional equilibria arising in cases BN and BS are denoted $q_w$ and $q_o$ as in Figure \ref{fig:bu2}(c))}. 

\begin{figure}[t!]
	\centering
	\subfigure[]{\includegraphics[width=.49\textwidth]{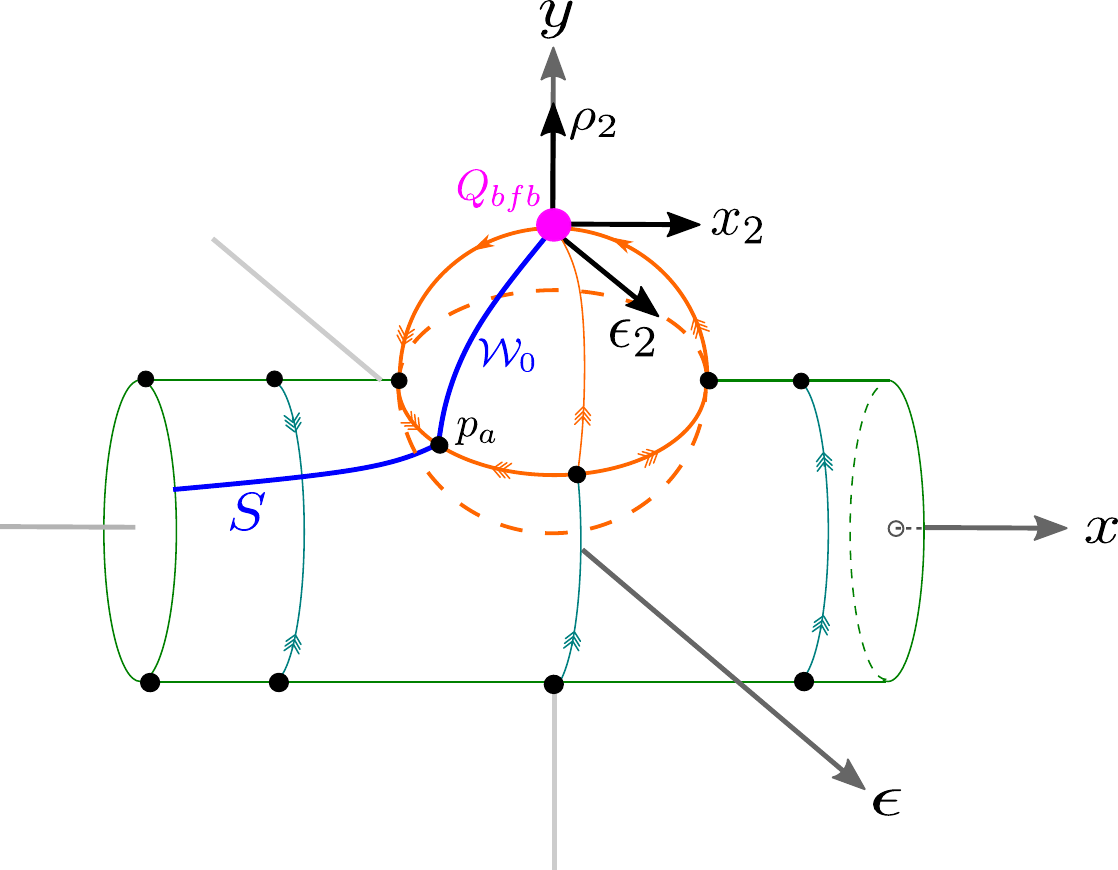}}
	\subfigure[]{\includegraphics[width=.49\textwidth]{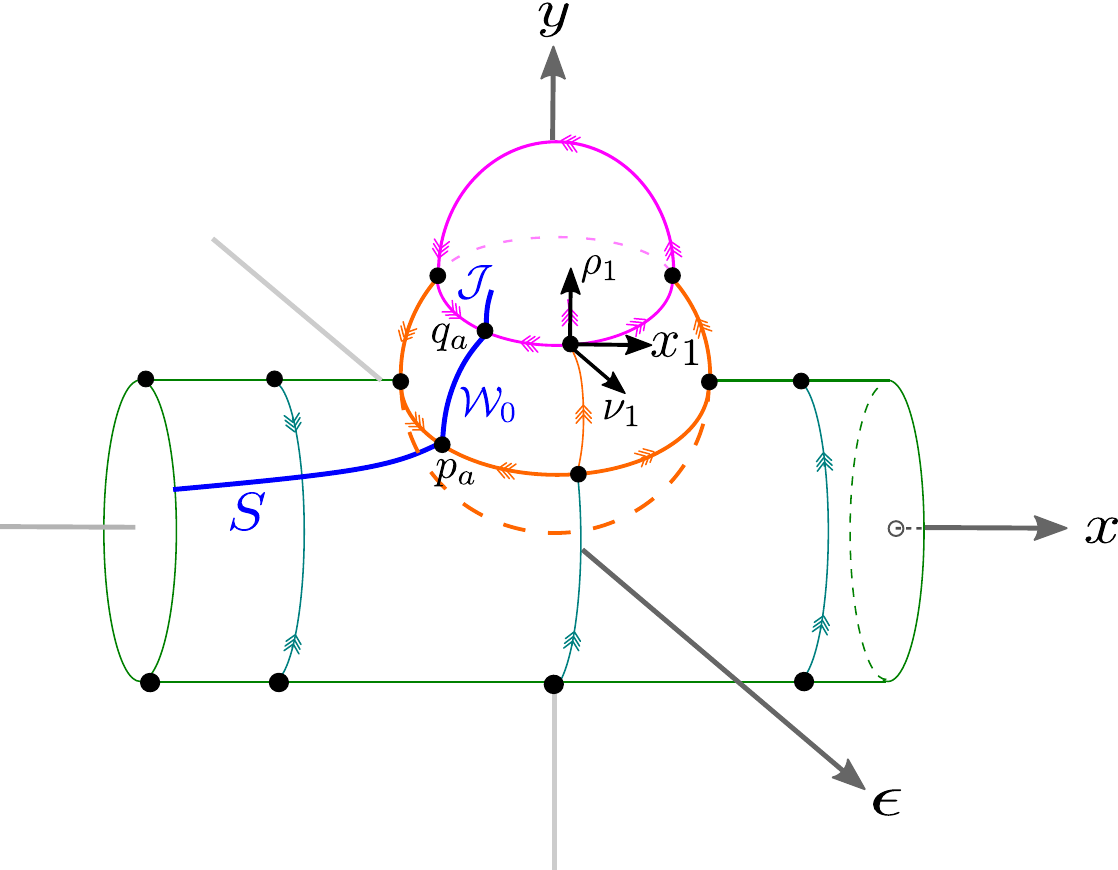}}
	
	\subfigure[]{\includegraphics[width=.49\textwidth]{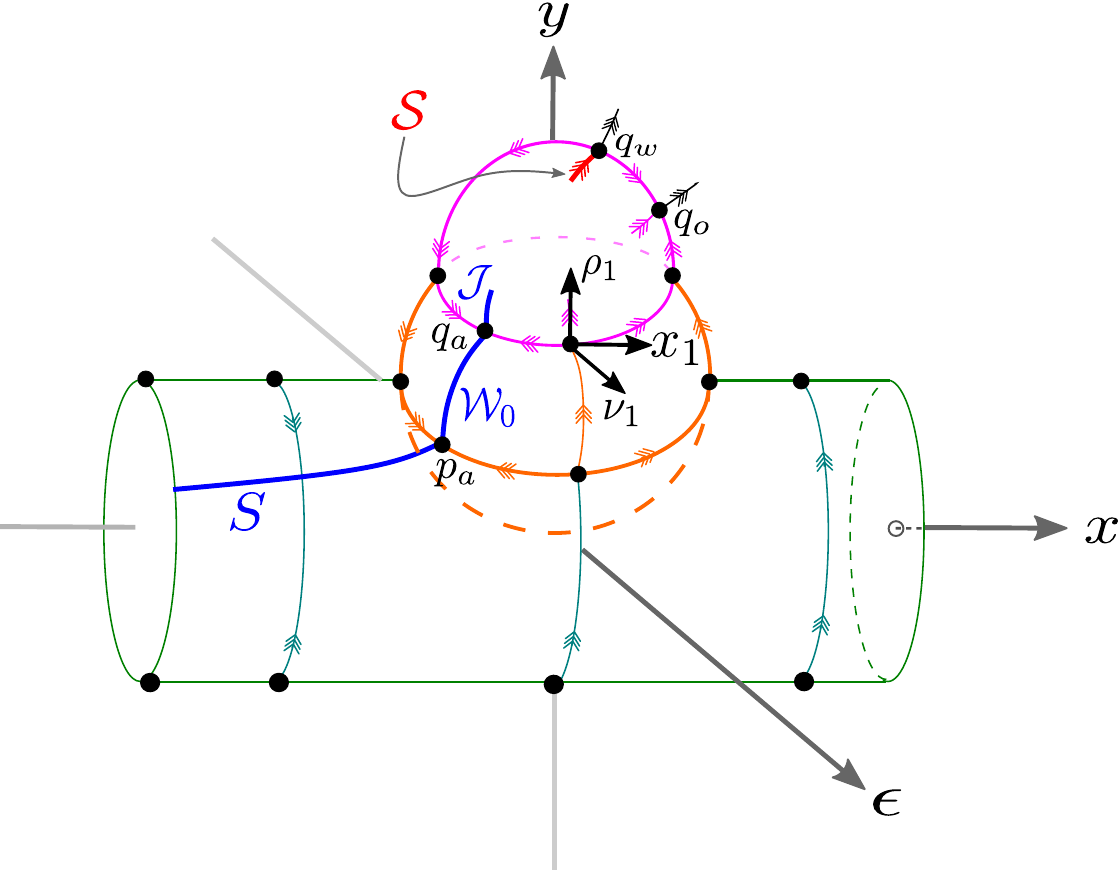}}
	\subfigure[]{\includegraphics[width=.49\textwidth]{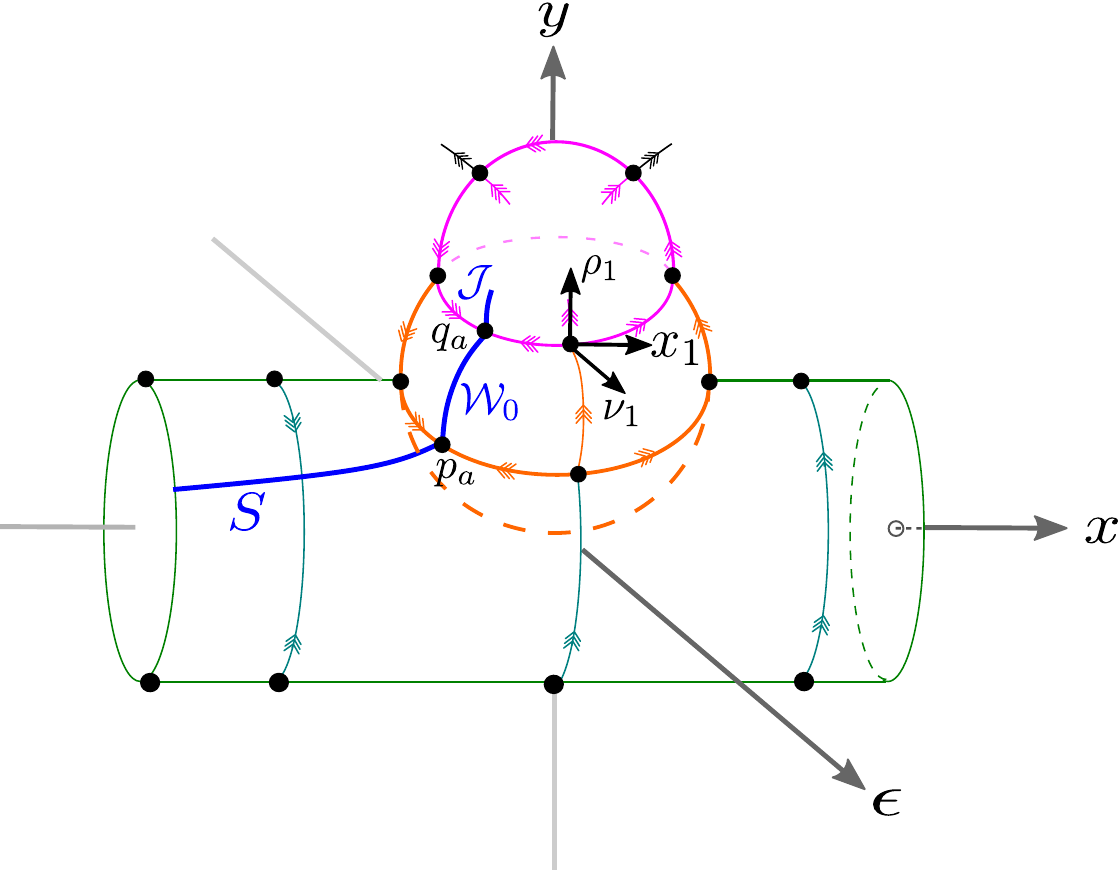}}
	\caption{(a) Dynamics and geometry after spherical blow-up of $Q$ via \eqref{eq:bu_Q}. The critical manifold $S$ in blue connects to the blow-up sphere at an attracting, partially hyperbolic point $p_a$, and an attracting center manifold $\mathcal W$, also in blue, emanates from $p_a$ over the blow-up sphere shown in orange. If $\mu \neq 0$, all degeneracy is resolved. For $\mu = 0$, the case shown here, $\mathcal W$ is a critical manifold $\mathcal W_0$ which connects to the degenerate point $Q_{bfb}$ (magenta), which corresponds to the BE singularity. Local coordinates $(x_2,\rho_2,\epsilon_2)$ centered at $Q_{bfb}$ are also shown. (b) Dynamics and geometry following spherical blow-up of $Q_{bfb}$ via \eqref{eq:bu_Qbfb} in case BF. By restricting to the invariant set defined by the scaling \eqref{eq:scaling}, the blow-up 3-sphere (magenta) can be projected into into $(\check x, \check \rho, \check \epsilon)-$space as described in the text, and plotted in 3D. Following blow-up, $\mathcal W_0$ connects to an attracting, partially hyperbolic point $q_a$. An attracting center manifold $\mathcal J$ contained within $\{\nu = 0\}$, also in blue, extends from $q_a$ onto the new blow-up sphere. Local coordinates $(x_1,\rho_1,\nu_1)$ defined via \eqref{eq:coord_change} used to describe the dynamics on the sphere are also shown. (c) resp.~(d) Dynamics and geometry after blow-up in cases BN resp.~BS. Here one identifies additional equilibria within $\{\nu = \check \epsilon = \check \mu = 0 \}$.} 
	\label{fig:bu2}
\end{figure}
\

It follows from previous work \cite{Jelbart2020d,Kristiansen2019c} that for each fixed $\mu \neq 0$, the blow-up transformations \eqref{eq:bu_cyl} and \eqref{eq:bu_Q} are sufficient to resolve all degeneracies in system \eqref{eq:normal_form}. For $\mu = 0$, an additional degeneracy persists due to the BE singularity. In this case, $\mathcal W$ becomes an attracting critical manifold $\mathcal W_0$, and connects to another degenerate point $Q_{bfb} \in \{\rho = \epsilon = 0\}$ at the top of the blow-up sphere \cite{Jelbart2020d}. This case is shown in Figure \ref{fig:bu2}(a). 
The point $Q_{bfb}$ is located at the origin in local coordinates $(x_2,\rho_2,\epsilon_2)$ defined by
\begin{equation}\label{eq:hatr1}
\hat r = 1 : \  x = \rho_2^{k(1+k)} x_2 , \qquad r_1 = \rho_2^{2k(1+k)} , \qquad \epsilon_1 = \rho_2^{1+k} \epsilon_2 ,
\end{equation}
for $\mu = 0$ only. Appending the trivial equation $\mu'=0$ to the system obtained in these coordinates, $Q_{bfb}$ is identified as a nonhyperbolic equilibrium within the extended $(x_2,\rho_2,\epsilon_2,\mu)-$space. Finally, degeneracy at $Q_{bfb}$ is resolved via the weighted spherical blow-up
\begin{equation}
\label{eq:bu_Qbfb}
\nu \geq 0, \ \left(\check x, \check \rho , \check \epsilon, \check \mu \right) \in S^3 \mapsto
\begin{cases}
x_2 = \nu^{k(1+k)} \check x , \\
\rho_2 = \nu \check \rho , \\
\epsilon_2 = \nu^{1+k} \check \epsilon  , \\
\mu = \nu^{2k(1+k)} \check \mu ,
\end{cases}
\end{equation}
which replaces $Q_{bfb}$ with the 3-sphere $\{\nu = 0\} \times S^3$. Following this spherical blow-up, and a desingularization amounting to division by \KUKKK{$\nu^{k(1+k)}$}, the critical manifold $\mathcal W_0$ terminates at a partially hyperbolic and \SJ{(partially)} attracting equilibrium $q_a$ contained within $\{\nu = \check \rho = \check \mu = 0\}$, see Figure \ref{fig:bu2}. An attracting center manifold $\mathcal J$ contained within $\{\nu = 0\}$, i.e.~on the new blow-up sphere, emanates from $q_a$, thereby extending $\mathcal W_0$. In the case that the BEB is of either BN or BS type, one also identifies equilibria 
on the top of the blow-up sphere within $\{\nu = \check \epsilon = \check \mu = 0 \}$, \SJJ{see Figures \ref{fig:bu2}(c) and (d).}

\

The sequence of blow-up transformations \eqref{eq:bu_cyl}, \eqref{eq:bu_Q} and \eqref{eq:bu_Qbfb} can be written in the following form upon composition:
\begin{align}\label{eq:comp_final}
\nu\ge 0,\ \left(\check x, \check \rho , \check \epsilon, \check \mu \right) \in S^3 \mapsto
\begin{cases}
x = \nu^{2k(1+k)} \check \rho^{k(1+k)} \check x , \\
y= \nu^{2k(1+k)}  \check \rho^{2k(1+k)} , \\
\epsilon = \nu^{2(1+k)^2} \check \rho^{\WM{(2k+1)}(1+k)} \check \epsilon  , \\
\mu = \nu^{2k(1+k)} \check \mu .
\end{cases}
\end{align}

{\begin{remark}
\SJJ{Note that the $\mu-$coordinate is not shown in Figure \ref{fig:bu2}(b), (c) or (d).} Due to the conservation of $\mu$ and the original small parameter $\epsilon$ \KUKK{it follows that}
\begin{align}
\hat \mu:=\frac{\mu}{\epsilon^{k/(1+k)}} = \check \mu \check \rho^{-{k(2k+1)}} \check \epsilon^{-k/(1+k)},\label{eq:scaling}
\end{align}
is \KUKK{also} a conserved quantity, \KUKK{even} for $\nu=0$. This conserved quantity induces a foliation of the blow-up 3-sphere by lower-dimensional 2-spheres parameterized by $\hat \mu \in \mathbb R$, thereby permitting a 3-dimensional representation as in Figure \ref{fig:bu2}. In the following we will, when it is convenient to do so, view $\hat \mu$ as our bifurcation parameter on the sphere. 
\end{remark}
}

Applying \SJJ{\eqref{eq:comp_final}} to the doubly extended system
\begin{equation}
\label{eq:doubly_extended}
\{(x',y')=\epsilon X(x,y,\mu,\epsilon),\epsilon'=0,\mu'=0\} ,
\end{equation}
and performing a desingularization which corresponds to division of the right-hand-side by $\nu^{2(1+k)^2} \check \rho^{(1+k)^2} \check \epsilon$, resolves all degeneracy in system \eqref{eq:normal_form}. This enables a description of the unfolding of the singularly perturbed BEBs for all $0<\epsilon\ll 1$. 

\begin{lemma}
	\label{lem:desing_prob}
	
	A desingularized system governing the singular limit dynamics in the scaling regime \SJ{defined by $\mu = \hat \mu \epsilon^{k/(1+k)}$} 
	can be obtained from the doubly extended system \eqref{eq:doubly_extended} 
	by an application of the coordinate transformation
	\begin{equation}
	\label{eq:coord_change}
	\left(x_1, \nu_1, \rho_1, \hat \mu \right) \in \mathbb R \times \mathbb R_+^2 \times \mathbb R \mapsto
	\begin{cases}
	x = \nu_1^{2k(1+k)} \rho_1^{k(1+k)} x_1 , \\
	y = \nu_1^{2k(1+k)} \rho_1^{2k(1+k)} , \\
	\epsilon = \nu_1^{2(1+k)^2} \rho_1^{(1+k)(1+2k)} , \\
	\mu = \hat \mu \nu_1^{2k(1+k)} \rho_1^{k(1+2k)} ,
	\end{cases}
	\end{equation}
	followed by the desingularization
	\begin{equation}
	\label{eq:time_change}
	d\tilde t = \nu_1^{2(1+k)^2} \rho_1^{(1+k)^2} dt ,
	\end{equation}
	and finally, restriction to the invariant subspace $\{\nu_1 = 0\}$ corresponding to $\epsilon = 0$. 
	The resulting system is
	\begin{equation}
	\label{eq:desing_prob}
	\begin{split}
	x_1' &= \rho_1^{k(1+k)} \left((\tau - \gamma) \beta + \hat \mu \rho_1^{k^2} + \tau x_1 - \delta \rho_1^{k(1+k)} \right) + k x_1 \left(\beta + x_1 \right) , \\
	\rho_1' &= \frac{1}{k} \rho_1 \left(\beta + x_1 \right) ,
	\end{split}
	\end{equation}
	\SJJ{where we write $\beta := \beta_+ = \phi_+(0)$.} Moreover, system \eqref{eq:desing_prob} is topologically equivalent to 
	\begin{equation}
	\label{eq:desing_prob2}
	\begin{split}
	X' &=\left(\hat \mu+\tau X -\delta Y \right) Y^k -  (\gamma - \tau)\beta ,\\
	Y'&= X \hat Y^k+\beta ,
	\end{split}
	\end{equation}	
	on $\{Y>0\}$ via the diffeomorphism defined by
	\begin{align}\label{eq:XYhere}
	(X,Y)\mapsto \begin{cases}x_1 &= Y^{k} X,\\
	\rho_1 &= Y^{1/k}.
	\end{cases}
	\end{align}
	
\end{lemma}

\begin{proof}
	The transformation \eqref{eq:coord_change} is simply obtained from \eqref{eq:comp_final} by working in the chart $\check \epsilon=1$ with chart-specific coordinates $(x_1,\nu_1,\rho_1,\mu_1)$ defined by
\begin{align*}
	x &= \nu_1^{2k(1+k)} \rho_1^{k(1+k)} x_1 , \\
	y &= \nu_1^{2k(1+k)} \rho_1^{2k(1+k)} , \\
	\epsilon &= \nu_1^{2(1+k)^2} \rho_1^{(1+k)(1+2k)} , \\
	\mu &= \nu_1^{2k(1+k)} \mu_1.
	\end{align*}
	In this chart, $\hat \mu=\mu_1 \rho_1^{-k\WM{(2k+1)}}$, recall \eqref{eq:scaling}, which gives the desired result upon using this expression to eliminate $\mu_1$. From this, we obtain \eqref{eq:desing_prob2} by a calculation, see \SJ{Lemma} \ref{lem:desing_ext} below for further details as well as \cite[Lemma 3.2 and Remark 3.4]{Jelbart2020d}.
\end{proof}

Both systems \eqref{eq:desing_prob} and \eqref{eq:desing_prob2} are useful for describing the unfolding of singularly perturbed BEB in system \eqref{eq:normal_form}. System \eqref{eq:desing_prob} arises from a central projection of the final blow-up transformation, and is preferred for purposes of global computations within the blown-up space. System \eqref{eq:desing_prob2} is derived by a direct parameter rescaling, \KUK{and although it is preferred for local computations pertaining to e.g.~bifurcations, it is less suited to global analyses.}\footnote{\SJ{See Remark \ref{rem:here} for more details.}}

\KUK{Notice however, that \eqref{eq:desing_prob2} can also be obtained more directly by composing \eqref{eq:XYhere} with \eqref{eq:coord_change}. This gives
	\begin{align}
	(X,Y,\epsilon,\hat \mu) \mapsto \begin{cases}
	x =\epsilon^{k/(1+k)} X\\
	y =\epsilon^{k/(1+k)} Y,\\
	\mu =\epsilon^{k/(1+k)}\hat \mu,
	\end{cases}\label{eq:xyXY}
	\end{align}
	after eliminating $\nu_1$. Inserting this into \eqref{eq:normal_form} gives \eqref{eq:desing_prob2} for $\epsilon\rightarrow 0$ upon desingularization.  }

It is also possible to scale $x$ and $y$ by $\mu$ for $\mu>0$ instead of $\epsilon$; in fact, this is more well-suited for $\hat \mu\rightarrow \infty$. Therefore if we define
	\begin{align*}
	\hat \epsilon = \hat \mu^{-(1+k)/k}
	\end{align*}
	then
	\begin{align}
	(\widehat X,\widehat Y,\mu,\hat \epsilon) \mapsto \begin{cases}
	x = \mu \widehat X\\
	y =\mu \widehat Y,\\
	\epsilon =\mu^{(1+k)/k} \hat \epsilon,
	\end{cases}\label{eq:xyXYpm}
	\end{align}
	transforms \eqref{eq:normal_form} into following system
	\begin{equation}\label{eq:xyXYpmeqs}
	\begin{aligned}
	{\widehat X}' &=(1+\tau {\widehat X} -\delta {\widehat Y}) {\widehat Y}^k-(\gamma-\tau)\beta \hat \epsilon^k ,\\
	{\widehat Y}'  &= {\widehat X} {\widehat Y}^k +\beta \hat \epsilon^k,
	\end{aligned}
	\end{equation}
	for $\mu\rightarrow 0$ upon desingularization.
	\SJJ{System} \eqref{eq:xyXYpmeqs} \SJJ{is} smoothly topologically equivalent \SJJ{to} \eqref{eq:desing_prob2} on $\hat \mu> 0$ through the transformation
	\begin{align*}
	(X,Y,\hat \mu)\mapsto \begin{cases} {\widehat X}  =\hat \mu^{-1} X,\\
	{\widehat Y} =\hat \mu^{-1} Y.
	\end{cases}
	\end{align*}
The limit $\hat \mu\rightarrow -\infty$ can be studied via an analogous scaling by $-\mu$ for $\mu<0$, but we will not need this in our analysis.


\begin{remark}
	\label{rem:quantitative}
	In this work we focus on the qualitative dynamics near a nondegenerate BE bifurcation. For general systems \eqref{eq:general} with a BE bifurcation at $(z,\alpha) = (z_{bf}, \alpha_{bf})$, however, \KUK{Lemma \ref{lem:desing_prob} offers a direct route to obtain quantitative information about the dynamics without the need for bringing the system into normal form, by first shifting $(\tilde z, \tilde \alpha) = (z - z_{bf}, \alpha - \alpha_{bf})$, and then applying the coordinate transformation \eqref{eq:coord_change} and desingularization by \eqref{eq:time_change} with $\tilde z=(x,y)$.}
\end{remark}

\subsection{Unfolding all 12 singularly perturbed BEBs}
\label{ssec:results_bifurcatons}

The limiting bifurcation structure can be derived for each case using either of the systems \eqref{eq:desing_prob} or \eqref{eq:desing_prob2}. We may consider system \eqref{eq:desing_prob2} for simplicity, which by \cite[Lemma 3.5]{Jelbart2020d} has either 0, 1 or 2 equilibria. The corresponding bifurcation diagrams with $\epsilon\ll1$ are obtained after lifting results for $\epsilon = 0$. 



\begin{thm}
	\label{thm:bifs}
	Consider system \eqref{eq:normal_form}. There exists an $\epsilon_0>0$ such that for all $\epsilon \in (0,\epsilon_0)$, the following assertions hold:
	\begin{enumerate}
		\item[(i)] Fix $\gamma/\delta>0$. Saddle-node bifurcation occurs for $\mu = \mu_{sn}(\gamma,\epsilon)$, where
		\begin{equation}
		\label{eq:1sn}
		\mu_{sn}(\gamma,\epsilon) = \frac{(1+k) \delta} k \left(\frac{k \beta \gamma}{\delta} \right)^{1/(1+k)} \epsilon^{k/(1+k)} + o\left(\epsilon^{k/(1+k)}\right) . 
		\end{equation}
		\item[(ii)] Fix $\tau < 0$ and $\gamma < \delta/\tau$. Supercritical Andronov-Hopf bifurcation occurs for $\mu = \mu_{ah}(\gamma,\epsilon)$, where
		\begin{equation}
		\label{eq:1ah}
		\mu_{ah}(\gamma,\epsilon) = \frac{k \delta + \tau \gamma}{k} \left(\frac {k\beta} \tau \right)^{1/(1+k)} \epsilon^{k/(1+k)} + o\left(\epsilon^{k/(1+k)}\right) . 
		\end{equation}
		\item[(iii)] Fix $\tau > 0$. Parameter-space surfaces defining saddle-node and Andronov-Hopf bifurcations in $(\gamma,\mu,\epsilon)-$space extend to intersect in a curve of supercritical Bogdanov-Takens bifurcations given by
		\begin{equation}
		\label{eq:1bt}
		(\mu_{bt},\gamma_{bt})(\epsilon) = \left(\frac{(1+k) \delta}{k} \left(\frac{k \beta}{\tau} \right)^{1/(1+k)} \epsilon^{k/(1+k)} + o\left(\epsilon^{k/(1+k)}\right) , \frac \delta \tau \right) . 
		\end{equation}
		\item[(iv)] Fix $\tau > 0$ and $0< \gamma < \delta / \tau$. Homoclinic-to-saddle bifurcation occurs along $\mu = \mu_{hom}(\gamma,\epsilon)$, which is given locally near $(\mu_{bt},\gamma_{bt})(\epsilon)$ by 
		\[
		\begin{split}
		\mu_{hom}(\gamma,\epsilon) &= 
		\left[\left(\frac{k \beta}{\tau} \right)^{1/(1+k)} \left(\frac{(1+k) \delta} k + \frac{\tau}{k} \left(\gamma - \frac{\delta}{\tau}\right) \right) + \mathcal O\left( \left(\gamma - \frac{\delta}{\tau}\right)^2 \right) \right]
		\epsilon^{k/(1+k)} \\
		&+ o\left(\epsilon^{k/(1+k)}\right) .
		\end{split}
		\]
		There is no homoclinic bifurcation for $\gamma < 0$.
		\item[(v)] Viewed within the $(\gamma,\mu)-$plane, the curves $\mu_{sn}(\gamma,\epsilon)$, $\mu_{ah}(\gamma,\epsilon)$ and $ \mu_{hom}(\gamma,\epsilon)$ are all quadratically tangent at $(\gamma_{bt},\mu_{bt})(\epsilon)$ and satisfy
		\[
		0 < \mu_{sn}(\gamma,\epsilon) < \mu_{ah}(\gamma,\epsilon) < \mu_{hom}(\gamma,\epsilon) ,
		\]
		where all three coexist.
	\end{enumerate}
\end{thm}

\begin{figure}[t!]
	\centering
	\subfigure[]{\includegraphics[width=.48\textwidth]{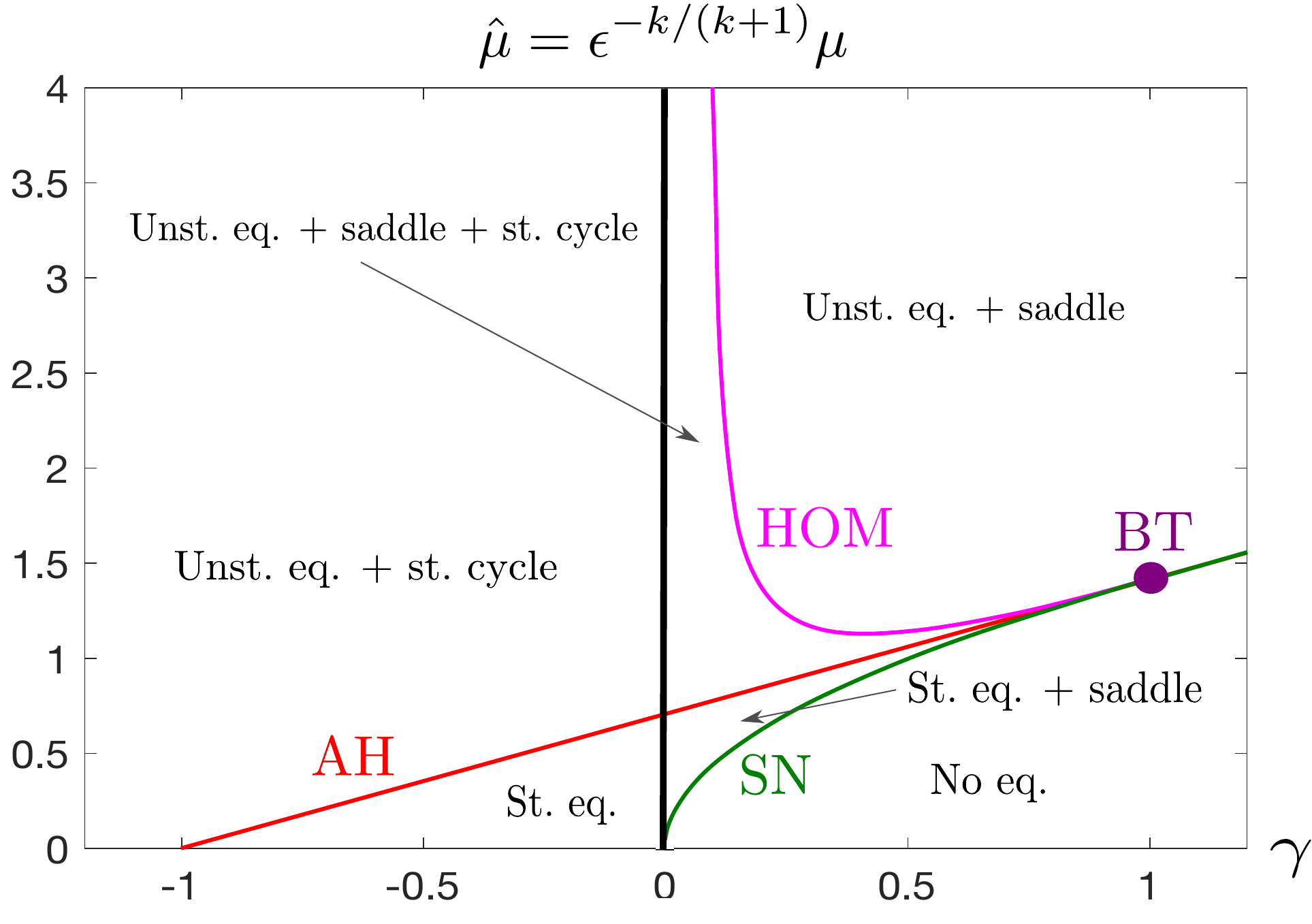}}
	\subfigure[]{\includegraphics[width=.48\textwidth]{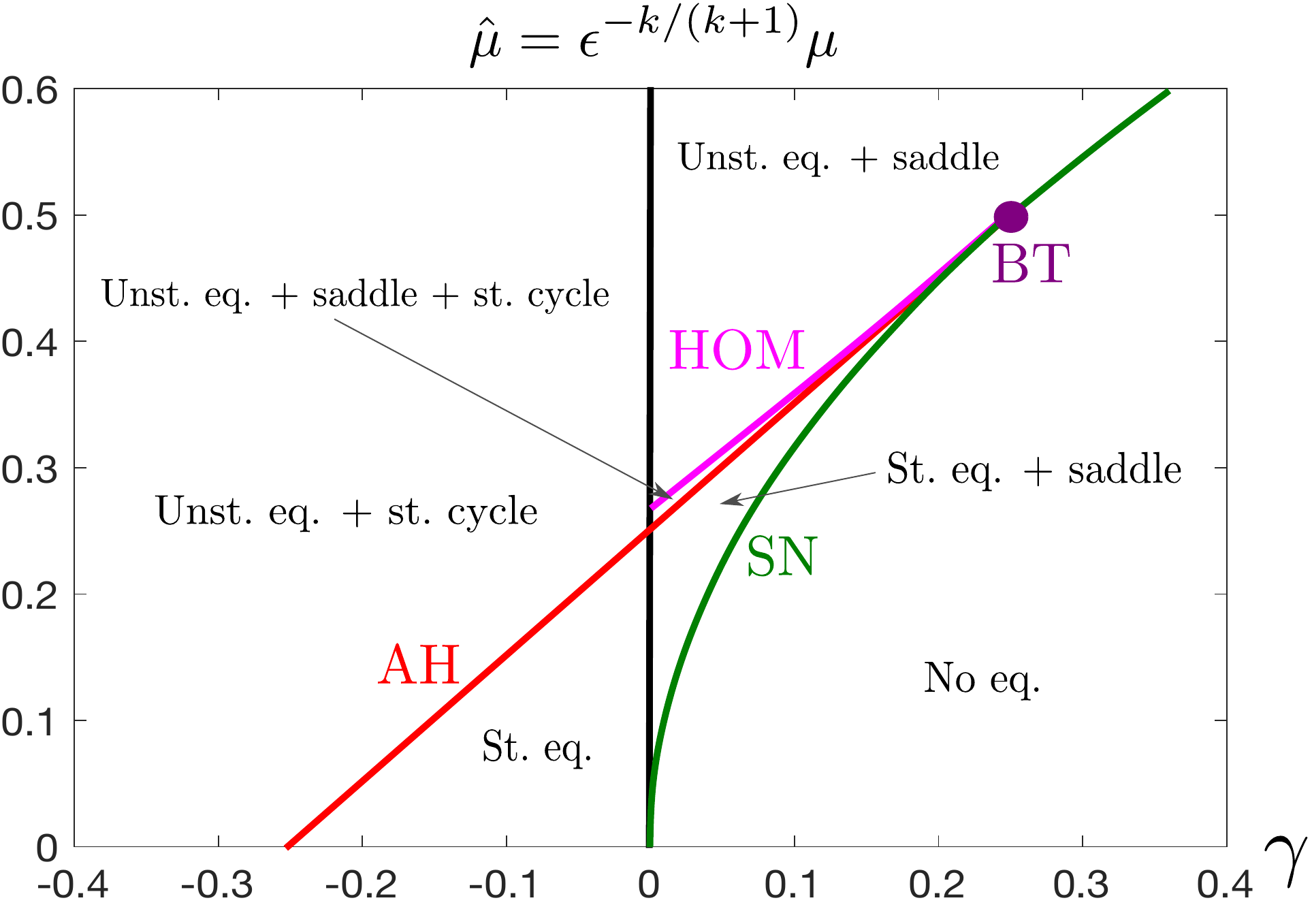}}
	
	\subfigure[]{\includegraphics[width=.48\textwidth]{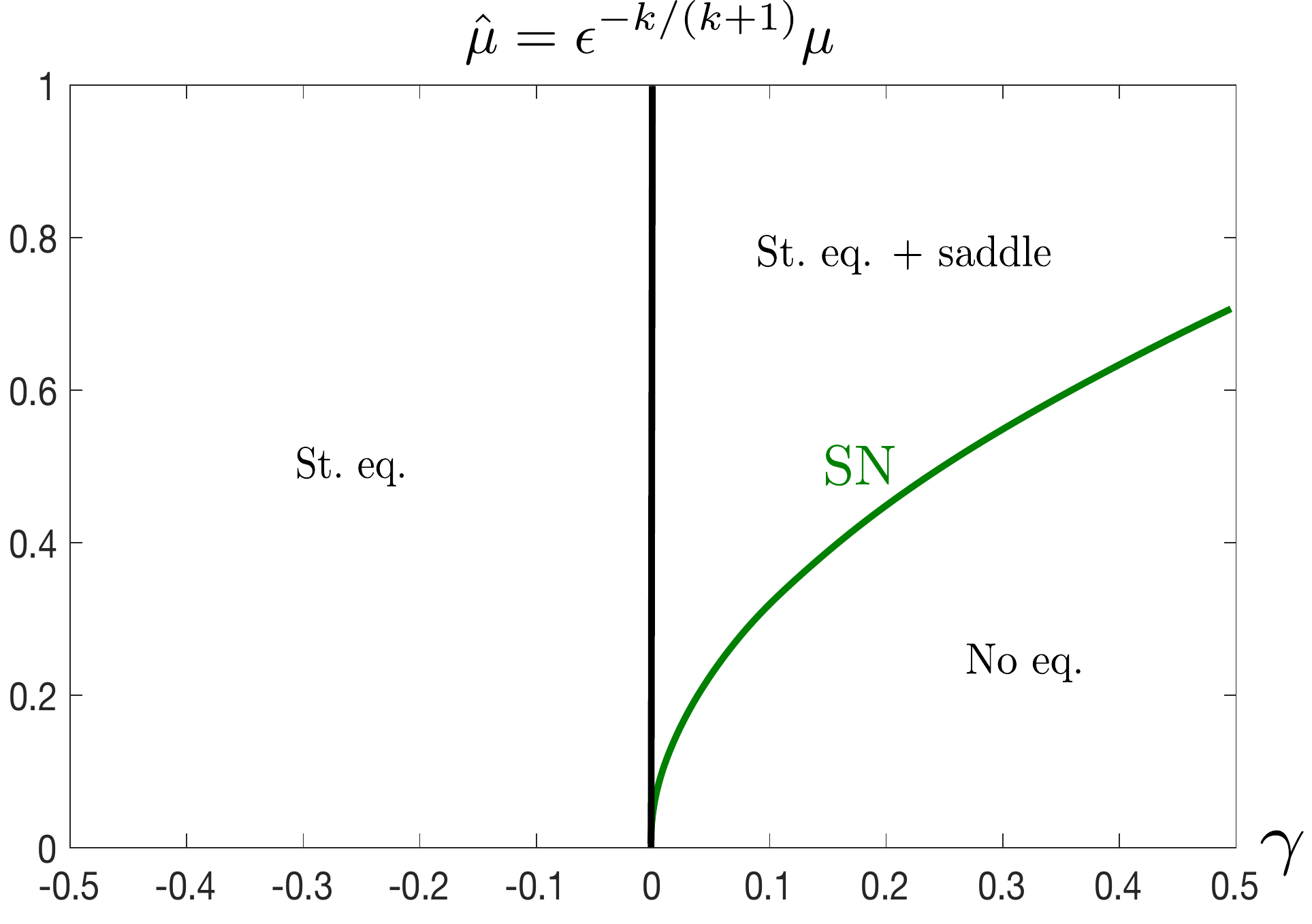}}
	\subfigure[]{\includegraphics[width=.48\textwidth]{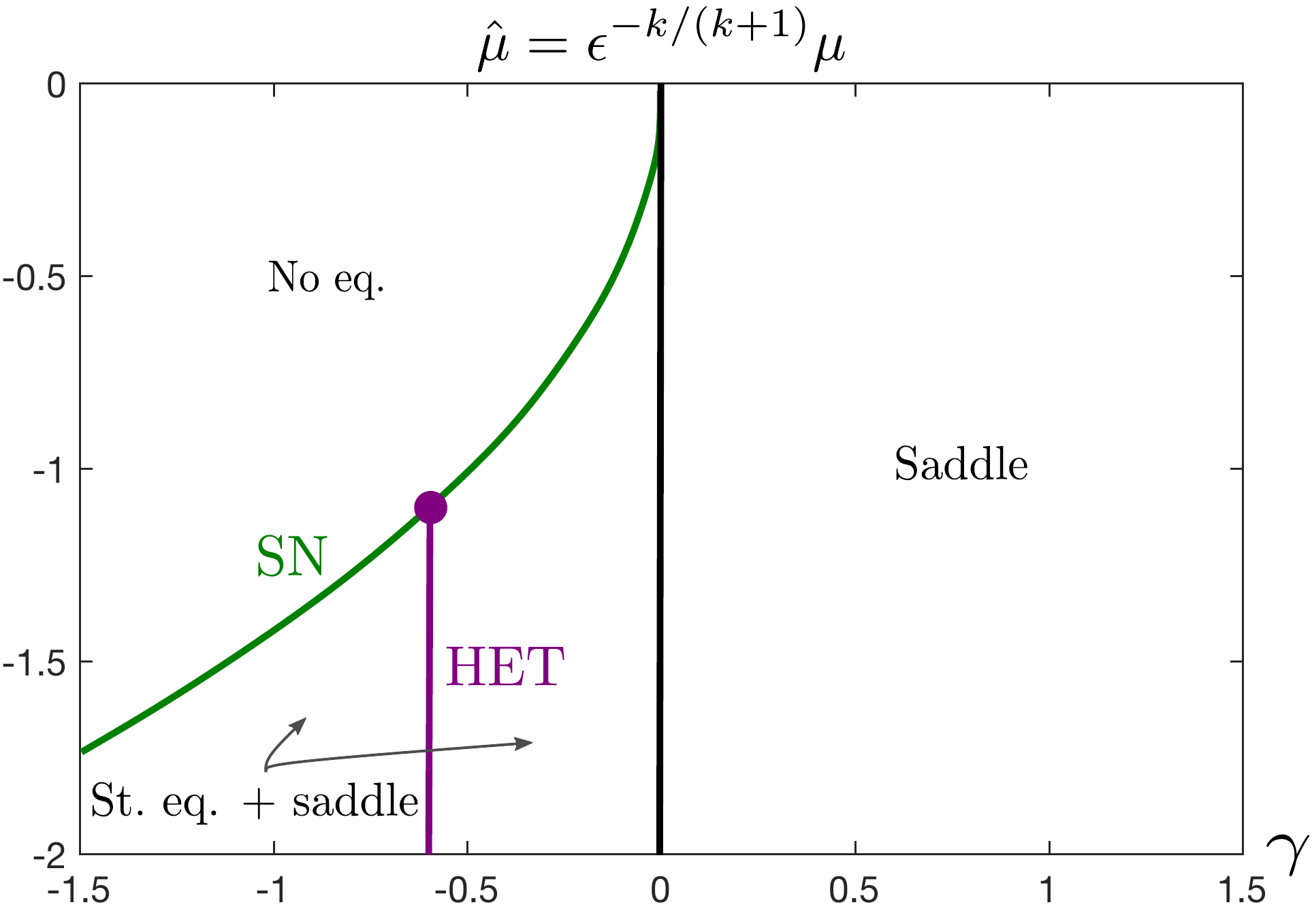}}
	\caption{2-parameter bifurcation diagrams for all 12 unfoldings 
	for the desingularized system \eqref{eq:desing_prob2}. Saddle-node (SN), supercritical Andronov-Hopf (AH), homoclinic (HOM) and Bogdanov-Takens (BT) bifurcations are shown in green, red, magenta and purple respectively. Homoclinic curves were computed numerically using MatCont \cite{MATCONT}. \SJ{Here} $(k,\beta) = (1,1/2)$. Unfoldings corresponding to BE singularities with I/O orientation of the Filippov flow can be plotted on the same diagram since I/O correspond to $\gamma<0$/$\gamma > 0$, while $\gamma = 0$ is omitted. (a): Cases BF$_{1,2,3}$, with $\tau = \delta = 1$. Cases BF$_{1,2}$ are contained within $\gamma > 0$ and separated by the homoclinic curve, with BF$_1$ (BF$_2$) on the left (right). BF$_3$ is contained within $\gamma < 0$. (b): Cases BN$_{2,3}$ with $\tau = 2$, $\delta = 1/2$. BN$_3$ (BN$_2$) is contained within $\gamma < 0$ ($\gamma>0$). Note the possibility for oscillatory dynamics in case BN$_2$. 
	(c): Cases BN$_{1,4}$ with $\tau=-2$, $\delta = 1/2$. BN$_1$ (BN$_4$) is contained within $\gamma < 0$ ($\gamma>0$). \WM{We do not show cases BF$_{4,5}$ here, since they are qualitatively similar to BN$_{1,4}$.}
	(d) Cases BS$_{1,2,3}$ \SJ{with $\tau=1, \ \delta = -1$}. Cases BS$_{1,2}$ are contained within $\gamma < 0$ and separated by a \SJ{(numerically computed)} distinguished heteroclinic, \SJ{denoted HET and shown in purple} (see item (iv) in the text). Case BS$_3$ is contained within $\gamma > 0$. The diagrams in (a)-(c) all extend for $\hat \mu < 0$, and the diagram in (d) extends for $\hat \mu > 0$.}
	\label{fig:bifs}
\end{figure}

A proof for Theorem \ref{thm:bifs} based on an adaptation of the proof of \cite[Theorem 3.6]{Jelbart2020d} is given in Section \ref{ssec:proof_of_the_bifurcation_results}. The idea is that bifurcations can be identified first for the desingularized system \eqref{eq:desing_prob2}, for which saddle-node, Andronov-Hopf and homoclinic bifurcations are identified along parameter space curves given by
\begin{equation}
\label{eq:sn_desing}
\hat \mu_{sn}(\gamma) := \lim\limits_{\epsilon \to 0} \frac{\mu_{sn}(\gamma,\epsilon)}{\epsilon^{k/(1+k)}}
= \frac{(1+k) \delta} k \left(\frac{k \beta \gamma}{\delta} \right)^{1/(1+k)} , \qquad 
\frac \gamma \delta > 0 ,
\end{equation}
\begin{equation}
\label{eq:ah_desing}
\hat \mu_{ah}(\gamma) := \lim\limits_{\epsilon \to 0} \frac{\mu_{ah}(\gamma,\epsilon)}{\epsilon^{k/(1+k)}}
= \frac{k \delta + \tau \gamma}{k} \left(\frac {k\beta} \tau \right)^{1/(1+k)} , \qquad 
\gamma \in \left(- \infty, \frac{\delta}{\tau} \right), \ \tau > 0 ,
\end{equation}
and 
\begin{equation}
\label{eq:hom_desing}
\begin{split}
\hat \mu_{hom}(\gamma) :&= \lim\limits_{\epsilon \to 0} \frac{\mu_{hom}(\gamma,\epsilon)}{\epsilon^{k/(1+k)}} \\
&= \left(\frac{k \beta}{\tau} \right)^{1/(1+k)} \left(\frac{(1+k) \delta} k + \frac{\tau}{k} \left(\gamma - \frac{\delta}{\tau}\right) \right) + \mathcal O\left( \left(\gamma - \frac{\delta}{\tau}\right)^2 \right), 
\end{split}
\end{equation}
respectively, where $\hat \mu_{hom}(\gamma)$ is defined for $\gamma < \delta / \tau$ and $\tau > 0$ in a neighbourhood of the Bogdanov-Takens point $(\hat \mu_{bt}, \gamma_{bt}) : = (\lim_{\epsilon \to 0} \epsilon^{-k/(1+k)} \mu_{bt}(\epsilon) , \gamma_{bt})$. 

\begin{remark}
	The corresponding statement for regularized PWS systems is obtained by taking the singular limit $\epsilon \to 0$ in Theorem \ref{thm:bifs}. In this context results should be stated in terms of the desingularized system \eqref{eq:desing_prob} \SJJ{(or \eqref{eq:desing_prob2})} alone, for which the identified bifurcations are described by equations \eqref{eq:sn_desing}, \eqref{eq:ah_desing} and \eqref{eq:hom_desing}. 
\end{remark}

Theorem \ref{thm:bifs} yields four qualitatively distinct 2-parameter bifurcation diagrams. These are shown for the desingularized system \eqref{eq:desing_prob2}, i.e.~in the limit $\epsilon \to 0$, in Figure \ref{fig:bifs}. Theorem \ref{thm:bifs} asserts that the corresponding diagrams for $\epsilon \ll 1$ sufficiently small are qualitatively similar. We make the following observations with respect to Figure \ref{fig:bifs}:
\begin{enumerate}
	\item[(i)] All 12 singularly perturbed BEBs are represented: BF$_{1,2,3}$ in (a), BN$_{2,3}$ in (b), BN$_{1,4}$ in (c), and BS$_{1,2,3}$ in (d). Cases BF$_{4,5}$ are qualitatively similar to BN$_{1,4}$ respectively in (c).
	\item[(ii)] Cases for which the underlying PWS BE has an incoming (outgoing) Filippov flow, see again Figure \ref{fig:be} and Table \ref{tab:classification}, are contained within $\gamma < 0$ ($\gamma > 0$).
	\item[(iii)] Cases BF$_{1,2}$ are both contained within $\gamma > 0$ in (a). The homoclinic branch represents the continuation of the separatrix which constitutes a boundary between the two cases, with BF$_1$ (BF$_2$) lying the the left (right) of this curve. Theorem \ref{thm:bifs} only provides a local parameterisation of the homoclinic curve. A global parameterisation is not given in this work; homoclinic curves in Figure \ref{fig:bifs} have been obtained by numerical continuation using MatCont \cite{MATCONT}. \SJ{However,} additional properties of the homoclinic branch in Figure \ref{fig:bifs}(a) are also described in Section \ref{ssec:results_separatrices}. 
	\item[(iv)] Cases BS$_{1,2}$ are both contained within $\gamma < 0$, and separated by a distinguished solution which connects the unstable manifold of the saddle \SJ{along} the (unique) trajectory which is tangent to the strong eigendirection \SJ{of} the stable node. This is \SJ{also discussed} in Section \ref{ssec:results_separatrices}.
	\item[(v)] Andronov-Hopf and Bogdanov-Takens bifurcations are supercritical. Subcritical bifurcations are possible in the equivalent local normal form obtained by reversing time in system \eqref{eq:normal_form}.
	\item[(vi)] All bifurcations are `singular' in system \eqref{eq:normal_form} \WM{in} the sense that they occur within an $\epsilon-$dependent domain which shrinks to zero as $\epsilon \to 0$, at a rate prescribed by the scaling \eqref{eq:scaling}. 
	\item[(vii)] The BN$_2$ bifurcation in (b) features `hidden oscillations', i.e.~oscillations which cannot be identified in the PWS system \eqref{eq:PWS_normal_form}, within the wedge-shaped region bounded by the Andronov-Hopf and homoclinic curves.
	\item[(viii)] The decay coefficient $k \in \mathbb N_+$ associated with the regularization does not effect the topology of the bifurcation diagrams. It follows that within the class of regularizations defined by Assumptions \ref{assumption:2}-\ref{assumption:3}, the observed dynamics are qualitatively independent of the choice of regularization.
	\item[(ix)] \SJJ{Each of (non-equivalent) 2-parameter bifurcation diagram in Figure \ref{fig:bifs} can be obtained from any of the others by suitably varying the additional parameters $(\tau,\delta)$, either across one of the boundaries $\delta =0$, $\tau = 0$ or $\Delta = 0$, or through the origin $\tau = \delta = 0$; see again Table \ref{tab:classification}. A complete description of the dynamics involves the unfolding a (singular) codimension-4 bifurcation at $(\tau, \delta, \gamma ,\hat \mu) = (0,0,0,0)$. This unfolding is expected to involve (singular) codimension-3 bifurcations, and the unfolding of these bifurcations should involve the diagrams in Figure \ref{fig:bifs}.}
\end{enumerate}


\subsection{Separatrices: The boundaries between BF$_{1,2}$ and BS$_{1,2}$}
\label{ssec:results_separatrices}

In this section \SJ{we} present \SJ{a result} on the \SJ{homoclinic double-separatrix which constitutes a boundary between singularly perturbed BF$_{1,2}$ bifurcations. A heteroclinic double-separatrix forming a boundary between singularly perturbed BS$_{1,2}$ bifurcations is also discussed.}

\

The BF$_{1,2}$ boundary is formed by a \SJJ{saddle-homoclinic connection,} which is (\KUK{partially}) described in the following result. We define
		\begin{equation}
		\label{eq:gamma_hom}
		\gamma_{hom,0}:=- \frac{1}{2} e^{-\tau t_d/2} \sqrt{- \Delta} \csc \left(\frac{\sqrt{-\Delta}}{2} t_d \right),
		\end{equation}
		where $t_d$ is the first positive root of
		\begin{align}\label{eq:Rt}
		R(t): = 
		1 + e^{-\tau t / 2} \left( \frac{\tau}{\sqrt{-\Delta}} \sin \left(\frac{\sqrt{-\Delta}}{2} t \right) - \cos\left(\frac{\sqrt{-\Delta}}{2} t \right) \right).
		\end{align}

\begin{proposition}
	\label{prop:BF}
	\KUK{({Outer expansion of the homoclinic separating BF$_1$ and BF$_2$}) There exist an $E_0>0$ sufficiently small, constants $\mu_+, K > 0$ and a continuous function $\gamma_{hom}^{outer}:[0,E_0]\times [0,\mu_+] \to \mathbb R$ such that for all $(\epsilon,\mu)$ in the sector defined by 
		\begin{equation}
		\label{eq:hom_interval}
		0\le \epsilon \le E_0 \mu\quad \mbox{and}\quad \mu \in [0,\mu_+],
		\end{equation}
		system \eqref{eq:normal_form} has a \SJJ{saddle-homoclinic} $\Gamma_{hom}^{outer}(\epsilon,\mu)$ along $\gamma = \gamma_{hom}^{outer}(\epsilon\mu^{-1},\mu)$. In particular,
		\[
		\gamma_{hom}^{outer}(0,0) = \gamma_{hom,0},
		\]
		and for each fixed $\mu\in (0,\mu_+)$, $\lim_{\epsilon\rightarrow 0}\Gamma_{hom}^{outer}(\epsilon,\mu)$ is a PWS homoclinic.}

	\KUK{({Inner expansion of the homoclinic separating BF$_1$ and BF$_2$}) At the same time, there exists an $\hat \epsilon_0>0$ small and a continuous function $\gamma_{hom}^{inner}:[0,\hat \epsilon_0]\rightarrow \mathbb R$ such that for all $\hat \mu \ge \hat \epsilon_0^{-k/(1+k)}$, the system \eqref{eq:desing_prob} has a \SJJ{saddle-homoclinic} $\widehat \Gamma_{hom}^{inner}(\hat \mu)$ along $\gamma=\gamma_{hom}^{inner}(\hat \mu^{-(1+k)/k})$. In particular, \[\gamma_{hom}^{inner}(0)=\gamma_{hom,0},\] and for each fixed $\hat \mu\ge \hat \epsilon_0^{-k/(1+k)}$ there exists an $\epsilon_0>0$ small enough such that for each $\epsilon \in (0,\epsilon_0)$ there exists \SJJ{saddle-homoclinic} $\Gamma_{hom}^{inner}(\epsilon,\hat \mu)$ of \eqref{eq:normal_form} along $\gamma=\gamma_{hom}^{inner}(\hat \mu^{-(1+k)/k})+o(1)$, $\mu = \epsilon^{k/(1+k)} \hat \mu$. 
		Here $\lim_{\epsilon\rightarrow 0}\Gamma_{hom}^{inner}(\epsilon,\hat \mu)$ is just $(x,y)=(0,0)$. }

\end{proposition}


\begin{figure}[t!]
	\centering
	\subfigure[]{\includegraphics[width=.48\textwidth]{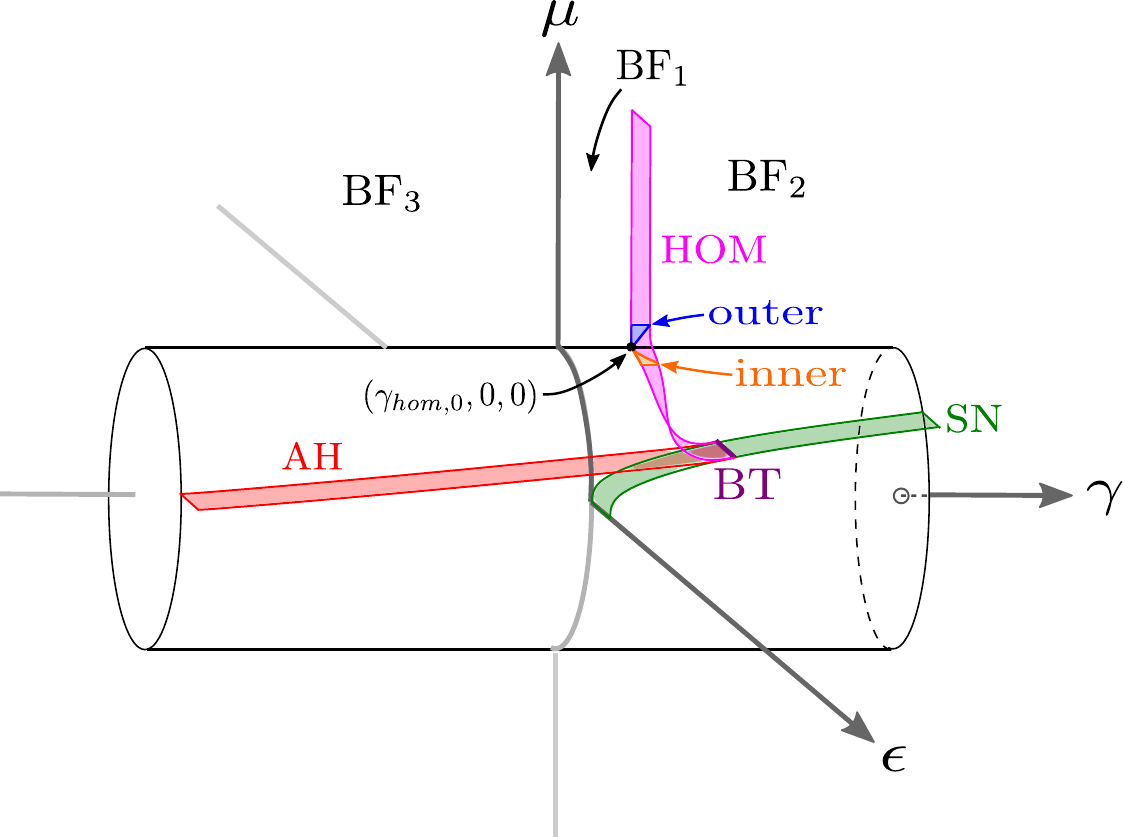}}
	\subfigure[]{\includegraphics[width=.48\textwidth]{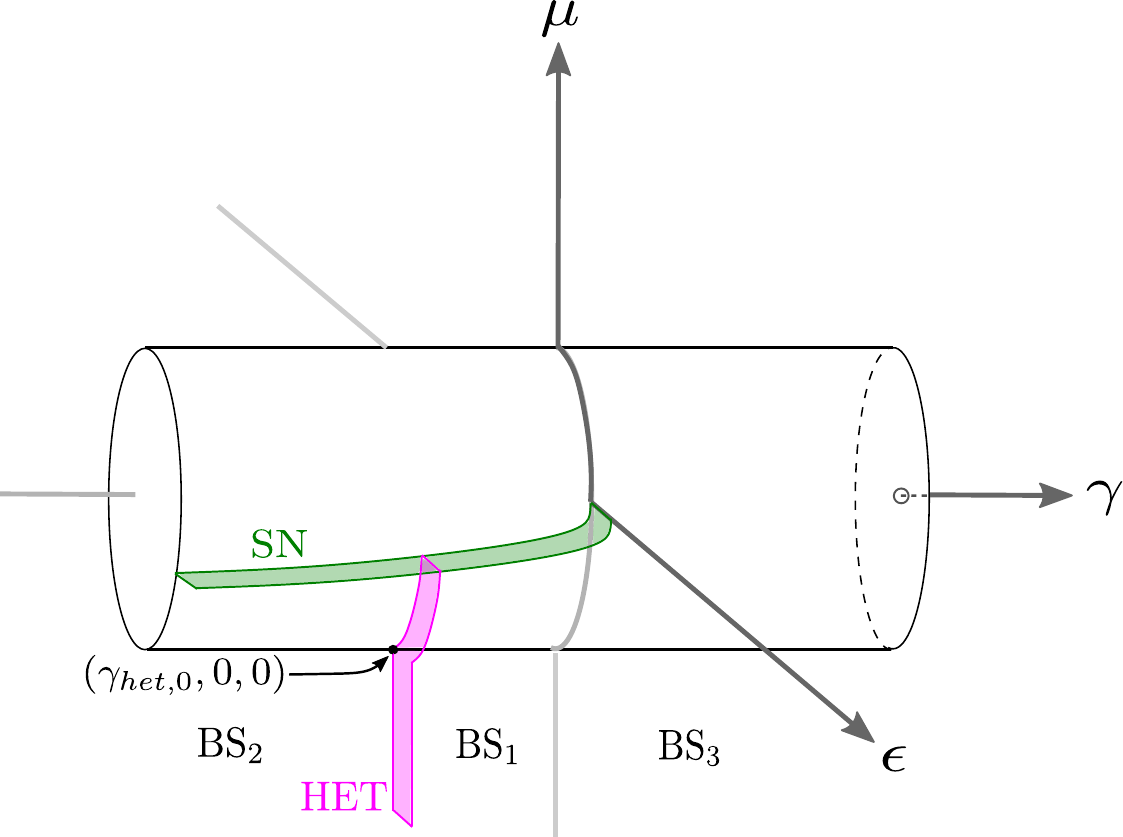}}
	\caption{Bifurcations and separatrices in $(\gamma,\mu,\epsilon)-$space. Cylindrical blow-up along $\mu = \epsilon = 0$ via \eqref{eq:bu_par} allows for the representation of both scaling regimes $\mu = \mathcal O(\epsilon^{k/(1+k)})$ and $\mu = \mathcal O(1)$ in a single space. Corresponding bifurcation diagrams from Figure \ref{fig:bifs} appear on the blow-up cylinder. (a) Global bifurcation diagram for singularly perturbed BF$_i$ bifurcations with $i=1,2,3$. The homoclinic branch in magenta forms a boundary between singularly perturbed BF$_1$ and BF$_2$ bifurcations. Proposition \ref{prop:BF} describes \SJ{the inner and outer asymptotics of the homoclinic branch for $\mu > 0$ and $\hat \mu \gg 1$ respectively, within non-overlapping wedges shown in blue and orange about the point $(\gamma_{hom,0},0,0)$ given by \eqref{eq:gamma_hom}.} 
	(b) \SJ{Expected} global bifurcation diagram for singularly perturbed BS bifurcations, with the distinguished heteroclinic branch forming a boundary between singularly perturbed BS$_1$ and BS$_2$ bifurcations; \SJ{see} \SJ{Remark \ref{rem:bs12}.} 
	}
	\label{fig:separatrices}
\end{figure}

\KUK{A proof is given in Section \ref{ssec:proof_of_prop_bf}. Proposition \ref{prop:BF} asserts the persistence of the PWS homoclinics in an outer regime and an inner regime. \SJJ{This} constitutes a (partial) boundary between singularly perturbed BF$_1$ and BF$_2$ unfoldings for $0 < \epsilon \ll 1$.}

\WM{
\begin{remark}
 Notice that the outer regime covers $\mu \sim \epsilon$ whereas the inner regime 
 covers $\mu \sim \epsilon^{k/(1+k)}$. The two regimes do not overlap for $\epsilon\rightarrow 0$. In principle, we should be able to cover the gap using our method, but we leave that for future work.
\end{remark}
}

Combining Theorem \ref{thm:bifs}(iv) and Proposition \ref{prop:BF}, we have asymptotic information about the branch of homoclinic solutions in Figure \ref{fig:bifs}(a) for $\hat \mu \sim \hat \mu_{bf}$, $\hat \mu \gg 1$ and $\mu \geq 0$. Our findings are represented schematically in Figure \ref{fig:separatrices}(a), which shows the \SJ{expected} global bifurcation diagram in $(\gamma,\mu,\epsilon)-$space after \WM{a weighted} cylindrical blow-up
\begin{equation}
\label{eq:bu_par}
\eta \geq 0, \ \left(\tilde \epsilon, \tilde \mu \right) \in S^1 \mapsto
\begin{cases}
\epsilon = \eta \tilde \epsilon , \\
\mu = \eta^{k/(1+k)} \tilde \mu ,
\end{cases}
\end{equation}
which replaces the degenerate line $\{(\gamma,0,0):\gamma \in \mathbb R\}$ corresponding to the BE singularity by the cylinder $\{\eta = 0\} \times \mathbb R \times S^1$. After desingularization in the family rescaling chart $\tilde \epsilon = 1$, the bifurcation diagram in Figure \ref{fig:bifs}(a) is identified on the cylinder, i.e.~within $\{\eta = 0\}$, which is invariant. The bifurcation diagram for $\mu > 0$ is bounded above the cylinder in Figure \ref{fig:separatrices}(a).





\begin{remark}
	\label{rem:bs12}
\SJ{
The BS$_{1,2}$ boundary is formed by the distinguished heteroclinic connection which connects saddle and node equilibria, tangentially to the strong eigendirection of the node. In the PWS normal form \eqref{eq:PWS_normal_form} obtained in the} \KUK{in the dual limit} \SJ{$\epsilon \to 0^+$, $\mu\rightarrow 0^-$,
%
this distinguished heteroclinic connection occurs for
	\[
	\gamma_{het,0} = \frac{\tau - \sqrt{-\Delta}}{2}.
	\]}
%
\KUK{It is straightforward to obtain an analagous result to Proposition \ref{prop:BF}, describing inner and outer expansions of such a heteroclinic connection, see \SJ{the} illustration in Figure \ref{fig:separatrices}(b). For simplicity, we have decided against including this result. Furthermore, numerical investigations (see Figure \ref{fig:bifs}(d)) support the existence of a simple (transverse) connection to the branch of saddle-node bifurcations with base along $\{(\gamma_{het}(0), \mu, 0) : \hat \mu < \hat \mu_{sn}(\gamma_{het}(0)) \}$ as shown in Figure \ref{fig:separatrices}(b).}
\end{remark}

\subsection{Explosion in case BN$_3$}
\label{ssec:results_bn3}

The case BN$_3$ in Figure \ref{fig:bn3} is somewhat special, due to the existence of a continuous family of PWS homoclinic cycles for $\mu = \epsilon = 0$. As indicated in Figure \ref{fig:pws_cycs}, we parametrize this family using the negative $x-$coordinate: 
\begin{equation}
\label{eq:PWS_cycs}
\Gamma(s) = \Gamma_{X^+}(s) \cup \Gamma_{sl}(s),
\end{equation}
for any $s\in (0,s_0)$ with $s_0>0$ sufficiently small, 
where $\Gamma_{X^+}(s)$ is the backward orbit of $(-s,0)$ following $X^+\vert_{\mu=0}$ while $\Gamma_{sl}(s)$ is the forward orbit of $(-s,0)$ following $X_{sl}\vert_{\mu=0}$. 
Note that the orbits $\Gamma(s)$ are homoclinic to a BN$_3$ singularity, and should not therefore be confused with homoclinic orbits $\Gamma_{hom}$ from Proposition \ref{prop:BF} that are homoclinic to a hyperbolic sliding equilibrium on $\Sigma$.

\begin{figure}[t!]
	\centering
	\includegraphics[scale=.6]{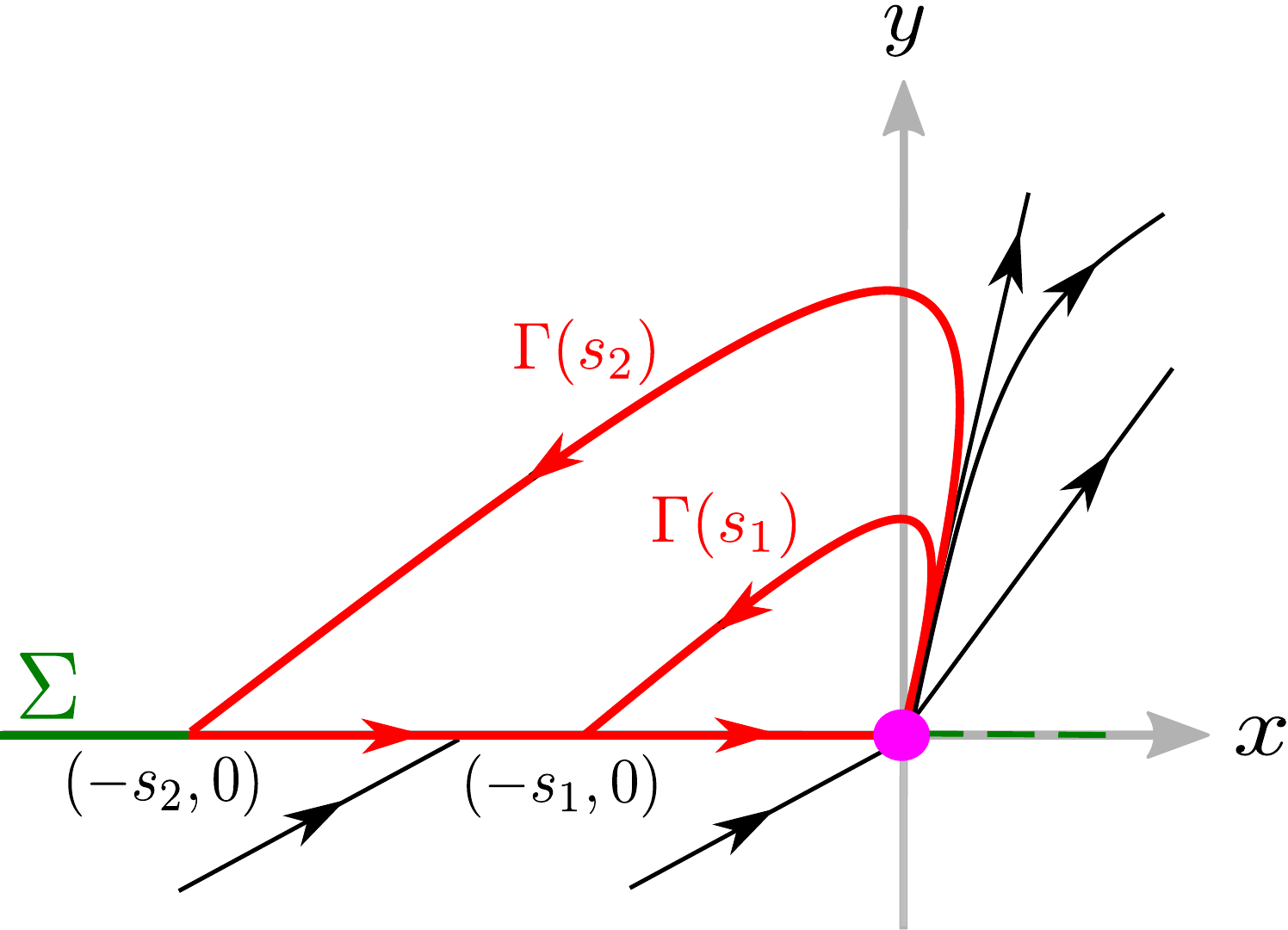}
	\caption{Representative PWS homoclinic orbits $\Gamma(s_1)$ and $\Gamma(s_2)$ defined by \eqref{eq:PWS_cycs}, with $s_2 > s_1$.}
	\label{fig:pws_cycs}
\end{figure}




Since $\Gamma(s)$ only exists for parameter values $\mu = \epsilon = 0$ corresponding to a BE singularity, we are motivated to consider the problem within the blown-up space described in Section \ref{ssec:results_blow-up}. Recall that on the sphere $\nu=0$ there exists an attracting two-dimensional center manifold $\mathcal J$ of a partially hyperbolic point $q_a$. Essentially, this manifold provides an extension of the critical manifold onto the sphere $\nu=0$. At the same time, for the present case BN$_3$, there is also a hyperbolic point $q_w$ on the sphere $\nu=0$, along $\check \epsilon=\check \mu=0$ with a two-dimensional stable manifold $\mathcal S : = W^s(q_w)$, see Figure \ref{fig:bu2}(c). Let $\mathcal J_{\hat \mu}$ and $\mathcal S_{\hat \mu}$ denote the manifolds obtained from $\mathcal J$ and $\mathcal S$ after restriction to the invariant subsets $\{\hat \mu = const.\}$ defined via the scaling \eqref{eq:scaling}.

The following result identifies the existence of a heteroclinic connection between $q_a$ and $q_w$ which will play an important role in the unfolding of the PWS cycles. The situation is sketched in Figure \ref{fig:bn_het_sphere}.

\begin{lemma}
	\label{lem:het}
	For each fixed $\gamma<0$, $\mathcal J$ and $\mathcal S$ intersect in a unique heteroclinic orbit connecting $q_a$ with $q_w$. 
	%
	%
\end{lemma}

\KUK{A similar result was proven in \cite[Proposition 2]{Kristiansen2019d} in the context of the substrate-depletion oscillator, which is degenerate as a BN$_3$ bifurcation (see Section \ref{ssec:outlook_canard_explosion} for further details). Nevertheless, the proof of 
	Lemma \eqref{lem:het}, which will be given in Section \ref{sec:thm_connection_proof}, in the course of proving Theorem \ref{thm:bn3} below, will follow the proof of  \cite[Proposition 2]{Kristiansen2019d}.} Using the parameter $\hat \mu$ defined in \eqref{eq:scaling}, the heteroclinic will be obtained for a unique value $\hat \mu=\hat \mu_{het}(\gamma)$ \SJ{corresponding to an intersection of manifolds $\mathcal S_{\hat \mu}$ and $\mathcal J_{\hat \mu}$ obtained as intersections of $\mathcal S$ and $\mathcal J$ with invariant level sets defined by \eqref{eq:scaling}.} The existence of a heteroclinic connection produces a family of heteroclinic cycles $\{\bar \Gamma(s)\}_{s \in (0,s_0)}$ with improved hyperbolicity properties, see Figure \ref{fig:bn_het_cycs}. In turn, this enables a perturbation of the PWS homoclinic \eqref{eq:PWS_cycs} into limit cycles for $0<\epsilon\ll 1$. 



\begin{figure}[t!]
	\centering
	\includegraphics[scale=2.2]{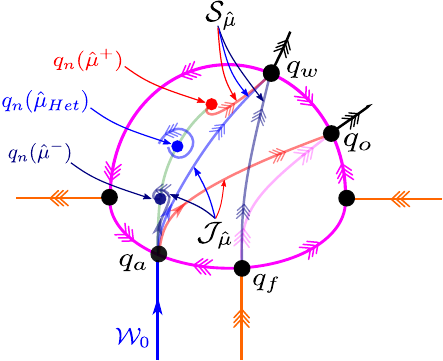}
	\caption{Dynamics on the the blow-up sphere in cases $\hat \mu = \hat \mu^-$, $\hat \mu = \hat \mu_{het}$ and $\hat \mu =\hat \mu^+$, where $\hat \mu^- < \hat \mu_{het} < \hat \mu^+$. \SJ{A 3-dimensional representation is possible after restricting to invariant subspaces defined by level sets \eqref{eq:scaling}. Part of the} path followed by the equilibrium $q_n(\hat \mu)$ under $\hat \mu-$variation is shown in green. \SJ{$\mathcal J_{\hat \mu}$ and $\mathcal S_{\hat \mu}$ denote manifolds obtained from $\mathcal J$ and $\mathcal S$ after restriction to $\{\hat \mu = const.\}$ via \eqref{eq:scaling}.} By Lemma \ref{lem:het}, $\mathcal S_{\hat \mu}$ and $\mathcal J_{\hat \mu}$ intersect for a unique parameter value $\hat \mu = \hat \mu_{het}$, providing a heteroclinic connection from $q_a$ to $q_w$, shown here in blue. This connection breaks regularly as $\hat \mu$ is varied over $\hat \mu_{het}$. Dynamics on each side of the connection are also shown, in dark blue and red.}
	\label{fig:bn_het_sphere}
\end{figure}

\begin{figure}[t!]
	\centering
	\includegraphics[scale=1]{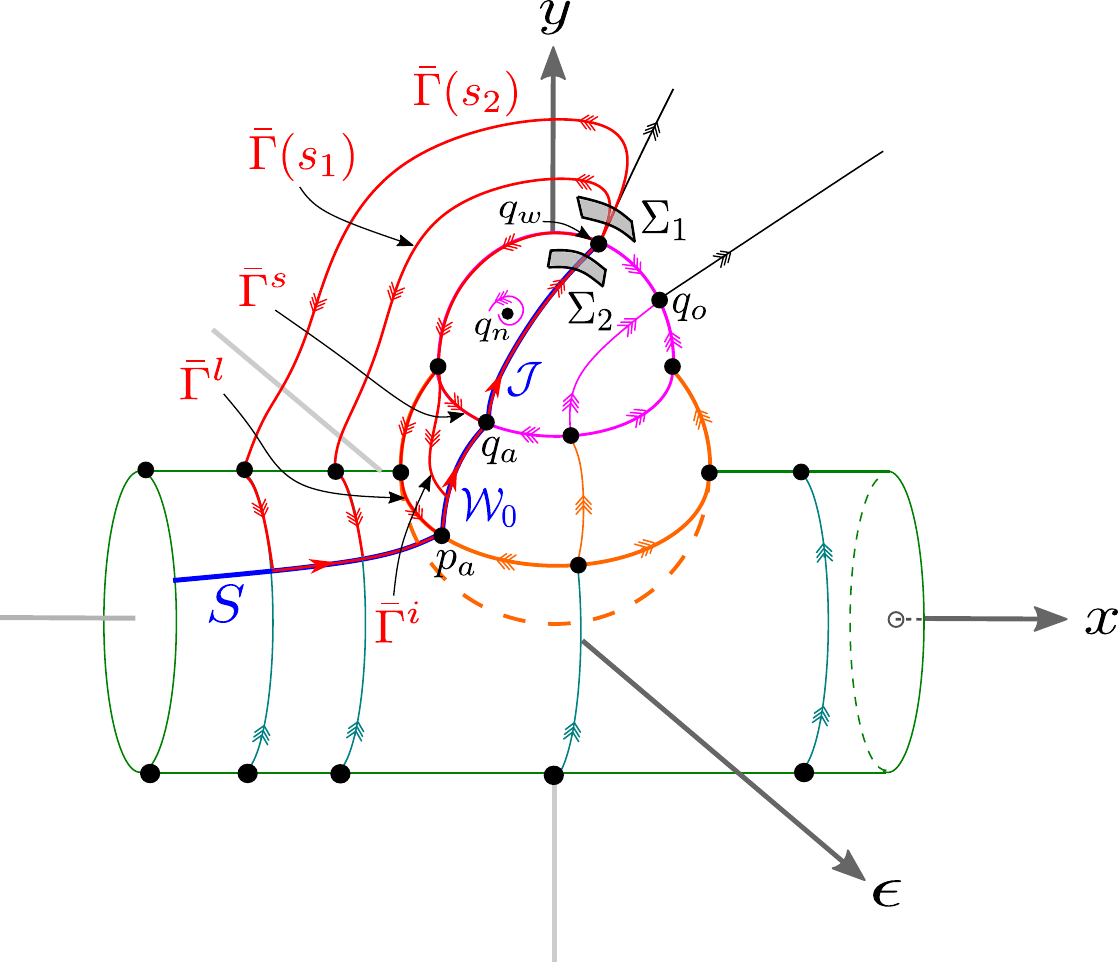}
	\caption{Nondegenerate singular cycles obtained when $\hat \mu = \hat \mu_{het}$ by concatenating orbit segments, after the resolution of all degeneracy via the sequence blow-up transformations described in Section \ref{ssec:results_blow-up}. The cycles $\bar \Gamma(s_1)$ and $\bar \Gamma(s_2)$ shown in red correspond to the PWS cycles $\Gamma(s_1)$ and $\Gamma(s_2)$ in Figure \ref{fig:pws_cycs}, respectively. In terms of the dynamics after blow-up, Theorem \ref{thm:bn3} describes the existence and  growth of limit cycles obtained as perturbations of singular cycles $\bar \Gamma(s)$ with $s>0$, i.e.~with orbit segments bounded away from the blow-ups spheres. Perturbations of singular cycles $\bar \Gamma^s$, $\bar \Gamma^l$, and the family of cycles bounded between (represented here by $\bar \Gamma^i$), are not described by Theorem \ref{thm:bn3}, \SJ{see Remark \ref{rem:connecting_cycs}}. It is \SJ{possible to show as in} \cite[Theorems 3.11 and E.1]{Jelbart2020d}, however, that these cycles mediate a connection to limit cycles on the (magenta) blow-up sphere. Transversal sections \SJ{$\Sigma_1$ and $\Sigma_2$} used in the proof of Theorem \ref{thm:bn3} are also shown.}
	\label{fig:bn_het_cycs}
\end{figure}


\begin{thm}
	\label{thm:bn3}
	Consider system \eqref{eq:normal_form}. Let 
	\begin{align*}
	\lambda:= \frac{2\sqrt{\Delta}}{\tau-\sqrt{\Delta}},
	\end{align*}	
	\KUK{and fix any $\nu\in (0,1)$. }
	Then for any $c>0$ sufficiently small, there exists an $\epsilon_0>0$ and an $s_0>0$ such that the following holds for each $\epsilon \in (0,\epsilon_0)$: There exists a parameterized family of stable limit cycles
	\begin{equation}
	s \mapsto \left(\mu(s,\epsilon), \Gamma(s,\epsilon)\right) , \qquad s \in (c,s_0) ,
	\label{hallo}
	\end{equation}
	which is continuous in $(s,\epsilon)$. In particular, 
	$\lim_{\epsilon\to 0} \Gamma(s,\epsilon) = \Gamma(s)$
	in Hausdorff distance, and \KUK{
		\[
		\mu(s,\epsilon) = \epsilon^{k/(1+k)}\hat \mu_{het}+o(\epsilon^{k/(1+k)}),
		\]
		being $C^1$ in $s\in (c,s_0)$ for each $\epsilon\in [0,\epsilon_0)$ with
		\begin{equation}
		\label{eq:growth_rate}
		\frac{\partial \mu}{\partial s}(s, \epsilon) = \mathcal O(\epsilon^{\nu k(1+\lambda)/(1+k)}).
		\end{equation}
	}	
	%
\end{thm} 

A proof is given in Section \ref{sec:thm_connection_proof}. The limit cycles described in this theorem are $\mathcal O(1)$ with respect to $\epsilon$. Although it is straightforward to use our method to connect these cycles with $o(1)$ cycles (essentially taking $c=K \epsilon^{k/(1+k)}$ in (\ref{hallo}) with $K>0$ sufficiently large, \KUKK{see also Remark \ref{rem:connecting_cycs}}) that are obtained as perturbations of the heteroclinic cycle on the sphere $\{\nu=0\}$, we have decided to focus on the $\mathcal O(1)$ cycles since (i) the result is easier to state, and (ii) we have not been able to connect the cycles \WM{all the way down to the Hopf bifurcation anyways (colloquial)}, recall Theorem \ref{thm:bifs}. Such a connection requires global information of the limit cycles on the sphere, which we have not been able to obtain.

\begin{remark}
	\label{rem:connecting_cycs}
	Figure \ref{fig:bn_het_cycs} also indicates the existence of a family of nondegenerate singular cycles $\bar \Gamma^i$, bounded between `small' and `large' heteroclinic cycles $\bar \Gamma^s$ and $\bar \Gamma^l$ respectively. Such a construction is straightforward, and similar to the construction of singular cycles given in \cite[Section D.1]{Jelbart2020d}. By analogy to the arguments presented in \cite{Jelbart2020d} in the case of singularly perturbed BF$_3$ bifurcation, \SJ{it is possible to prove a} connection between limit cycles that are $\mathcal O(1)$ with respect to $\epsilon$ and limit cycles on the (second) blow-up sphere, facilitated by a family of limit cycles obtained \SJ{as perturbations of} the singular cycles $\bar \Gamma^i$.
\end{remark}

\section{Proof of the theorem \ref{thm:bifs} and proposition \ref{prop:BF}}
\label{sec:proofs}

In this section we prove Theorem \ref{thm:bifs} \SJ{and Proposition \ref{prop:BF}}. We begin with a proof of Theorem \ref{thm:bifs}.

\subsection{Proof of the Theorem \ref{thm:bifs}}
\label{ssec:proof_of_the_bifurcation_results}

We proceed by studying the dynamics of the relevant desingularized system from Lemma \ref{lem:desing_prob}. Theorems \ref{thm:bifs} will follow immediately after lifting system \eqref{eq:desing_prob} out of the invariant plane $\{\nu_1 = 0\}$ into $\{\nu_1 \in [0,\sigma)\}$ for sufficiently small $\sigma>0$ and applying the blow-down transformation given by the inverse to \eqref{eq:coord_change} defined on $\{\epsilon > 0\} = \{\nu_1 > 0, \rho_1>0\}$.

\

System \eqref{eq:desing_prob} has been studied in detail in \cite{Jelbart2020d} in the context of singularly perturbed BF$_i$, $i=1,2,3$ bifurcations, and we shall refer to this work for many of the computations. It is shown in this work that system \eqref{eq:desing_prob} has either 0, 1 or 2 equilibria in $\{\rho_1>0\}$ determined by solutions to the equation
\begin{equation}
\label{eq:desing_eq}
\varphi(\rho_1) = \gamma \beta - \hat \mu \rho_1^{k^2} + \delta \rho_1^{k(1+k)} =0 , \qquad \rho_1 > 0, \ \hat \mu \in \mathbb R .
\end{equation}
For an equilibrium $p_\ast = (x_{1,\ast},\rho_{1,\ast}) \in \{\rho_1 > 0\}$, the Jacobian has trace
\[
\tr J (p_\ast) = - k\beta + \tau \rho_{1,\ast}^{k(1+k)} ,
\]
and determinant
\[
\det J(p_\ast) = -\rho_{1,\ast}^{k(1+2k)} \left(k \hat \mu - \delta (1+k) \rho_{1,\ast}^k \right) .
\]
These expressions can be used to show the existence of saddle-node and Andronov-Hopf bifurcations along the parameterized curves defined by \eqref{eq:sn_desing} and \eqref{eq:ah_desing} respectively; see \cite[pp.41-42]{Jelbart2020d}. In particular, the Andronov-Hopf bifurcation is shown to have first Lyapunov coefficient
\[
l_1 = - \frac{\beta  k^3 (1+k)}{16 (\delta - \gamma  \tau )}
 \left(\frac{\beta  k}{\tau }\right)^{-2/k(1+k)} ((2+k) \delta -\gamma  \tau ) ,
\]
using \WM{the software package Mathematica; compare with} \cite[eqn.~8.35]{Kuehn2015}. 
This implies a supercritical bifurcation for all $k \in \mathbb N_+$, since by \eqref{eq:ah_desing} we have $\gamma < \delta / \tau$ with $\delta, \tau > 0$ and therefore
\[
(2+k) \delta - \gamma \tau > \delta - \gamma \tau = \tau \left(\frac{\delta}{\tau} - \gamma \right) > 0 \quad   \implies \quad l_1 < 0 .
\]

Saddle-node and Andronov-Hopf curves continuously extend to an intersection
\begin{equation}
\label{eq:bt_desing}
\left( \hat \mu_{bt}, \gamma_{bt} \right) = \left( \frac{(1+k) \delta}{k} \left(\frac{k \beta}{\tau}\right)^{1/(1+k)} , \frac{\delta}{\tau} \right) ,
\end{equation}
corresponding to Bogdanov-Takens bifurcation in system \eqref{eq:desing_prob}. \SJJ{In particular, if we let $X_1(x_1,\rho_1,\hat \mu, \gamma)$ represent the right-hand-side in \eqref{eq:desing_prob} then one can show regularity of the map
\[
\left((x_1,\rho_1) , (\hat \mu, \gamma) \right) \mapsto 
\left(X_1(x_1,\rho_1,\hat \mu, \gamma) , \tr J (x_1, \rho_1 , \hat \mu, \gamma) , \det J (x_1, \rho_1 , \hat \mu, \gamma) \right)
\]
at the Bogdanov-Takens point by a direct calculation. The additional nondegeneracy conditions
\[
a_{20}(0) + b_{11}(0) \neq 0 , \qquad b_{20}(0) \neq 0 ,
\]
on coefficients $a_{20}(0), b_{11}(0)$ and $b_{20}(0)$ defined in \cite[Theorem 8.4]{Kuznetsov2013} are shown using the expressions in the cited work to be satisfied with
\[
a_{20}(0) + b_{11}(0) = -\frac{\beta ^2 \delta  k^3 (k+1) \left(\frac{\beta  k}{\tau
	}\right)^{-\frac{1}{k^2+k}}}{2 \tau ^2} ,
\qquad
b_{20}(0) = \beta  k^2 (k+1) \left(\frac{\beta  k}{\tau }\right)^{-\frac{1}{k^2+k}} ,
\]
both of which are nonzero within the parameter regime of interest.}

\SJJ{Finally,} standard results in bifurcation theory imply the existence of a neighbourhood $I_{hom} \ni \gamma_{bt}$ and smooth function $\hat \mu_{hom}:I_{hom} \to \mathbb R$ such that
\begin{equation}
\label{eq:hom_conds}
\hat \mu_{hom}(\gamma_{bt}) = \hat \mu_{bt} , \qquad
\hat \mu_{hom}'(\gamma_{bt}) = \hat \mu_{sn}'(\gamma_{bt}) = \hat \mu_{ah}'(\gamma_{bt}) ,
\end{equation}
and $\hat \mu_{hom}''(\gamma_{bt})$, $\hat \mu_{sn}''(\gamma_{bt})$ and $\hat \mu_{ah}''(\gamma_{bt})$ are all distinct \cite{Kuznetsov2013}. 
The local parameterisation in \eqref{eq:hom_desing} follows from \eqref{eq:hom_conds} after Taylor expansion about $\gamma = \gamma_{bt}$. In order to see that \SJJ{saddle-homoclinic} bifurcation cannot occur for $\gamma < 0$, we first observe the following:
\begin{itemize}
	\item For $\gamma < 0$, system \eqref{eq:desing_prob} has a single equilibrium within $\{\rho_1 > 0\}$, and two equilibria $\{(-\beta,0),(0,0)\} \in \{\rho_1 = 0\}$;
	\item The subspace $\{\rho_1 = 0\}$ is invariant.
\end{itemize}
It follows that a homoclinic orbit cannot exist, since the connecting orbit cannot enclose an equilibrium.

\

Lifting the expressions derived above for $\nu_1 \in [0,\sigma)$ with $\sigma > 0$ sufficiently small 
and applying the blow-down transformation, in particular the relation
\[
\hat \mu = \mu \epsilon^{-k/(1+k)} ,
\]
we obtain the desired result. \qed

\begin{remark}
	In the preceding proof $\sigma$ must be sufficiently small so that \eqref{eq:sn_desing}, \eqref{eq:ah_desing}, \eqref{eq:hom_desing} and \eqref{eq:bt_desing} can be extended in $(x_1,\rho_1,\nu_1,\mu_1)-$space via suitable applications of the implicit function theorem. We omit this argument -- which is standard -- for the sake of brevity, but refer the reader to \cite[eqn.~(D7)]{Jelbart2020d} where the extended system is considered in detail.
\end{remark}

\subsection{Proof of Proposition \ref{prop:BF}}
\label{ssec:proof_of_prop_bf}

The result for the \SJ{outer} regime with $\mu>0$ and $\epsilon\rightarrow 0$ is standard, using the established correspondence between the Filippov system and the regularization \cite{Bonet2016,Buzzi2006,Kristiansen2019c,Kristiansen2019,Llibre2009,Llibre2007}, once we introduce the scalings defined by \SJJ{$x=\mu \widehat X$ and $y=\mu \widehat Y, \epsilon  = \mu E$}. Indeed, we just perform the cylindrical blowup \SJJ{$(\widehat X,\widehat Y,E)=(\widehat X,0,0)$} for the extended system \SJJ{$\{(\widehat X, \widehat Y)' = E X(\mu \widehat X,\mu \widehat Y,\mu, \mu E),E'=0\}$}. 

We therefore focus on the inner expansion in the dual limit case, setting $\epsilon = \mu^{(1+k)/k}\hat \epsilon$ and letting $\mu \rightarrow 0$. 
For this we consider \SJ{system} \eqref{eq:xyXYpmeqs}. 
The case of $k=1$ is easier, so we will also focus on this case, repeated here for convenience
\begin{equation}\label{eq:xyXYpeps1}
\begin{aligned}
X' &=(1+\tau X -\delta Y) Y-(\gamma-\tau)\beta \hat \epsilon ,\\
Y'  &= X Y +\beta \hat \epsilon,
\end{aligned}
\end{equation}
\KUK{\SJJ{where we have dropped the} hat \SJJ{notation on $X$ and $Y$}}. We leave the discussion of the general case $k\in \mathbb N$ to the end of the section.

The system \eqref{eq:xyXYpeps1} is for $0<\hat \epsilon\ll 1$ a slow-fast system in nonstandard form \cite{Jelbart2020a,Wechselberger2019}. Indeed, for $\hat \epsilon=0$ we obtain the layer problem
\begin{equation}\label{eq:xyXYpeps1lp}
\begin{aligned}
X' &=(1+\tau X -\delta Y) Y,\\
Y'  &= X Y
\end{aligned}
\end{equation}
for which $\{Y=0\}$ is a manifold of equilibria. Linearization of any point $(X,0)$ gives $X$ as the only nonzero eigenvalue. Consequently, $S_a:=\{(X,0)\,:\, X<0\}$ is normally hyperbolic and attracting, $(0,0)$ is fully nonhyperbolic, \SJJ{and} $S_r:=\{(X,0)\,:\,X>0\}$ is normally hyperbolic and repelling. Notice also \SJJ{that} for $Y>0$ we obtain \SJJ{the equivalent system}
\begin{equation}\label{eq:XYplus}
\begin{aligned}
X' &= 1+\tau X-\delta Y,\\
Y' &=X,
\end{aligned}
\end{equation}
upon dividing the right hand side \SJJ{of \eqref{eq:xyXYpeps1lp}} by $Y$. Let $\phi_t$ denote the flow of \eqref{eq:XYplus}. We then define $\Gamma$ as $\{\phi_t(0,0)\}_{t\in (0,t_d]}$ where $t_d>0$ is the first return time to $Y=0$. Notice that $\Gamma$ is well-defined since \eqref{eq:XYplus} is just the linearization of \SJJ{the vector field} $X^+$ having, in the BF case considered, an unstable focus at $(0,\delta^{-1})$. 
It is a simple calculation to show that $t_d>0$ is the first positive root of $R(t)$, recall \eqref{eq:Rt}, and that $\Gamma\cap \{Y=0\}=(X_d,0)$ with 
\begin{align}
\label{eq:drop}
X_d = -\frac{2e^{\tau t_2/2} }{\sqrt{- \Delta}} \sin \left(\frac{\sqrt{-\Delta}}{2} t_d \right).
\end{align}

Next, setting $Y=\hat \epsilon Y_2$ brings \eqref{eq:xyXYpeps1} into a slow-fast system in standard form. Upon passing to a slow time and then setting $\hat \epsilon=0$, \SJJ{we obtain the following} reduced problem on $S_a$:
\begin{align}
\dot X &= -\beta X^{-1} (1+\gamma X),\label{eq:reducedprop35}
\end{align}
\SJJ{which has} a repelling equilibrium at $X=-\gamma^{-1}$, seeing that $\gamma>0$. \SJJ{We note that} reduced problem can also be obtained from more general procedures described in \cite{Jelbart2020a,Wechselberger2019}. 

\begin{figure}[t!]
	\centering
	\includegraphics[scale=0.8]{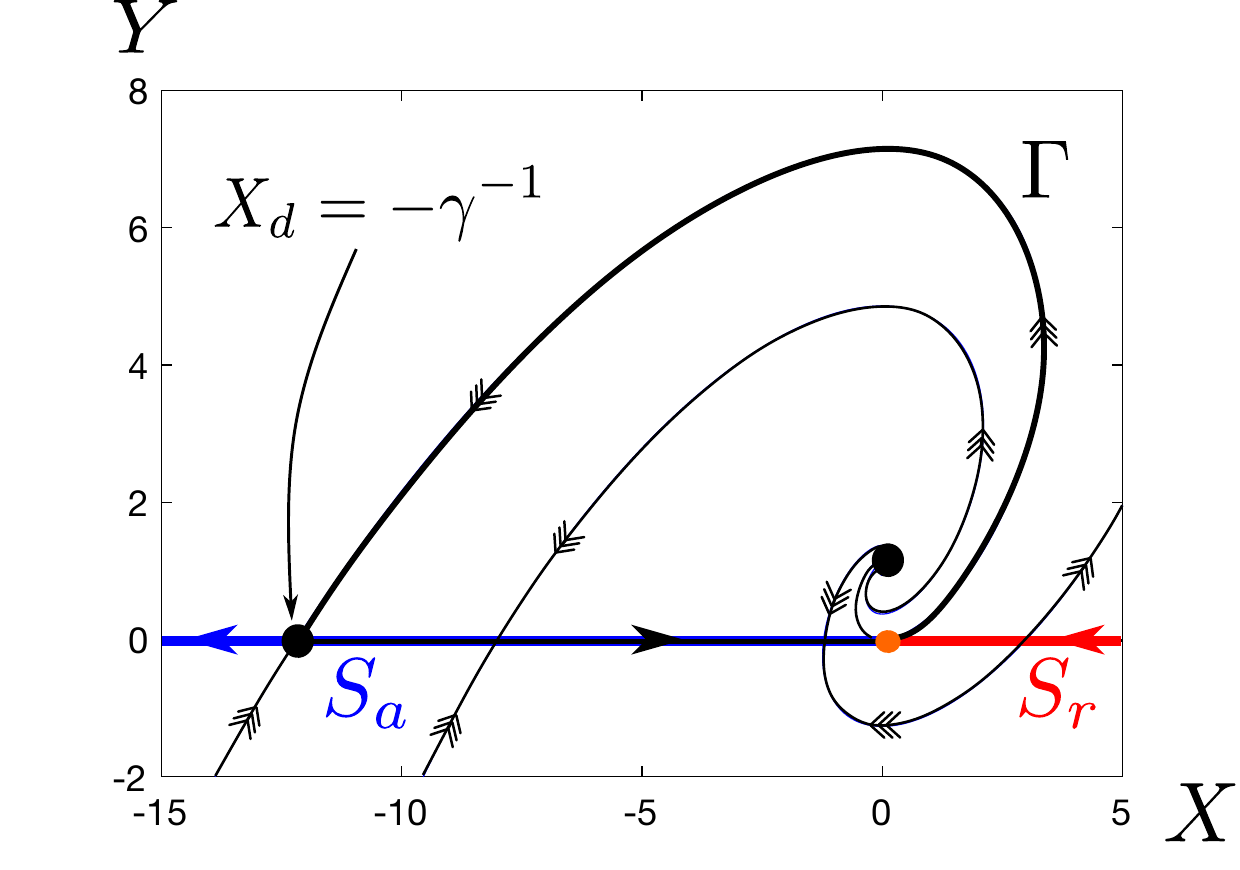}
	\caption{Singular limit dynamics for the nonstandard form slow-fast system \eqref{eq:xyXYpeps1} arising in case $k=1$. Attracting and repelling critical manifolds $S_a$ and $S_r$ are shown in blue and red respectively. The point $(0,0)$, shown in orange, is a regular fold point. There is an unstable focus at $(0,\delta^{-1})$, and an equilibrium $(-\gamma^{-1},0)$ which is repelling as an equilibrium for the reduced flow on $S_a$; both are indicated as black disks. We show the situation where $\gamma = \gamma_{hom,0} = -X_d^{-1}$ with $X_d$ given by \eqref{eq:drop}, for which there is a singular homoclinic orbit $\Gamma$ (shown here in black).}
	\label{fig:two_stroke}
\end{figure}

Combining our analysis of the layer problem and the reduced problem, we obtain Figure \ref{fig:two_stroke}. Specifically, for $\gamma=\gamma_{hom,0}:=-X_d^{-1}$ we have a singular \SJJ{saddle-homoclinic} connection. At the singular level $\hat \epsilon=0$, the connection is clearly transverse with respect to $\gamma$; in fact, $\Gamma$ is independent of $\gamma$ so this is obvious from $X_d=X_d(\gamma)$. For $0<\hat \epsilon\ll 1$ we then use Fenichel theory to perturb the saddle and the result of \cite{Krupa2001a} to track its unstable manifold near $\Gamma$. Defining a section $\Sigma$ transverse to $\Gamma$ within $Y>0$, we then obtain a bifurcation equation for the \SJJ{saddle-homoclinic} connection of the form $H(\gamma,\hat \epsilon)=0$, with $H$, \SJJ{which measures} the separation of the stable and unstable manifolds on $\Sigma$, being at least $C^1$ in $\gamma$, \SJJ{continuously dependent} on $\hat \epsilon\in [0,\hat \epsilon_0)$ \SJJ{and such that}
\begin{align*}
H(\gamma_{hom,0},0)=0,\,H'_{\gamma}(\gamma_{hom,0},0)\ne 0.
\end{align*}
The existence of $\gamma_{hom}^{inner}$ in Proposition \ref{prop:BF} follows \SJJ{after} applying the implicit function theorem to $H(\gamma,\hat \epsilon)=0$ at $(\gamma,\hat \epsilon) = (\gamma_{hom,0},0)$. For the final part of Proposition \ref{prop:BF}, we fix $\hat \epsilon$ small enough (i.e.~$\hat \mu$ large enough) and perturb in $\mu>0$ (or equivalently $\epsilon>0$, having fixed $\hat \epsilon$) small enough. 
\begin{remark}\label{rem:here}
	Notice for this last part that the $\mu$-perturbation of \eqref{eq:xyXYpeps1} will include terms of the form \[\phi_+(\epsilon y^{-1})=\phi_+(\mu^{1/k} \hat \epsilon Y^{-1}),\] using \eqref{eq:xyXYpm}, which \SJJ{are} ill-defined for $\mu=Y=0$. This is in the sense of which the charts \eqref{eq:xyXYpm} are more ill-suited for global computations. Here, however, fixing $\hat \epsilon>0$, where the saddle connection occurs within $Y>0$, we just require that the perturbation is continuous with respect to $\mu$ on this domain. To cover the \SJJ{saddle-homoclinic} case in a full neighbourhood of $(\epsilon,\mu)=(0,0)$, we have to work with our full blowup system, tracking the saddle across the first blowup sphere. We leave the details of this to future work.
\end{remark}

For $k\ge 2$, $\{Y=0\}$ is fully nonhyperbolic for $\hat \epsilon=0$. We then gain hyperbolicity by blowing up the points $(X,0,0)$ in the extended $(X,Y,\hat \epsilon)-$space \SJJ{via}
\begin{align*}
r\ge 0,\,(\bar Y,\bar E)\mapsto \begin{cases}
Y &= r \bar Y,\\
\hat \epsilon &= r\bar E,
\end{cases}
\end{align*}
\SJJ{followed by} a desingularization corresponding to division of the right hand side by $r^{k-1}$. 
Working in the directional chart corresponding to $\bar Y=1$ using chart-specified coordinates $(r_1,Y_1,E_1)$ defined by $Y=r_1,\,\hat \epsilon =r_1 E_1$, we find a normally hyperbolic and attracting critical manifold on $r_1=0$, carrying a reduced problem given by \eqref{eq:reducedprop35}. We therefore obtain the result as in the $k=1$ case, performing a separate blowup of $(X,r_1,E_1)=(0,0,0)$, replacing the result of \cite{Krupa2001b}, to track the slow manifold for $Y>0$ in this case. We leave out the details for simplicity. 

\section{Proof of theorem \ref{thm:bn3}}
\label{sec:thm_connection_proof}

We apply the blow-up procedure outlined in Section \ref{ssec:results_blow-up}. 
To describe the blow-up transformation \eqref{eq:comp_final} we focus on the following directional charts $\check \epsilon=1$ and $\check \rho=1$ with the chart-specified coordinates $\nu_1,\rho_1,x_1,\mu_1$ and $\nu_2,x_2,\epsilon_2,\mu_2$ defined by
\begin{align}
&\begin{cases}
x = \nu_1^{2k(1+k)} \rho_1^{k(1+k)} x_1 , \\
y= \nu_1^{2k(1+k)}  \rho_1^{2k(1+k)} , \\
\epsilon = \nu_1^{2(1+k)^2} \rho_1^{(2k+1)(1+k)}, \\
\mu = \nu_1^{2k(1+k)} \mu_1 ,
\end{cases}\label{eq:checkeps1}\\
&\begin{cases}
x = \nu_2^{2k(1+k)} x_2 , \\
y= \nu_2^{2k(1+k)}, \\
\epsilon = \nu_2^{2(1+k)^2} \epsilon_2  , \\
\mu = \nu_2^{2k(1+k)} \mu_2 ,
\end{cases}\label{eq:checkrho1}
\end{align}
respectively. We have the following smooth change of coordinates 
\begin{align}\label{eq:cc}
\nu_2 &=\nu_1 \rho_1,\,x_2 = x_1 \rho_1^{-k(1+k)},\,\mu_2 = \mu_1 \rho_1^{-2k(1+k)},\,\epsilon_2= \rho_1^{-(1+k)},
\end{align}
for $\rho_1>0$. In these charts, we obtain the desingularization by division of the right hand side by $\nu_1^{2(1+k)^2} \rho_1^{(1+k)^2}$ and $\nu_2^{2(1+k)^2}\epsilon_2$, respectively.

In the following lemma we present the desingularized equations in these respective charts. For this we first define $\widehat \theta_1$ and $\widehat \theta_2$ for $z>0$ and $q>0$ as follows:
\begin{align*}
\widehat \theta_1(u,v,w,z):= z^{-1} \theta_1(zu,zv,zw),\qquad
\widehat \theta_2(u,v,w,q,z) := z^{-1} q^{-1}\theta_2(zq u,zqv,zw),
\end{align*}
both having smooth extensions to $z=0$ and $q=0$, cf.~Theorem \ref{theorem:prop_normal_form}. Notice then that
\begin{align*}
\widehat \theta_1(u,v,w,0)=\widehat \theta_1(0,0,0,z) =\widehat \theta_2(u,v,w,q,0)=\theta_2(0,0,0,q,z)=0,
\end{align*}
for all $u,v,w,q,z$. 
\begin{lemma}\label{lem:desing_ext}
	The desingularized equations in the chart $\check \epsilon=1$ take the following form:
	\begin{equation}\label{eq:checkeps1eqs}
	\begin{aligned}
	x_1'&=f_1(x_1,\rho_1,\nu_1,\mu_1) +k x_1 g_1(x_1,\rho_1,\nu_1,\mu_1),\\
	\rho_1' &=\frac{1}{k}\rho_1 g_1(x_1,\rho_1,\nu_1,\mu_1),\\
	\nu_1' &=-\frac{2k+1}{2k(1+k)} \nu_1 g_1(x_1,\rho_1,\nu_1,\mu_1),\\
	\mu_1' &= (2k+1) \mu_1 g_1(x_1,\rho_1,\nu_1,\mu_1),
	\end{aligned}
	\end{equation}
	where
	\begin{align*}
	f_1(x_1,\rho_1,\nu_1,\mu_1) &=\left(\mu_1+\tau \rho_1^{k(1+k)} x_1-\delta \rho_1^{2k(1+k)}+ \widehat \theta_1(\rho_1^{k(1+k)}x_1,\rho_1^{2k(1+k)},\mu_2,\nu_1^{2k(1+k)})\right)\\
	&\times \left(1-\nu_1^{2k(1+k)} \rho_1^{k(1+k)} \phi_+(\nu_1^{2(1+k)} \rho_1^{1+k})\right)- \rho_1^{k(1+k)}\phi_+(\nu_1^{2(1+k)} \rho_1^{1+k})(\gamma-\tau),\\
	g_1(x_1,\rho_1,\nu_1,\mu_1) &=\left(x_1+\widehat \theta_2(x_1,\rho_1^{k(1+k)},\mu_1,\rho_1^{k(1+k)},\nu_1^{2k(1+k)})\right) \\
	&\times \left(1-\nu_1^{2k(1+k)} \rho_1^{k(1+k)}\phi_+(\nu_1^{2(1+k)} \rho_1^{1+k})\right) 
	+\phi_+(\nu_1^{2(1+k)} \rho_1^{1+k}).
	\end{align*}
	The quantity
	\begin{align}
	\hat \mu = \mu_1 \rho_1^{-k(2k+1)},\label{eq:scaling1}
	\end{align}
	is conserved for the flow of \eqref{eq:checkeps1eqs}.

	The desingularized equations in the chart $\check \rho=1$ take the following form:
	\begin{equation}\label{eq:checkrho1eqs}
	\begin{aligned}
	x_2'&=f_2(x_2,\epsilon_2,\nu_2,\mu_2) - x_2 g_2(x_2,\epsilon_2,\nu_2,\mu_2),\\
	\epsilon_2' &=-\frac{1+k}{k}\epsilon_2 g_2(x_2,\epsilon_2,\nu_2,\mu_2),\\
	\nu_2' &=-\frac{1}{2k(1+k)} \nu_2 g_2(x_2,\epsilon_2,\nu_2,\mu_2),\\
	\mu_2' &= -\mu_2 g_2(x_2,\epsilon_2,\nu_2,\mu_2),
	\end{aligned}
	\end{equation}
	where
	\begin{align}
	f_2(x_2,\epsilon_2,\nu_2,\mu_2) &=\left(\mu_2+\tau x_2-\delta + \widehat \theta_1(x_2,1,\mu_2,\nu_2^{2k(1+k)})\right)\nonumber\\
	&\times \left(1-\nu_2^{2k(1+k)} \epsilon_2^{k} \phi_+(\nu_2^{2(1+k)} \epsilon_2)\right)- \epsilon_2^k \phi_+(\nu_2^{2(1+k)} \epsilon_2)(\gamma-\tau),\nonumber\\
	g_2(x_2,\epsilon_2,\nu_2,\mu_2) &=\left(x_2+\widehat \theta_2(x_2,1,\mu_2,\nu_2^{2k(1+k)})\right)\left(1-\nu_2^{2k(1+k)} \epsilon_2^{k}\phi_+(\nu_2^{2(1+k)} \epsilon_2)\right)\nonumber\\
	&+\epsilon_2^k \phi_+(\nu_2^{2(1+k)} \epsilon_2).\label{eq:g2}
	\end{align}
	The quantity
	\begin{align}
	\hat \mu = \mu_2 \epsilon_1^{-k/(k+1)},\label{eq:scaling2}
	\end{align}
	is conserved for the flow of \eqref{eq:checkrho1eqs}.
	
\end{lemma}

\begin{proof}
	\SJJ{This follows by lengthy, but standard calculations. We defer the proof to Appendix \ref{app:bu_eqns} for expository reasons.}
\end{proof}

In the following, we analyze the two charts separately.
\subsection{The dynamics in $\check \epsilon=1$}
First, we notice that on the set defined by $\nu_1=0$, the system \eqref{eq:checkrho1eqs} becomes
\begin{equation}\label{eq:nu10}
\begin{aligned}
x_1' &= \mu_1+\tau \rho_1^{k(1+k)} x_1-\delta \rho_1^{2k(1+k)}- \rho_1^{k(1+k)}\beta(\gamma-\tau) + k x_1 \left(\beta + x_1 \right) , \\
\rho_1' &= \frac{1}{k} \rho_1 \left(\beta + x_1 \right),\\
\mu_1' &=(2k+1)\mu_1 \left(\beta + x_1 \right).
\end{aligned}
\end{equation}
using \SJJ{$\phi_+(0) = \beta$ and}
\begin{align*}
f_1(x_1,\rho_1,0,\mu_1) &=\mu_1+\tau \rho_1^{k(1+k)} x_1-\delta \rho_1^{2k(1+k)}- \rho_1^{k(1+k)}\beta(\gamma-\tau),\\
g_1(x_1,\rho_1,0,\mu_1) &=x_1+\beta.
\end{align*}
Since $\hat \mu = \mu_1 \rho_1^{-k(2k+1)}$ is conserved in this chart, recall \eqref{eq:scaling1}, we can eliminate $\mu_1$ from \eqref{eq:nu10} and in this way we obtain the $(x_1,\rho_1)$-system in 
\eqref{eq:desing_prob}.

On the other hand, within $\rho_1=\mu_1=0$ we have
\begin{equation}\label{eq:x1nu1system}
\begin{aligned}
x_1'&= kx_1(x_1+\beta),\\
\nu_1' &=-\frac{2k+1}{2k(1+k)}\nu_1 (x_1+\beta),
\end{aligned}
\end{equation}
Here we find the fully hyperbolic equilibrium $q_f$ with $x_1=\nu_1=0$. In particular, a simple calculations shows that \SJJ{within} $\nu_1=0$, $q_f$ is a source. 

On the other hand, for \eqref{eq:x1nu1system} we also find $x_1=-\beta$, $\nu_1\ge 0$ as the critical manifold $\mathcal W_0$, see Figure \ref{fig:bu2}, of partially hyperbolic points. Indeed, the linearization of any point on $\mathcal W_0$ has \SJJ{a} single nonzero eigenvalue \SJ{$-k\beta$}, also at the point $q_a$ with coordinates $(x_1,\rho_1,\nu_1,\mu_1)=(-\beta,0,0,0)\in \mathcal W_0$. At $q_a$, we therefore have a three-dimensional attracting center manifold. \SJJ{We} shall denote the $\nu_1=0$ subset of this manifold by $\mathcal J$, as indicated in Figure \ref{fig:bu2}. Using the parameter $\hat \mu$, we may foliate $\mathcal J$ into invariant subsets $\mathcal J_{\hat \mu}$. For simplicity, we denote the projection of $\mathcal J_{\hat \mu}$ onto the $(x_1,\rho_1)$-subspace by the same symbol. Then $\mathcal J_{\hat \mu}$ becomes an attracting center manifold of the point $(x_1,\rho_1)=(-\beta,0)$, which we for simplicity also denote by $q_a$, for the system \eqref{eq:desing_prob}. A simple calculation shows that it takes the following smooth graph form:
\begin{align}\label{eq:Jhatmugraph}
x_1 = -\beta +\frac{\gamma}{k}\rho_1^{k(1+k)}(1+\mathcal O(\rho_1)),
\end{align}
over $\rho_1\ge 0$ locally near $q_a$. This gives 
\begin{align*}
\rho_1' = \frac{\gamma}{k^2}\rho_1^{k(1+k)+1}(1+\mathcal O(\rho_1)),
\end{align*}
and $\rho_1>0$ is therefore locally increasing on $\mathcal J_{\hat \mu}$. In conclusion, we have the following.
\begin{lemma}
	Consider \eqref{eq:desing_prob}. Then $q_a$ is a nonhyperbolic saddle on $\{\rho_1\ge 0\}$ and the center manifold $\mathcal J_{\hat \mu}$ is unique on this set as the nonhyperbolic unstable manifold of $q_a$ for all $\hat \mu\in \mathbb R$.
\end{lemma}

Finally, we emphasize that on $\nu_1=0$ we also have the family of equilibria parameterized by \SJ{\eqref{eq:desing_eq} with $\rho_1\ge 0$.} 
This is the `boundary-node' $q_n$ in this chart, which we also parametrize using $\hat \mu$, writing $q_n(\hat \mu)$ \SJ{in Figure \ref{fig:bn_het_sphere}}. In particular, 
using \eqref{eq:desing_eq}, $q_n(\hat \mu)$ has coordinates $(x_1,\rho_1) = (-\beta,\rho_{1,n}(\hat \mu))$ with $\rho_{1,n}(\hat \mu)$ being given implicitly by
\begin{align}
\hat \mu = \rho_{1,n}(\hat \mu)^{-k^2} \beta \gamma+\rho_{1,n}(\hat \mu)^k \delta.\label{eq:hatmurho1}
\end{align}
Notice that \SJ{\eqref{eq:hatmurho1}} defines a unique $\rho_{1,n}(\hat \mu)>0$ for each $\hat \mu$ since $\gamma<0$. 

\

We will need the following result in our proof of Lemma \ref{lem:het}.
\begin{lemma}\label{lem:hatmuneg}
	There exists a $\hat \mu^-$ such that the $\omega$-limit set of $\mathcal J_{\hat \mu}$ is $q_n(\hat \mu)$ for all $\hat \mu\le \hat \mu^-$. 
\end{lemma}
\begin{proof}
	The result follows from the center manifold theory, the fact that $q_n(\hat \mu)$ is a stable node for $\hat \mu\ll -1$ and finally that $q_n(\hat \mu)\rightarrow q_a$ for $\hat \mu\rightarrow -\infty$. 
\end{proof}

\subsection{The dynamics in $\check \rho=1$}
We consider \eqref{eq:checkrho1eqs}. Within the invariant set defined by $\nu_2=0$ we have
\begin{align*}
x_2' &= \mu_2+\tau x_2-\delta-\epsilon_2^k\beta (\gamma-\tau) -x_2(x_2+\epsilon_2^k \beta ),\\
\epsilon_2' &= -\frac{1+k}{k} \epsilon_2 (x_2+\epsilon_2^k \beta),\\
 \mu_2' &=-\mu_2 (x_2+\epsilon_2^k \beta),
\end{align*}
using that 
\begin{align*}
f_2(x_2,\epsilon_2,0,\mu_2) &= \mu_2+\tau x_2-\delta-\epsilon_2^k \beta (\gamma-\tau),\\
g_2(x_2,\epsilon_2,0,\mu_2) &=x_2+\epsilon_2^k \beta.
\end{align*}
Specifically, within the invariant set defined by $\epsilon_2=\nu_2=\mu_2 = 0$ we have
\begin{align*}
x_2' &=\tau x_2-\delta - x_2^2,
\end{align*}
producing the two equilibria $q_w$ and $q_o$ with 
\begin{align}\label{eq:qw}
x_{2} = x_{2,w}:=\frac12 \tau - \frac12 \sqrt{\Delta},\quad x_{2} = x_{2,o}:=\frac12 \tau + \frac12 \sqrt{\Delta} ,
\end{align}
respectively. Recall that $\Delta=\tau^2-4\delta>0$. Both points are fully hyperbolic for \eqref{eq:checkrho1eqs}, but within $\nu_2=0$ the point $q_o$, which corresponds to the strong eigendirection, is an attracting node, whereas $q_w$ is a saddle, having a one-dimensional unstable manifold along $\epsilon_2=\mu_2=0$ and a two-dimensional stable manifold $\mathcal S:=W^s(q_w)$. Using the conservation of $\hat \mu = \mu_2\epsilon_2^{-k/(1+k)}$, we foliate $\mathcal S$ into invariant subsets $\mathcal S_{\hat \mu}$ for $\hat \mu\in \mathbb R$ and $\mathcal S_\infty$, corresponding to $\hat \mu\rightarrow \infty$ contained within $\epsilon_2=0$ where
\begin{equation}\label{eq:Sinfeqs}
\begin{aligned}
x_2' &= \mu_2+\tau x_2-\delta -x_2^2,\\
\mu_2' &=-\mu_2 x_2,
\end{aligned}
\end{equation}
\SJ{and $\mathcal S_\infty$ is a stable manifold of $(x_2,\mu_2)=(x_{2,w},0)$.} Here we find $q_{n,\infty}$, corresponding to $q_n(\hat \mu)$ \SJ{when} $\hat \mu\rightarrow \infty$, as $(x_2,\mu_2)=(0,\delta)$, which is a hyperbolic and unstable node. In fact, we have the following.
\begin{lemma}\label{lem:equivalenceXp}
	The system \eqref{eq:Sinfeqs} on $\{\mu_2>0\}$ is smoothly topologically equivalent with 
	\begin{equation}\label{eq:Xpluslin}
	\begin{aligned}
	x' & =\tau x - \delta y, \\
	y' &=x.
	\end{aligned}
	\end{equation}
	on $\{y>-\delta^{-1}\}$.  
\end{lemma}
\begin{proof}
	A simple calculation shows that the diffeomorphism
	\begin{align*}
	(x,y)\mapsto \begin{cases}
	x_2 = (\delta^{-1}+y)^{-1} x,\\
	\mu_2 =(\delta^{-1}+y)^{-1} ,
	\end{cases}
	\end{align*}
	$\{y>-\delta^{-1}\}$,
	brings \eqref{eq:Xpluslin} into \eqref{eq:Sinfeqs}, which completes the proof. 
\end{proof}

As a corollary, the $\alpha$-limit set of $\mathcal S_{\infty}$ is $q_{n,\infty}$. But then by regular perturbation theory, \KUK{and the hyperbolicity of $q_{n,\infty}$,} we obtain the following result, which we also need in our proof of Lemma \ref{lem:het}.
\begin{corollary}\label{cor:hatmupl}
	There exists a $\hat \mu^+>0$ large enough such that the $\alpha$-limit set of $\mathcal S_{\hat \mu}$ is $q_n(\hat \mu)$ for all \SJJ{$\hat \mu \geq\hat \mu^+$}. 
\end{corollary}

For $\hat \mu \in \mathbb R$, we project $\mathcal S_{\hat \mu}$ onto the $(x_2,\epsilon_2)-$space and denote the projection by the same symbol. A simple calculation shows that it takes the following local form:
\begin{align}
x_2=x_{2,w}-\frac{2}{\tau+\sqrt{\Delta}} \hat \mu \epsilon_{2}^{k/(1+k)} + \mathcal O(\epsilon_2),\label{eq:Shatmu}
\end{align}
for $\epsilon_2>0$ small enough. 

\subsection{Proof of Lemma \ref{lem:het}}
To prove Lemma \ref{lem:het}, we \SJJ{combine our analyses in charts $\check \epsilon = 1$ and $\check \rho = 1$ in order to show the} 
existence of a unique $\hat \mu_{het}$ such that $\mathcal J_{\hat \mu_{het}}$ intersects $\mathcal S_{\hat \mu_{het}}$, transversally with respect to $\hat \mu$.

Before we prove the existence of \SJJ{$\hat \mu_{het}$}, we first show 
that any heteroclinic $\gamma_{het}(t)=(x_{1,het}(t),\rho_{1,het}(t))$ must be monotonically increasing in $\rho_1$, \KUK{i.e.~}$\rho_{1,het}'(t)>0$ for all $t\in \mathbb R$. By the local analysis near $q_a$ and $q_w$, this is true for locally (i.e. for $t\rightarrow \pm \infty$). Moreover, \SJJ{using} \eqref{eq:Shatmu} \SJJ{and} the change of coordinates in \eqref{eq:cc} it follows that $x_{1,het}'(t)>0$ for $t\gg 1$.  Subsequently, recall that $q_n(\hat \mu)$ with coordinates $(x_1,\rho_1)=(-\beta,\rho_{1,n}(\hat \mu))$ is the unique equilibrium for $\rho_1>0$. Then since the $\rho_1$-nullcline is $x_1=-\beta$, it follows that $\dot x_1\gtrless 0$ on $x_1=-\beta$ for $\rho_1\lessgtr \rho_{1,n}(\hat \mu)$. Consequently, if there is a largest $t_1$ such that $\rho_{1,het}'(t_1)=0$, then $\{\gamma_{het}(t)\}_{t\ge t_1}$ and $x_1=-\beta$ together enclose a region to the left which is backward invariant, contradicting the definition of $\gamma_{het}$. We conclude that any heteroclinic $\gamma_{het}$ is monotone in $\rho_1$. 


Next, for the existence of $\hat \mu_{het}$, we use a monotonicity argument  \KUK{as in \cite[App. A]{Kristiansen2019d}}. Specifically, by Lemma \ref{lem:hatmuneg} and Corollary \ref{cor:hatmupl}
there can be no heteroclinics for $\hat \mu\le \hat \mu^-$ or $\hat \mu\ge \hat \mu^+$.
\begin{lemma}\label{lem:collect}
	Consider any $\hat \mu\le \hat \mu^-$. Then:
	\begin{itemize}
		\item The $\omega$-limit set of $\mathcal J_{\hat \mu}$ is $q_n(\hat \mu)$.
		\item The $\alpha$-limit set of $\mathcal S_{\hat \mu}$ is $q_f$.
	\end{itemize}
	Consider any $\hat \mu\ge \hat \mu^+$. Then:
	\begin{itemize}
		\item The $\omega$-limit set of $\mathcal J_{\hat \mu}$ is $q_o$.
		\item The $\alpha$-limit set of $\mathcal S_{\hat \mu}$ is $q_n(\hat \mu)$.
	\end{itemize}
\end{lemma}
\begin{proof}
	\SJJ{This} follows from Lemma \ref{lem:hatmuneg}, Corollary \ref{cor:hatmupl} and \SJJ{the} Poincar\'e-Bendixson \SJJ{theorem}; \SJ{see Figure \ref{fig:bn_het_sphere}}.
\end{proof}
%
%

Following this result, we therefore fix an interval $I=[\hat \mu^-,\hat \mu^+]$ of $\hat \mu$-values, and then insert a section $\Sigma$ at $\rho_1=c$ for $c>0$ small enough. 
The $\rho_1$-nullcline intersects 
$\Sigma$ in a tangency point $(x_{1,t},c)$ for $x_{1,t}:=-\beta$ so that $\dot \rho_1\gtrless 0$ for all points on $\Sigma$ with $x_{1}\gtrless -\beta$. By the previous analysis any heteroclinic connection intersects $\Sigma$ with $x_1>x_{1,t}$. 
The center manifold calculation, see \eqref{eq:Jhatmugraph}, shows that the manifold $\mathcal J_{\hat \mu}$ intersects the section $\Sigma$ transversally in a point $(x_{1,c}(\hat \mu),c)$ for each $\hat \mu\in I$ with $x_{1,c}(\hat \mu)>x_{1,t}$, for all $\hat \mu\in I$ so that $\dot x_{1}>0$, $\dot \rho_1>0$ \SJJ{in a neighbourhood of $\mathcal J \cap \Sigma$}. By Lemma \ref{lem:collect}, we have that for $\hat \mu =\hat \mu^-$ the manifold $\mathcal S_{\hat \mu}$ intersects $\Sigma$ in a point $(x_{1,s}(\hat \mu^-),c)$ with $x_{1,s}(\hat \mu^-)>x_{1,c}(\hat \mu^-)$. The intersection is therefore transverse and we can continue $x_{1,s}(\hat \mu)$ smoothly for larger values of $\hat \mu>\hat \mu^-$. 
However, by Lemma \ref{lem:collect} we know that $\mathcal S_{\hat \mu}$ does not intersect $\Sigma$ for all $\hat \mu\in I$. 
The process of continuing $x_{1,s}$ for larger values of $\hat \mu>\hat \mu^-$ will therefore have to stop when either: $x_{1,s}$ grows unboundedly or $x_{1,s}\rightarrow x_{1,t}^+$. We can exclude the former by the analysis in the $\check \rho=1$ chart. Therefore there is a $\hat \mu_t>\hat \mu^-$ such that $x_{1,s}(\hat \mu)\rightarrow x_{1,t}^+$ for $\hat \mu\rightarrow \hat \mu_t^-$. With $x_{1,c}(\hat \mu_t)>x_{1,t}$ we conclude that the smooth function: $\hat \mu \mapsto x_{1,c}(\hat \mu)-x_{1,s}(\hat \mu)$ for $\hat \mu<\hat \mu_t$ changes sign at least once. The corresponding root corresponds to a heteroclinic connection. 
This connection is unique by the monotonicity of $\rho_{1,het}(t)$ and the fact that the associated Melnikov integral, being the derivative of the Melnikov distance function, has one sign. To see the latter, 
consider \eqref{eq:desing_prob} and notice that the derivative of the right hand side with respect to $\hat \mu$ is
$(\rho_1^{k(2k+1)},0)$. Therefore the sign of the Melnikov integrand \cite{Kuznetsov2013} is determined by 
\begin{align}
(x_{1,het}'(t),\rho_{1,het}'(t)) \wedge (\rho_{1,het}^{k(2k+1)}(t),0)= - \rho_{1,het}'(t)\rho_{1,het}^{k(2k+1)}<0.\label{eq:Melnikov}
\end{align}
This also shows that the intersection of $\mathcal J$ and $\mathcal S$ is transverse, \SJJ{completing} the proof of Lemma \ref{lem:het}.

\subsection{Finishing the proof of Theorem \ref{thm:bn3}}

In Figure \ref{fig:bn_het_cycs} we combine our findings into a new figure illustrating the improved singular cycles $\overline \Gamma(s)$, see the figure caption for further details. 
We then obtain the family of attracting limit cycles in Theorem \ref{thm:bn3} with the prescribed growth rate by perturbing $\overline \Gamma(s)$. For this, we work near $\hat \mu\approx \hat \mu_{het}$ and define two sections $\Sigma_1$ and $\Sigma_2$ as illustrated in Figure \ref{fig:bn_het_cycs}. We then flow points on $\Sigma_1$ forward and backward and measure their separation on $\Sigma_2$. The sections are defined in the chart $\check \rho=1$ with coordinates $(x_2,\epsilon_2,\nu_2,\mu_2)$, recall \eqref{eq:checkrho1}, as follows:
\begin{align*}
\Sigma_{1}:\quad \nu_2=\xi,\,x_2\in I, 0\le \epsilon_2,\mu_2\le \chi,\\
\Sigma_{2}:\quad \epsilon_2=\chi,\,x_2\in I, 0\le \nu_2,\mu_2\le \xi,
\end{align*}
for $\chi,\nu>0$ small enough and $I$ a small enough neighbourhood of $x_{2,w}$ such that the following local arguments apply
near $q_w = (x_{2,w},0,0,0)$, recall \eqref{eq:qw}. The bifurcation equation is then given by 
\begin{align}
F(x_2,\hat \mu,\epsilon_2) - B(x_2,\hat \mu,\epsilon_2)=0,\label{eq:bifeq}
\end{align}
where $F$ and $B$ are defined as the $x_2$-coordinates of the points on \SJJ{$\Sigma_2$} obtained by following \SJJ{initial conditions} $(x_2,\epsilon_2,\xi,\mu_2)$ \SJJ{on $\Sigma_1$} forward and backward, respectively, where $\mu_2=\epsilon_2^{k/(1+k)}\hat \mu$. Notice, by conservation of $\epsilon$ and $\hat \mu$, solutions of \eqref{eq:bifeq} with $\epsilon_2>0$ define closed orbits. Let $J$ be a sufficiently small neighbourhood of $\hat \mu=\hat \mu_{het}$. Then we have \SJ{the following}.
\begin{proposition}\label{prop:FB}
	Consider any $\nu \in (0,1)$. Then there exists an $\epsilon_{20}$ such that \SJ{
	$I\times J \ni (x_2,\hat \mu) \mapsto F(x_2,\hat \mu,\epsilon_2)$} and $I\times J \ni (x_2,\hat \mu) \mapsto B(x_2,\hat \mu,\epsilon_2)$ are both well-defined and $C^1$ depending continuously on $\epsilon_2\in [0,\epsilon_{20})$. In particular,\SJ{
	\begin{align}
	\label{eq:FB}
	F(x_2,\hat \mu_{het},0) = B(x_2,\hat \mu_{het},0),\quad F'_{\hat \mu}(x_2,\hat \mu_{het},0) \ne  B'_{\hat \mu}(x_2,\hat \mu_{het},0),
	\end{align}
	}for all \SJJ{$x_2\in I$}. Moreover, let $\lambda=2\sqrt{\Delta}/(\tau-\sqrt{\Delta})$ as defined in Theorem \ref{thm:bn3}. Then 
	there is a $c_F>0$ such that
	\[
%
	F'_{x_2}(x_2,\hat \mu_{het},\epsilon_2) = \mathcal O(e^{-c_F/\epsilon_2}), \quad 
	B'_{x_2}(x_2,\hat \mu_{het},0) = \mathcal O(\epsilon_2^{\nu k\lambda/(k+1)}).
	\]
\end{proposition}

Before we prove this proposition, we will first show that it implies Theorem \ref{thm:bn3}. For this, we first notice that for any $c\in (0,s_0)$ there exists a $\xi(c)>0$ small enough such that \SJ{$\overline \Gamma_{X^+}(s)$} for $s\in (c,s_0)$ intersects $\Sigma_1$ once in a single point $(x_2(s),0,\xi,0)$ with $x_2'(s)>0$. \SJJ{Theorem \ref{thm:bn3} follows after} applying the implicit function theorem to \eqref{eq:bifeq} and using the properties described in Proposition \ref{prop:FB}. Notice in particular, that this gives a solution of \eqref{eq:bifeq} of the form $\hat \mu=\hat \mu(x_2,\epsilon_2)$. Seeing that $\epsilon_2=\xi^{-(1+k)}\epsilon$ and $\mu=\epsilon^{k/(1+k)}\hat \mu$ by \eqref{eq:checkrho1} \SJ{on $\Sigma_1$}, we obtain the desired $\mu(s,\epsilon) = \epsilon^{k/(1+k)}\hat \mu(x_2(s),\xi^{-(1+k)}\epsilon)$. 

\subsection{Proof of Proposition \ref{prop:FB}}
The properties of $F$ \SJ{are} standard\SJ{; see \cite{Jelbart2020d}. Here} we just summarize the approach. Firstly, as we flow points $\Sigma_1$ forward, following \SJ{$\overline \Gamma_{X^+}(s)$}, they are eventually exponentially contracted towards the slow manifold. We then guide the flow along this manifold, \SJJ{first along} $\mathcal W_0$, and eventually \SJJ{along} $\mathcal J_{\hat \mu}$ using the center manifold at $q_a$. Seeing that $\hat \mu \approx \hat \mu_{het}$, we conclude that \SJJ{solutions} reach $\Sigma_2$ \SJ{so} that $F$ is well-defined. In particular, $F(x_2,\hat \mu,0)$ is the $x_2$-value of the intersection $\mathcal J_{\hat \mu}\cap \Sigma_2$. 

For $B$ we proceed more carefully. 
Firstly, we recall that $q_w$ is fully hyperbolic. Indeed, the linearization has the following non-zero eigenvalues:
\begin{align*}
\lambda x_{2,w}, \ -\frac{1+k}{k}x_{2,w}, \ -\frac{1}{2k(1+k)} x_{2,w}, \  -x_{2,w},
\end{align*}
where $\lambda=2\sqrt{\Delta}/(\tau-\sqrt{\Delta})$ as defined in Theorem \ref{thm:bn3}. To describe the map $\Sigma_1\rightarrow \Sigma_2$ defined by the backward flow of \eqref{eq:checkrho1eqs}, we use partial linearizations (since resonances in general will preclude a full linearization) within the invariant spaces $\{\epsilon_2=0\}$ and $\{\nu_2=\mu_2=0\}$ \SJJ{in order} to obtain the following. 

\begin{lemma}
	\label{lem:coord_trans}
	There exists a $C^1$ diffeomorphism bringing  \eqref{eq:checkrho1eqs} into the following system
	\begin{equation}\label{eq:normalized2}
	\begin{aligned}
	x_2' &= \lambda x_2+\epsilon_2^k \mathcal O(\nu_2^{2k(1+k)}+\epsilon_2\nu_2^{2(1+k)}), \\
	\epsilon_2' &= -\frac{k+1}{k}  \epsilon_2, \\
	\nu_2' &= \frac{\nu_2}{2k(1+k)},\\
	\mu_2' &=-\mu_2,
	\end{aligned}
	\end{equation}
	upon a regular reparametrization of time. 
\end{lemma}

\begin{proof}
\SJ{See Appendix \ref{app:coord_trans}.}
\end{proof}

We now integrate \eqref{eq:normalized2} backwards from $\nu_2=\xi$ to $\epsilon_2=\chi$. A simple calculation, based upon a Gronwall-type estimate, gives
\begin{align*}
x_2 \mapsto \left(\chi^{-1} \epsilon_2\right)^{k\lambda/(1+k)} x_2+\mathcal O(\epsilon_2^{1/(1+k)}),
\end{align*}
which defines $B$ in the new local coordinates. From this, we then similarly obtain the desired estimate of the $x_2$-derivative of $B$ in Proposition \ref{prop:FB} \SJJ{using the relevant variational equations (these are taken with respect to system \eqref{eq:normalized1} in Appendix \ref{app:coord_trans})}.  Moreover, it is clear that $B(x_2,\hat \mu,0)$ coincides with the $x_2$-coordinate of the intersection $\mathcal S_{\hat \mu}\cap \Sigma_2$. Consequently, \eqref{eq:FB} holds by Lemma \ref{lem:het} and the transverse intersection of $\mathcal J$ and $\mathcal S$. 
%
%

\section{Outlook}
\label{sec:Outlook}

The unfolding of BE singularities in BEB under parameter variation is generic in PWS systems. It follows that singularly perturbed BEB is also generic under parameter variation in singular perturbation problems losing smoothness along a codimension-1 switching manifold $\Sigma$ as a perturbation parameter $\epsilon \to 0$. In this manuscript the notion of singularly perturbed BEB is formally defined (see Definition \ref{def:beb}), and a classification based on known classifications for PWS BEB from \cite{Filippov1988,Hogan2016,Kuznetsov2003} is given; see Table \ref{tab:classification}. We showed in Theorem \ref{theorem:prop_normal_form} that the local normal form \eqref{eq:normal_form} first derived in \cite{Jelbart2020d} for the analysis of singularly perturbed BF bifurcations in particular, is capable of generating all 12 singularly perturbed BEBs. It is worthy to note that a corresponding PWS normal form \eqref{eq:PWS_normal_form} is \SJJ{also} obtained in the limit $\epsilon \to 0$. 

Following the introduction of the normal form \eqref{eq:normal_form}, we studied its dynamics in parameter regions corresponding to each singularly perturbed BEB. Using a sequence of blow-up transformations to resolve a loss of smoothness along $\Sigma$, and subsequently, degeneracy arising from the BE singularity itself, we derived two desingularized systems \eqref{eq:desing_prob} and \eqref{eq:desing_prob2} in Lemma \ref{lem:desing_prob}. Studying the dynamics of these systems allowed for a detailed description of the unfolding for all 12 singularly perturbed BEBs. This was presented succinctly in Theorem \ref{thm:bifs} and Figure \ref{fig:bifs}. In many cases, we were able to provide explicit paramterizations for the location of codimension 1 and 2 bifurcations involved in the unfoldings. It is worthy to note that in general, the bifurcation structure depends quantitatively, but not qualitatively, on the decay rate $k$ determining the rate at which the system loses smoothness along $\Sigma$ as $\epsilon \to 0$; see equation \eqref{eq:reg_asymptotics}. In particular, all identified bifurcations are singular, in the sense that they occur within a parameter regime $\mu = \mathcal O(\epsilon^{k/(k+1)})$ which shrinks to zero in the PWS limit $\epsilon \to 0$.

We then demonstrated the suitability of our framework for studying the geometry of so-called double-separatrices, which constitute non-trivial boundaries between cases BF$_{1,2}$ and BS$_{1,2}$ in parameter space. A result on the boundary between cases BF$_{1,2}$ was presented in Proposition \ref{prop:BF}, but a complete analysis of this and the BS$_{1,2}$ boundary is left for future work.

Finally, special attention was devoted to the so-called BN$_3$-explosion, which may be considered a `generic analogue' of the (degenerate) explosion identified already in \cite{Kristiansen2019d}, which can be considered as the `$\gamma = 0$ case' of system \eqref{eq:normal_form}. We \SJJ{showed} that the continuous family of singular cycles shown in Figure \ref{fig:bn_het_cycs} perturbs to a continuous family of stable limit cycles for $0<\epsilon\ll1$. This is described in Theorem \ref{thm:bn3}, where the growth rate of the cycles is also quantified as a function of $\epsilon$ and the parameter $k$ determining the rate at which the system loses smoothness. We emphasize that the focus of  \cite{Kristiansen2019d} was on the existence of relaxation oscillations, \SJJ{whereas here we focus on the details of} the explosion itself. 

\

We conclude with some discussion on the relation to explosive onset of oscillations in classical slow-fast systems and other singularly perturbed BEBs is discussed below. Applications, more degenerate cases of interest, \SJJ{and singular bifurcations of higher codimension} are also considered.

\subsection{Relation to classical canard explosion and singularly perturbed BF$_3$ explosion}
\label{ssec:outlook_canard_explosion}

	\KUK{In Theorem \ref{thm:bn3} we described a novel explosion mechanism of limit cycles due to the BN$_3$ bifurcation. This `explosion'
	is reminiscent of the canard explosion phenomenon in classical slow-fast systems. Here too, an entire family of singular cycles exists for a unique parameter value. Geometrically both families are upon blow-up identified in a similar way through heteroclinic cycles; for Theorem \ref{thm:bn3} the heteroclinic cycles occur due to the transverse intersection of the manifolds $\mathcal J_{\hat \mu}$ and $\mathcal S_{\hat \mu}$.
	Moreover, in both cases, the singular cycles perturb to limit cycles for $0<\epsilon\ll1$, see e.g.~\cite{Dumortier1996,Krupa2001b,Kuehn2015}. 
	
	However, the `explosion' described by Theorem \ref{thm:bn3} differs from the classical canard explosion phenomenon in a number of important respects. First, for the BN$_3$ bifurcation, there is only an attracting slow manifold. \SJ{Since there is no repelling slow manifold, there are no canards.} 
	Instead, \SJ{repulsion in} the BN$_3$ \SJ{explosion} comes from the unstable node. Consequently, the limit cycles in Theorem \ref{thm:bn3} are also always stable since the contraction of the slow manifold dominates the hyperbolic repulsion from the node; in the canard case the limit cycles can be either attracting, repelling or neither, the details depending on \WM{a} slow divergence integral \cite{Krupa2001b}. These differences also manifest themselves through different growth rates. In the BN$_3$ case, the growth rate \eqref{eq:growth_rate} is algebraic, whereas the classical canard explosions are characterised by exponential growth (since \eqref{eq:growth_rate} is exponentially small in this case).
	
	The onset of oscillations that are $\mathcal O(1)$ with respect to $\epsilon$ in the singularly perturbed BF$_3$ bifurcation, described in detail in \cite{Jelbart2020d}, is again different. \SJ{In the BF$_3$ case} one also identifies a family of singular cycles, but these cycles all lie on the blow-up spheres. \SJ{As a consequence, the family of limit cycles obtained after perturbation and blow-down} 
	is not explosive in any way. We will discuss this further in the following section in the context of a regularized stick-slip oscillator capable of producing both bifurcations BN$_3$ and BF$_3$ (albeit degenerate ones) upon parameter variations. }
	

\subsection{Degenerate singularly perturbed BE bifurcation in applications}
\label{ssec:outlook_applications}


\SJ{Singularly perturbed} BN$_3$ explosion of the kind described in Theorem \ref{thm:bn3} occurs in a regularized Gause problem; see \cite{Jelbart2020d} for the regularized model, and \cite{Gause1936} for the original PWS system. As previously stated, \KUKK{a singularly perturbed BN} bifurcation also occurs in the model for substrate-depletion in \cite{Kristiansen2019d}, where it \SJ{shown to provide a mechanism for} the onset of the relaxation-type oscillations. 
\SJ{In} the context of the normal form \eqref{eq:normal_form}, this case is `degenerate' due \SJ{to} $\gamma = 0$. As shown in \cite{Kristiansen2019d}, this produces a critical manifold $S$ with no reduced flow, however with an additional (infra-)slow timescale. 

\begin{figure}[t!]
	\centering
	\includegraphics[scale=0.6]{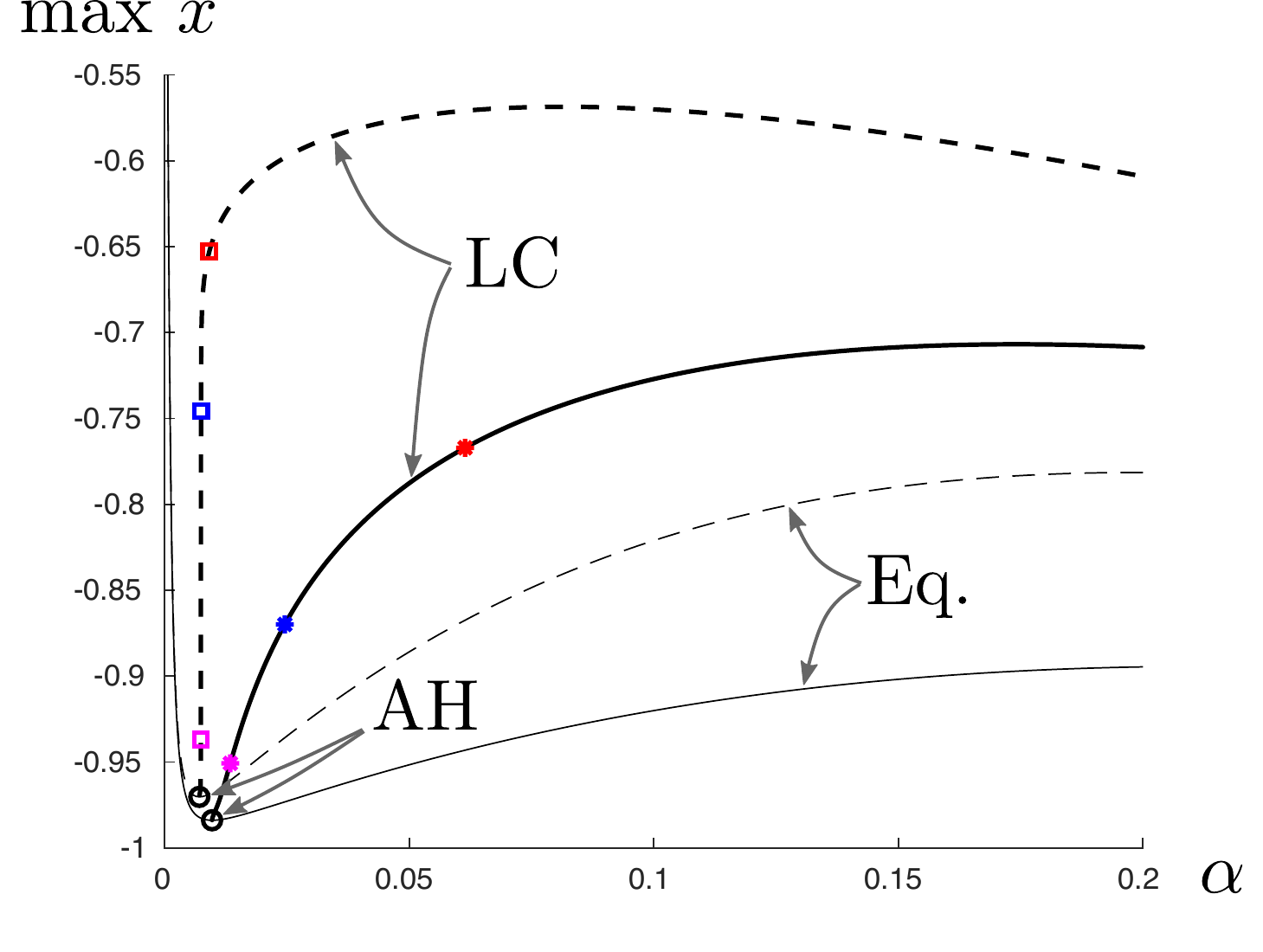}
	\caption{Bifurcation diagram for \SJ{system \eqref{eq:general} defined by \eqref{eq:stickslip} with \eqref{eq:frictionlaw} and \eqref{eq:ss_reg}}. The full lines are for the parameter values in \eqref{eq:bf3par} whereas the dashed lines are for the parameter values in \eqref{eq:bn3par}, corresponding to (degenerate) \SJ{BF$_3$ and BN$_3$} bifurcations, respectively. The thinner lines track the equilibrium whereas the thicker curves correspond to the limit cycles \SJ{(denoted LC in the figure)}, using $\max x$ as a measure of the amplitude, emerging from the two \SJ{Andronov-}Hopf \SJ{AH} bifurcation points indicated by two black circles. The limit cycles corresponding to the purple, blue and red points ({disks} for the \KUKK{BF} case and squares for the \SJ{BN} case) are illustrated in Figure \ref{fig:kuk_ph}(a) and (b), respectively.}
	\label{fig:kuk_bif}
\end{figure}

\begin{figure}[t!]
	\centering
	\subfigure[]{\includegraphics[width=.45\textwidth]{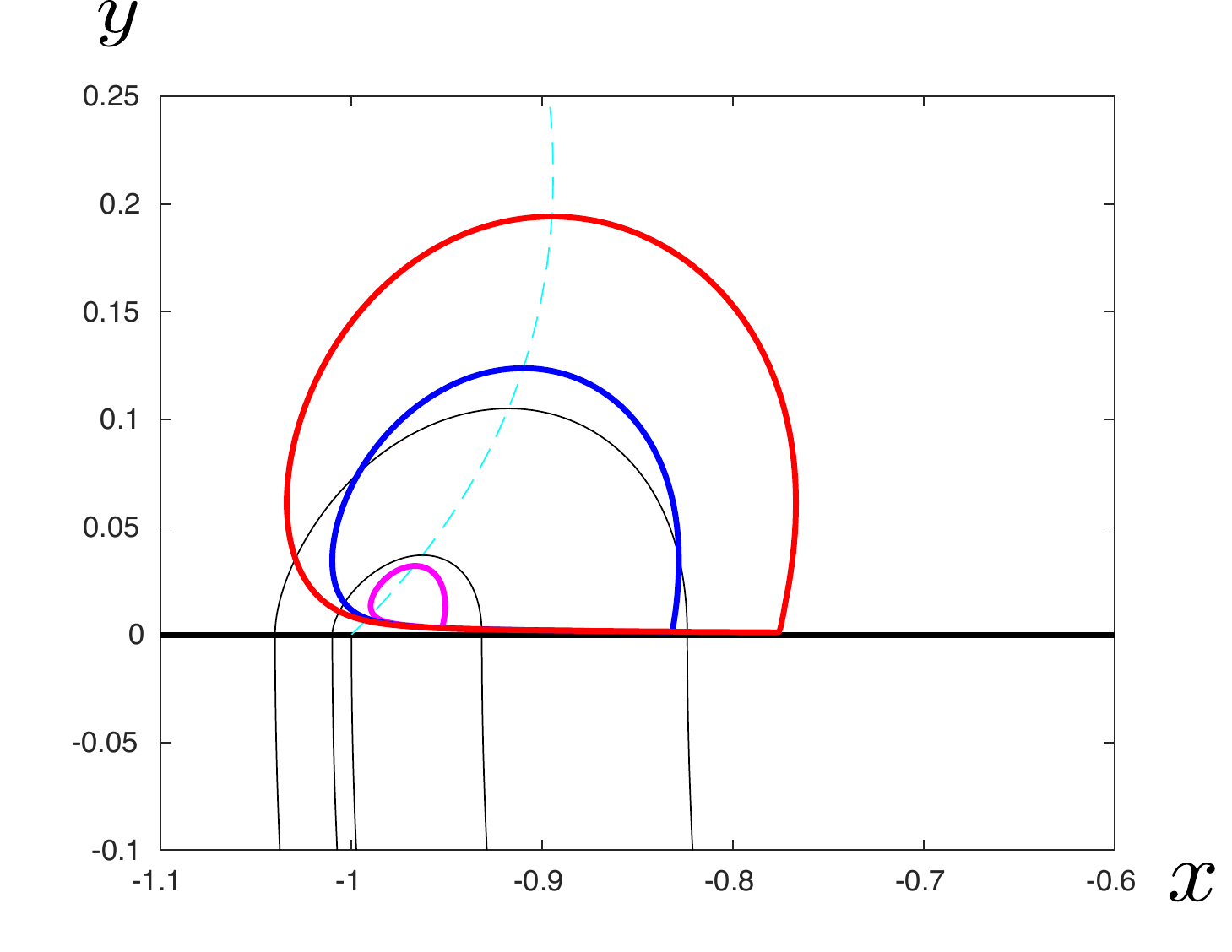}}
	\subfigure[]{\includegraphics[width=.45\textwidth]{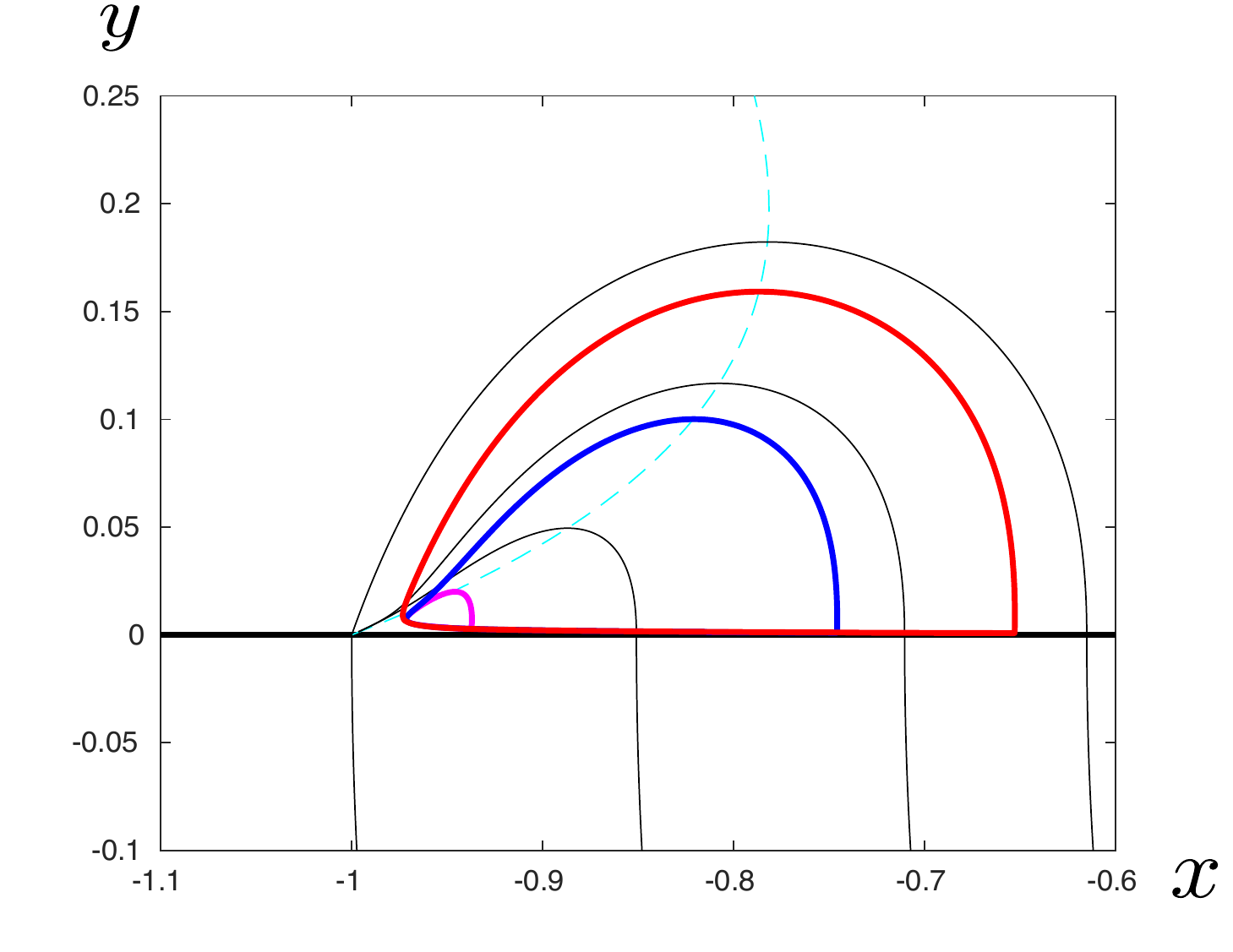}}
	\caption{Limit cycles of \SJ{system \eqref{eq:general} defined by \eqref{eq:stickslip} with \eqref{eq:frictionlaw} and \eqref{eq:ss_reg}}. Here (a) and (b) correspond to the parameter values in \eqref{eq:bf3par} and \eqref{eq:bn3par} producing (degenerate) \SJ{BF$_3$} and \SJ{BN$_3$} bifurcations, respectively. The thinner black lines show the phase portraits of the corresponding PWS system at the singular bifurcation $\alpha=0$. \SJ{The switching manifold along $\{y=0\}$ is shown in black, and the} dotted cyan \SJJ{curve} is the nullcline of $Z^+(\cdot,0)$.}
	\label{fig:kuk_ph}
\end{figure}

Similar (degenerate) explosions may also be observed in regularized stick-slip oscillators under variation of the belt speed, \SJ{see e.g.~}\cite[Section 5.1]{Jelbart2020d}. This model takes the form \eqref{eq:general} satisfying Assumption \ref{assumption:1} with 
\begin{align}
 Z^+(x,y,\alpha) = \begin{pmatrix}
               y-\alpha\\
               -x-\mu(y)
       \end{pmatrix},\quad 
 Z^-(x,y,\alpha) = \begin{pmatrix}
                y-\alpha\\
                -x+\mu(-y)
       \end{pmatrix} ,       \label{eq:stickslip}
\end{align}
where $\mu$ describes the friction law. In \cite{Kristiansen2019c} it is given as 
\begin{align}
 \mu(y) = \mu_m+(\mu_s-\mu_m) e^{-\rho y} + c y, \label{eq:frictionlaw}
\end{align}
which was proposed by \cite{Berger2002} and also studied in \cite{papangelo2017a,Won2016}. Here $\mu(0)=\mu_s$ and $\mu_m>\mu_m>0$, $\rho>0$, $c\in( 0,\rho(\mu_s-\mu_m))$ to ensure that $\mu'(0)<0$. 
The model with (\ref{eq:stickslip}) has a BEB for $\alpha=0$ at $(x,y)=(-\mu_s,0)$. An easy calculation, see also \cite{Jelbart2020d}, shows that this bifurcation \SJ{can be either} a degenerate BN$_3$ \SJ{for} $\mu'(0)<-2$, \SJ{or a degenerate BF$_3$ for} $\mu'(0)\in (-2,0)$. Although these bifurcations are degenerate with $\gamma=0$, we nevertheless use this example to illustrate in Figures \ref{fig:kuk_bif} and \ref{fig:kuk_ph} the differences between these cases. Specifically, in Figure \ref{fig:kuk_bif} we show a bifurcation diagram for two different sets of parameters: The dotted lines are for 
\begin{align}
 \mu_s = 1,\,\mu_m=0.5,\,c=0.85\quad \mbox{and} \quad \rho=4,\label{eq:bf3par}
\end{align}
while the full lines are for the same values except \SJJ{with} $\rho=7.5$, i.e. 
\begin{align}
 \mu_s = 1,\,\mu_m=0.5,\,c=0.85\quad \mbox{and}\quad  \rho=7.5. \label{eq:bn3par}
\end{align}
These two cases give $\mu'(0)=-1.10$ and $\mu'(0)=-2.90$, respectively, and therefore correspond to (degenerate) \SJ{BF$_3$ and BN$_3$} cases. \SJ{We use a regularization
\begin{equation}
\label{eq:ss_reg}
\phi(y) = \frac{1}{2} \left( 1 + \frac{y}{\sqrt{y^2 + 1}} \right) 
\end{equation}
which satisfies Assumptions \ref{assumption:1}, \ref{assumption:2} and \ref{assumption:3} with} \KUK{$k=2$, and set $\epsilon = 0.001$ in system \eqref{eq:general}.} As discussed, we only see an explosive growth \SJ{of the limit cycle amplitude in the BN$_3$ case}. For further comparison, Figure \ref{fig:kuk_ph} illustrates examples of limit cycles; (a) \SJ{in case} \SJ{BF$_3$}, \SJJ{and} (b) \SJJ{in case} \SJ{BN$_3$}. The colours correspond to the colours of the points in Figure \ref{fig:kuk_bif}. See figure captions for further details.

\begin{figure}[t!]
	\centering
	\includegraphics[scale=.4]{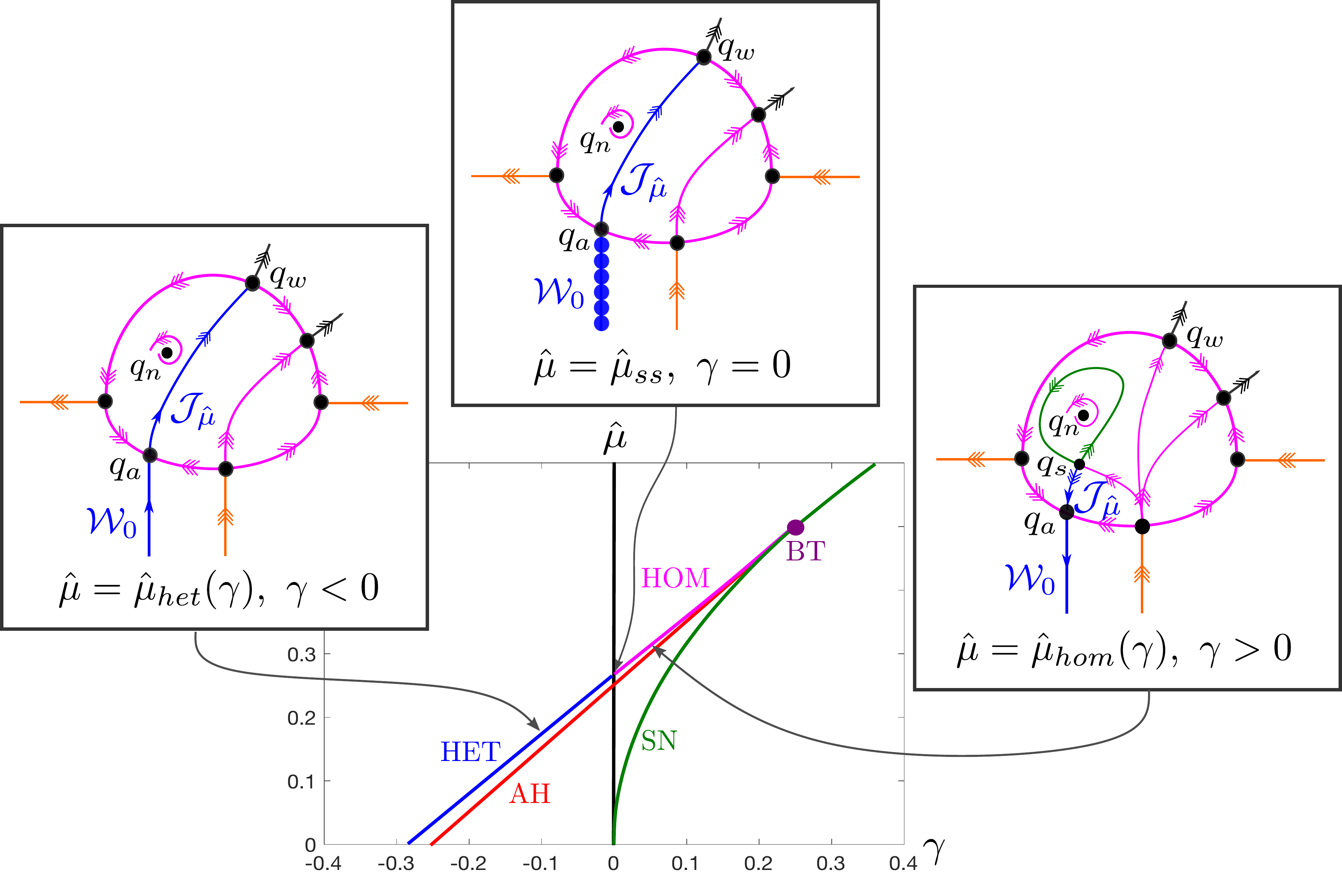}
	\caption{Bifurcation diagram from Figure \ref{fig:bifs}(b), overlaid with the numerically computed branch of heteroclinics $\hat \mu = \hat \mu_{het}(\gamma)$ of Lemma \ref{lem:het}. Homoclinic and heteroclinic branches appear to extend and intersect for $\gamma = 0$. The degenerate singularly perturbed BN bifurcation for $\gamma = 0$, i.e.~between cases \SJJ{BN$_2$} and BN$_3$, has been described in detail already in an application in \cite{Kristiansen2019d}. Corresponding dynamics after blow-up are shown for $\gamma < 0$, $\gamma = 0$ and $\gamma > 0$.}
	\label{fig:bifs_het}
\end{figure}

\subsection{Connecting BN$_{2,3}$ across $\gamma=0$}
\label{ssec:bn13_connection}

Since Theorem \ref{theorem:prop_normal_form} applies for all $\gamma \in \mathbb R$, it follows that both smooth and PWS normal forms for degenerate BE bifurcations with $\gamma = 0$ are obtained by setting $\gamma = 0$ in \eqref{eq:normal_form} and \eqref{eq:PWS_normal_form} respectively. Hence, it is expected that much of the analysis presented herein can be applied in order to study these degenerate cases, which can be thought of as `boundary cases' separating (singularly perturbed) BF$_{1,3}$, BF$_{4,5}$, BN$_{2,3}$, BN$_{1,4}$ and BS$_{1,3}$ bifurcations; see again Figure \ref{fig:bifs} and the caption.

The BN$_{2,3}$ boundary is particularly interesting as (i) it arises naturally in the context of substrate-depletion oscillations as described above \cite{Kristiansen2019d}, and (ii) there is evidence that upon extension through $\gamma = 0$, the homoclinic branch identified in case \SJJ{BN$_2$} with $\gamma>0$ connects to the heteroclinic branch for $\gamma < 0$ which is responsible for the explosion in case BN$_3$. Computing the heteroclinic connection of Lemma \ref{lem:het} numerically and plotting it over the bifurcation diagram in Figure \ref{fig:bifs}(b), we obtain the diagram in Figure \ref{fig:bifs_het}.  
Here we see the expected transition from case BN$_3$ to \SJJ{BN$_2$} as $\gamma$ crosses zero. 
Our observations also provide evidence that $\hat \mu_{het}(\gamma) < \hat \mu_{ah}(\gamma)$ on $\gamma < 0$. Note that if $\hat \mu_{het}(\gamma)$ is an analytic continuation of $\mu_{hom}(\gamma)$ \SJ{through $\gamma = 0$}, then despite appearances in Figure \ref{fig:bifs_het}, it cannot be linear. This would follow from the nonlinearity in the local parameterization of $\mu_{hom}(\gamma)$ near the BT point.

\SJJ{
	\subsection{Higher codimensions}
\label{ssec:higher_codim}

As highlighted in observation (ix) following the statement of Theorem \ref{thm:bifs}, the diagrams in Figure \ref{fig:bifs} can be obtained from one another under suitable variation in two additional parameters $\tau$ and $\delta$. In particular, it follows from the classification in Table \ref{tab:classification} that $\delta = 0$, $\tau = 0$ and $\Delta = 0$ form boundaries between the cases represented in Figure \ref{fig:bifs}. For example, panel (b) is obtained from panel (a) by crossing from $\Delta <0$ to $\Delta > 0$ with $\tau, \delta > 0$. To see this transition in the blown-up space, compare Figures \ref{fig:bu2}(b) and \ref{fig:bu2}(c). For $\Delta = 0$, a saddle-node bifurcation occurs on the intersection of the upper blow-up sphere (shown in magenta) with the plane $\{\epsilon = \ 0\}$, giving rise to the equilibria $q_w$ and $q_o$. Our analysis provides a framework within which transitions such as these can be analysed, thereby providing a solid foundation and program for future work.

Additional variation in $(\tau,\delta)$ leads naturally to higher codimension (singular) bifurcations. Consider for example the effect of crossing $\delta = 0$ for fixed $\tau > 0$. As $\delta \to 0^+$, the codimension-2 Bogdanov-Takens point $(\hat \mu_{bt},\gamma_{bt})(\delta) \to (0,0)$ in either Figure \ref{fig:bifs}(a) or \ref{fig:bifs}(b). Conversely, we expect that the codimension-2 point corresponding to the intersection of the heteroclinic and saddle-node curves in Figure \ref{fig:bifs}(d) tends to $(0,0)$ in the limit $\delta \to 0^-$; see again Remark \ref{rem:bs12}. These observations indicate the existence of a codimension-3 (singular) bifurcation for $\hat \mu = \gamma = \delta = 0$, $\tau > 0$, which involves the above-mentioned codimension-2 bifurcations in its unfolding. Moreover, since topologically non-equivalent diagrams also arise if one considers instead the same case but with $\tau < 0$, it follows that the dynamics are organised by a singular codimension-4 bifurcation for $\gamma = \hat \mu = \delta = \tau = 0$. The current manuscript therefore serves as a strong foundation for ample future work in this direction. 
}

\bibliographystyle{siam}
\bibliography{beb_bib-1}

\appendix

\section{Proof of the normal form theorem \ref{theorem:prop_normal_form}}
\label{app:normal_form}

Following a suitable parameter-dependent coordinate translation we may assume that system \eqref{eq:general} has a nondegenerate BE bifurcation at $z_{bf}=(0,0)$ when $\alpha_{bf}=0$. It follows by arguments analogous to \cite[p.38]{Jelbart2020d} (see also \cite{Jelbart2020b} for further details) that the system
\begin{equation}
\label{eq:expanded_sys}
\begin{split}
\dot u &= \phi\left(y \epsilon^{-1} \right) \left[ s_1(\alpha) + a(\alpha) u + b(\alpha) y + \varphi_1\left(u, y, \alpha \right) \right] ,  \\
\dot y &= 1 + \phi\left(y \epsilon^{-1} \right) \left[ - 1 + s_2(\alpha) + c(\alpha) u + d(\alpha) y + \varphi_2\left(u, y, \alpha  \right) \right] ,
\end{split}
\end{equation}
can be obtained from system \eqref{eq:general} after a smooth invertible local coordinate transformation of the form $u=L(x,y,\alpha)$, and \SJJ{a} transformation of time amounting to division by $Z^-_2((M(u,y,\alpha), y), \alpha)$, which is locally nonzero due to \eqref{eq:beb_conds}.\footnote{The $(\dot{\ })$ notation in \eqref{eq:expanded_sys} denotes differentiation with respect to the new (transformed) time.} As described in \cite[p.38]{Jelbart2020d}, the function $L(x,y,\alpha)$ can be chosen such that local orbit segments of $Z^-(x,y,\alpha)$ are given by level sets $L(x,y,\alpha)=const$. The quantities $a(\alpha), b(\alpha), c(\alpha), d(\alpha)$ are smooth functions of $\alpha$ such that $a(0) = a, \ b(0) = b, \ c(0) = c, \ d(0) = d$ are constant, $s_i(\alpha)$, $\varphi_i(y,y,\alpha)$, $i=1,2,$ are smooth functions satisfying $s_i(0) = 0, \ i = 1,2$, $s_2'(0) > 0$, and $||\varphi_i(u,y,\alpha)|| = \mathcal O (||(u,y,\alpha)||^2)$. Note that due to the transformation of time, the orientation is reversed if $Z_2^-((0,0),0) < 0$.

System \eqref{eq:expanded_sys} inherits a BEB at $(u,y)=(0,0)$ for $\alpha=0$, of the same topological type as the BE bifurcation in the original system \eqref{eq:general}. In particular, it follows from Definition \ref{def:beb} that the following nondegeneracy conditions are satisfied:
\[
\tau := a + d \neq 0 , \qquad \delta := ad-bc \neq 0, \qquad \tau^2 - 4\delta \neq 0 .
\]
Indeed, the requirement that the eigenvectors $v_\pm(\alpha_{bf})$ in Definition \ref{def:beb} are transverse to $\Sigma$ ensures that either $b \neq 0$, $c \neq 0$ or both $b,c \neq 0$.\footnote{This is immediate for BF bifurcations; transverse intersection of $v_\pm(\alpha_{bf})$ and $\Sigma$ is only required for the BN and BS bifurcations.} Without loss of generality we may assume that $c > 0$.\footnote{The alternative $c < 0$ leads to an equivalent normal form which can be obtained from system \eqref{eq:normal_form} via $(x,\mu) \mapsto (-x,-\mu)$ as described in \cite[Rem.~2.7]{Jelbart2020d}.} System \eqref{eq:normal_form} is obtained from \eqref{eq:expanded_sys} after making a linear coordinate transformation
\[
v = c (u + w(\alpha)) + d y , \qquad u = \frac{1}{c} \left( v - d y\right) - w(\alpha) ,
\]
where $w : I_\alpha \to \mathbb R$ is a smooth function satisfying $w(0) = 0$, and a suitable application of the inverse function theorem. This leads to the system
\[
\begin{split}
\dot v &= d + \phi\left(y \epsilon^{-1} \right) \left( -d + \mu + \tau v - \delta y + \tilde \theta_1(v,y,\alpha) \right),  \\
\dot y &= 1 + \phi\left(y \epsilon^{-1} \right) \left( -1 + v + \tilde \theta_2(v,y,\alpha) \right) ,
\end{split}
\]
where we have defined a new parameter
\[
\mu(\alpha) := cs_1(\alpha) + ds_2(\alpha) - c\tau w(\alpha) + \mathcal O(\alpha^2) ,
\]
and $w(\alpha) = (s_2'(0)/c)\alpha + \mathcal O(\alpha^2)$, see again \cite[p.38]{Jelbart2020d} for details. Notice that $\mu'(0) \neq 0$ by the determinant condition in \eqref{eq:beb_conds}, so that $\mu = \mu(\alpha)$ is invertible with inverse $\alpha(\mu)$ such that $\alpha(0)=0$ and $\alpha'(0)\neq 0$. Hence we may define $\theta_i(v,y,\mu) : = \theta_i(v,y,\alpha(\mu))$ for $i=1,2$. Setting $\gamma := \tau - d$ and (by a slight abuse of notation) $v=x$, this yields the form in \eqref{eq:normal_form}.

Finally, the Filippov/sliding vector field $X_{sl}(x,\mu)$ in \eqref{eq:Filippov_VF} is obtained directly from system \eqref{eq:PWS_normal_form} using the formula \eqref{eq:Filippov_x} with $Z^\pm = X^\pm$. The form for $\gamma$ is motivated by
\[
\gamma = X_{sl}'(0,0) .
\]
%
%

\section{Proof of Lemma \ref{lem:desing_ext}}
\label{app:bu_eqns}

We focus on \eqref{eq:checkeps1eqs}, the details of \eqref{eq:checkrho1eqs} being almost identical \SJ{and} therefore left out. 
To obtain \eqref{eq:checkeps1eqs} we insert \eqref{eq:checkeps1} into the extended system $\{(x',y')=\epsilon X(x,y,\mu,\epsilon),\epsilon'=0,\mu'=0\}$. The easiest way to do this is to 
use the fact that \eqref{eq:checkeps1} is the composition of three mappings defined by 
\eqref{eq:K1}, \eqref{eq:hatr1} and finally
\begin{equation}\label{eq:checkeps1f}
\begin{aligned}
x_2 &=\nu_1^{k(1+k)} x_1,\\
\rho_2 &=\nu_1^{k+1}\rho_1,\\
\epsilon_2 &=\nu_1^{k+1},\\
\mu &=\nu^{2k(1+k)}\mu_1.
\end{aligned}
\end{equation}
We therefore compute the resulting equations in turn. First, we insert \eqref{eq:K1} into $\{(x',y')=\epsilon X(x,y,\mu,\epsilon),\epsilon'=0,\mu'=0\}$. This gives
\begin{equation}\label{eq:bary1eqs}
\begin{aligned}
x' &=r_1 \left[(1 - \epsilon_1^k \phi_+(\epsilon_1)) (\mu+\tau x-\delta r_1 +\theta_1(x,r_1,\mu))-\epsilon_1^k\phi_+(\epsilon_1)(\gamma-\mu)\right] ,\\
r_1'&=r_1 \left\{\left(1 - \epsilon_1^k \phi_+(\epsilon_1)\right)\left(x+\theta_2(x,r_1,\mu)\right)+\epsilon_1^k \phi_+(\epsilon_1)\right\},\\
\epsilon_1' &=-\epsilon_1\left\{\left(1 - \epsilon_1^k \phi_+(\epsilon_1)\right)\left(x+\theta_2(x,r_1,\mu)\right)+\epsilon_1^k \phi_+(\epsilon_1)\right\},
\end{aligned}
\end{equation}
and $\mu'=0$, upon desingularizing through the division by $\epsilon_1$ on the right hand side.
Subsequently, we insert \eqref{eq:hatr1} into \eqref{eq:bary1eqs}. This gives
\begin{equation}\label{eq:bary1epseqs}
\begin{aligned}
x_2' &= a_2(x_2,\rho_2,\epsilon_2,\mu))-\frac12 x_2 b_2(x_2,\rho_2,\epsilon_2,\mu),\\
\rho_2' &= \frac{1}{2k(1+k)} \rho_2 b_2(x_2,\rho_2,\epsilon_2,\mu),\\
\epsilon_2' &= -\frac{2k+1}{2k} \epsilon_2 b_2(x_2,\rho_2,\epsilon_2,\mu),
\end{aligned}
\end{equation}
along with $\mu'=0$, upon desingularization through the division by $\rho_1^{k(1+k)}$ on the right hand side, where
\begin{align*}
a_2(x_2,\rho_2,\epsilon_2,\mu) &= (1 - \rho_2^{k(1+k)}\epsilon_2^k \phi_+(\rho_2^{1+k}\epsilon_2)) \bigg(\mu+\tau \rho_2^{k(1+k)}x_2-\delta \rho_2^{2k(1+k)} \\
&+\theta_1(\rho_2^{k(1+k)} x_2,\rho_2^{2k(1+k)},\mu)\bigg)-\rho_2^{k(1+k)}\epsilon_2^k\phi_+(\rho_2^{1+k}\epsilon_2)(\gamma-\mu),\\
b_2(x_2,\rho_2,\epsilon_2,\mu) &= (x_2+\rho_2^{-k(1+k)} \theta_2(\rho_2^{k(1+k)}x_2,\rho_2^{2k(1+k)},\mu))(1-\rho_2^{k(1+k)}\epsilon_2^k \phi_+(\rho_2^{1+k}\epsilon_2))\\
&+\epsilon_2^k \phi_+(\rho_2^{1+k}\epsilon_2).
\end{align*}
Notice by Theorem \ref{theorem:prop_normal_form} that $\rho_2^{-k(1+k)} \theta_2(\rho_2^{k(1+k)}x_2,\rho_2^{2k(1+k)},\mu)$ in these expressions 
has a smooth extension to $\rho_2=0$. In particular, $a_2(x_2,0,\epsilon_2,0) = 0$ \SJ{and} $b_2(x_2,0,\epsilon_2,0) =x_2+\epsilon_2^k \phi_+(0)=x_2+\epsilon_2^k\beta$, recall \eqref{eq:beta}.
Finally, we insert \eqref{eq:checkeps1f} into \eqref{eq:bary1epseqs}. This produces the final result upon dividing the right hand side by the common factor $\nu_1^{k(1+k)}$. 

\

The conservation of the quantities \eqref{eq:scaling1} and \eqref{eq:scaling2} follows from \eqref{eq:scaling} by setting $\check \epsilon=1$ and $\check \rho=1$ respectively.

\section{Proof of Lemma \ref{lem:coord_trans}}
\label{app:coord_trans}

Consider \eqref{eq:checkrho1eqs} within $\epsilon_2=0$. Then $\Psi_2$, defined as the $\epsilon_2=0$ restriction of \eqref{eq:checkrho1}:
\begin{align*}
\Psi_2:\,(x_2,\nu_2,\mu_2)\mapsto \begin{cases}
x = \nu_2^{2k(1+k)} x_2,\\
y =\nu_2^{2k(1+k)},\\
\mu =\nu_2^{2k(1+k)},
\end{cases}
\end{align*}
gives a smooth topological equivalence between \eqref{eq:checkrho1eqs}$_{\epsilon_2=0}$ and the $\mu$-extended system $\{(x',y')=X^+(x,y,\mu),\mu'=0\}$ on $\{y>0\}$. The latter system, since $X^+$ just has a hyperbolic and unstable node for $\mu$ small enough, is itself smoothly conjugated, see e.g.~\cite{Kuznetsov2013}, to the linearization:
\begin{equation}\label{eq:Xpluslin1}
\begin{aligned}
x' &= \mu + \tau x-\delta y,\\
y' &= x,\\
\mu'&=0,
\end{aligned}
\end{equation}
near  $(x,y,\mu)=(0,0,0)$. 
It is then obvious that $\Psi_2^{-1}$ gives a smooth topological equivalence between \eqref{eq:Xpluslin1} and the system \eqref{eq:checkrho1eqs}$_{\epsilon_2=0}$ with $\theta_i\equiv 0$ on $\nu_2>0$.  Putting this together, we have a smooth diffeomorphism
\begin{align*}
(\tilde x_2,\tilde \nu_2,\tilde \mu_2)\mapsto \begin{cases}
x_2 = \tilde x_2+\mathcal O_1(2),\\
\nu_2 = \tilde \nu_2 ( 1+\mathcal O_2(1)),\\
\mu_2 = \tilde \mu_2(1+\mathcal O_3(1)),
\end{cases}
\end{align*}
bringing \eqref{eq:checkrho1eqs}$_{\epsilon_2=0}$ on $\{\nu_2>0\}$ into the same form with $\theta_i\equiv 0$, upon a regular reparameterizations of time. It is straightforward to show that this mapping extends smoothly to $\{\tilde \nu_2=0\}\Leftrightarrow \{\nu_2=0\}$. We therefore define $\tilde \epsilon_2$ by the condition 
\begin{align*}
\epsilon &= \tilde \nu_2^{2(1+k)^2} \tilde \epsilon_2 = \nu_2^{2(1+k)^2} \epsilon_2,
\end{align*}
i.e. 
\begin{align*}
\tilde \epsilon_2 = \epsilon_2 ( 1+\mathcal O_2(1))^{-2(1+k)^2}. 
\end{align*}
Applying the diffeomorphism \SJ{$(\tilde x_2,\tilde \epsilon_2,\tilde \nu_2,\tilde \mu_2)\mapsto (x_2,\epsilon_2,\nu_2,\mu_2)$}, defined in this way, then gives
\begin{equation}\label{eq:normalized1}
\begin{aligned}
x_2' &= -x_2-\frac{\mu_2+\tau x_2 -\delta}{x_2}+\epsilon_2^k h_2(x_2,\epsilon_2,\nu_2,\mu_2) , \\
\epsilon_2' &= -\frac{k+1}{k}  \epsilon_2, \\
\nu_2' &= \frac{\nu_2}{2k(1+k)},\\
\mu_2' &=-\mu_2,
\end{aligned}
\end{equation}
upon dropping the tildes, 
for some smooth function $h_2$, upon dividing the right hand side by a positive factor near $x_{2,w}$ (notice in particular that we obtain \eqref{eq:normalized1} for $\epsilon_2=0$ by dividing the right hand side of \eqref{eq:checkrho1eqs} with $\theta_i\equiv 0$ by $g_2$, which is positive near $(x_{2,w},0,0,0)$).



Next, consider \eqref{eq:normalized1} within $\nu_2=\mu_2=0$:
\begin{align*}
x_2' &= -x_2-\frac{\mu_2+\tau x_2-\delta}{x_2}+\mathcal O(\epsilon_2^k),\\
\epsilon_2' &=-\frac{1+k}{k}\epsilon_2.
\end{align*}
For this sub-system, $(x_2,\epsilon_2)=(x_{2,w},0)$ is a hyperbolic saddle, the linearization having $\lambda,-(1+k)/k$ as eigenvalues, with $\lambda$ given as in Theorem \ref{thm:bn3}. 
Consequently, by Belitskii's theorem \SJ{\cite{Belitskii1973}, see also \cite{Homburg2010},} there exists a $C^1$-linearization of the form
\begin{align}
(\tilde x_2,\epsilon_2)\mapsto x_2=h(\tilde x_2,\epsilon_2),\label{eq:x2epslin}
\end{align}
with $h_2(0,0)=x_{2,w}$ such that 
\begin{align*}
x_2' &=\lambda x_2,\\
\epsilon_2' &=-\frac{1+k}{k}\epsilon_2,
\end{align*}
upon dropping the tildes. Lifting \eqref{eq:x2epslin} to the full space, we finally obtain
\eqref{eq:normalized2}. The order of the remainder easily follows from the expressions for $f_2$ and $g_2$.

\end{document}